\documentclass[10pt]{iopart}
\pdfoutput=1
\usepackage{geometry}
\usepackage[title]{appendix}
\usepackage{iopams}
\usepackage{epsfig}
\usepackage{amssymb}
\usepackage{amsthm,amssymb}
\usepackage{mathrsfs}
\usepackage{booktabs}
\usepackage{bm}
\usepackage{array}
\newtheorem{theorem}{Theorem}[section]
\newtheorem{proposition}{Proposition}[section]
\usepackage{algorithm}
\usepackage{algorithmic}
\usepackage{amstext}

\newtheorem{lemma}[theorem]{Lemma}

\newtheorem{remark}[theorem]{Remark}

\newtheorem{assumption}{Assumption}[section]

\usepackage{caption}
\usepackage{graphicx}
\usepackage{epstopdf}
\usepackage{subfigure}
\usepackage{multirow}
\usepackage{rotating}
\graphicspath{{fig/}}
\usepackage[space, compress, sort]{cite}
\usepackage{lineno,hyperref}
\usepackage{cleveref}
\usepackage{color,soul}

\makeatletter 
\@addtoreset{equation}{section}
\makeatother  
\renewcommand{\theequation}{\arabic{section}.\arabic{equation}}


\begin{document}

\title[]{Stochastic gradient descent method with convex penalty for ill-posed problems in Banach spaces}

\author{ Ruixue Gu$^1$, Zhenwu Fu$^2$$^{,\footnotemark[1]}$, Bo Han$^2$ and Hongsun Fu$^1$}
\address{ $^1$ School of Science, Dalian Maritime University, Dalian 116026, PR China\\[1mm]
$^2$ Department of Mathematics, Harbin Institute of Technology, Harbin, Heilongjiang Province, 150001, PR China}
\ead{ruixue\_gu@dlmu.edu.cn, fuzw@hit.edu.cn, bohan@hit.edu.cn and fuhongsun@dlmu.edu.cn}
\footnotetext[1]{Corresponding author}
\begin{abstract}
In this work, we investigate a stochastic gradient descent method for solving inverse problems that can be written as systems of linear or nonlinear ill-posed equations in Banach spaces. The method uses only a randomly selected equation at each iteration and employs the convex function as the penalty term, and thus it is scalable to the problem size and has the ability to detect special features of solutions such as nonnegativity and piecewise constancy. To suppress the oscillation in iterates and reduce the semi-convergence of such methods, by incorporating the spirit of discrepancy principle, an adaptive strategy for choosing the step size is suggested. Under certain conditions, we establish the regularization results of the method under an {\it a priori} stopping rule. Several numerical simulations on computed tomography and schlieren imaging are provided to demonstrate the effectiveness of the method. Finally, we study an {\it a posteriori} stopping rule for SGD-$\theta$ method and show the finite iterations termination property.
\end{abstract}
%
\vspace{1pc}
\noindent{ Keywords}: stochastic gradient descent method, linear and nonlinear inverse problems, system of ill-posed equations, convex penalty, convergence analysis, tomography
%
%
%

\setulcolor{red}
\section{Introduction}

In this paper, we consider ill-posed inverse problems governed by the system
\begin{equation}\label{nonlinear equation}
F_i(x)=y_i, ~~~i=1,2,\ldots,N
\end{equation}
consisting of $N$ equations, where $F_i:\mathscr{D}(F_i)\subset \mathcal{X} \rightarrow {\mathcal{Y}_i}$ are linear or nonlinear operators between Banach spaces $\mathcal{X}$ and ${\mathcal{Y}_i}$ with domain $\mathscr{D}(F_i)$. The systems of the form (\ref{nonlinear equation}) arise naturally in many practical applications, such as various tomography techniques and geophysics \cite{Engl1996Regularization,Natterer2001,Hanafy1991}. Throughout this paper, we assume that (\ref{nonlinear equation}) has a solution, which might not be unique. To incorporate the {\it priori} information of the sought solution, one may choose a proper, lower semi-continuous, convex function $\theta :\mathcal{X} \to \left( { - \infty ,\infty } \right]$ as the penalty term and select a solution $x^\dag$ of (\ref{nonlinear equation}) such that
\begin{equation}\label{minimum solution}
{D_{{\xi _0}}}\theta \left( {{x^\dag },{x_0}} \right): = \mathop {\min }\limits_{x \in \mathscr{D}\left( \theta  \right) \cap \mathcal{D}} \left\{ {{D_{{\xi _0}}}\theta \left( {x,{x_0}} \right):{F}_i\left( x \right) = {y_i}}, i=1,2,\ldots,N \right\},
\end{equation}
where $x_0 \in \mathscr{D}\left( {\partial \theta } \right)$ is an initial guess and $\xi _0  \in \partial \theta\left( {x_0} \right)$ is in the subdifferential of $\theta$ at $x_0$ (refer to (\ref{2.1}) for definition). Here $\mathcal{D}: = \bigcap\nolimits_{i = 1}^N {\mathscr{D}\left( {{F_i}} \right)}  \ne \emptyset $ and ${D_{{\xi _0}}}\theta\left( {x,{x_0}} \right)$ denotes the Bregman distance \cite{1967-Bregman-p200-217} induced by $\theta$ at $x_0$ in the direction $\xi_0$ (see (\ref{bregman distance}) for the definition). Because of unavoidable measurement errors, the exact data $y: = \left( {{y_1},{y_2},\cdots,{y_N}} \right)$ are generally not known; we can only access noisy data ${y^\delta }: = \left( {y_1^\delta ,y_2^\delta ,\cdots,y_N^\delta } \right)$ satisfying
 \[\left\| {y_i^\delta  - {y_i}} \right\| \le {\delta _i},~i=1,2,\ldots,N\]
 with the noise level $\delta_i>0$. Such inverse problems are inherently ill-posed in the sense that small perturbations of the data may result in large perturbations of the solution. Thus, how to use noisy data $y^\delta$ to reconstruct the solution $x^\dag$ becomes a central topic, see \cite{Engl1996Regularization,BARBARAFTIN} and the references therein.

By reformulating (\ref{nonlinear equation}) as a single equation
\[F\left( x \right) = y~~\text{with}~~F: = \left( {{F_1}, \cdots ,{F_N}} \right): \mathcal{D} = \bigcap\nolimits_{i = 1}^N {\mathscr{D}\left( {{F_i}} \right)}  \to \mathcal{Y} = {\mathcal{Y} _1} \times  \cdots  \times {\mathcal{Y}_N}\]
equivalently, a variety of regularization methods have been developed for solving $F\left( x \right) = y$ directly, including variational regularization methods \cite{Engl1996Regularization,Hofmann2007,Engl1989} and iterative regularization methods \cite{Hanke1995A,Bo2011Iterative,Zhong2019,KNS2008}. Among existing methods, the most prominent regularization method is the Landweber iteration \cite{Rommel_2020,Jin_2022}, with the form of
\begin{equation}\label{land}
 \begin{array}{l}
\xi _{n + 1}^\delta  = \xi _n^\delta  - t_n^\delta {F'}{\left( {x_n^\delta } \right)^*}{J_r^\mathcal{Y}}\left( {F\left( {x_n^\delta } \right) - {y^\delta }} \right),\\[1mm]
x_{n + 1}^\delta  = \arg \mathop {\min }\limits_{x \in \mathcal{X}} \left\{ {\theta \left( x \right) - \left\langle {\xi _{n + 1}^\delta ,x} \right\rangle } \right\},
\end{array}
\end{equation}
where $t_n^\delta$ is the step size and $J_r^\mathcal{Y}$ denotes the duality mapping of $\mathcal{Y}$ corresponding to the gauge function $t \to {t^{r - 1}}$ with  $1<r<\infty$. 
The iteration (\ref{land}) was proven to be a convergent regularization method in \cite{Jin_2022} when the stopping index $n_{\delta}$ of iteration is determined by an {\it a priori} stopping rule or the discrepancy principle
\begin{equation*}
\|F(x_{n_{\delta}}^{\delta})-y^{\delta}\|\leq\tau\delta<\|F(x_n^\delta)-y^{\delta}\|,~~0\le n<n_\delta
\end{equation*}
with $\tau$ a constant satisfying $\tau >1$ and
$$\delta : = \sqrt[r]{{\delta _1^r +  \cdots  + \delta _N^r}}$$
 denoting the total noise level of the noisy data. We note that in each iteration, the minimization problem for defining $x_{n + 1}^\delta$ in (\ref{land}) can be easily solved for many interesting choices of $\theta$, see \cref{3numbericalex} for elaboration; however, the update of $\xi _{n + 1}^\delta$ requires the solvers of the forward and adjoint problems for all $N$ equations in (\ref{nonlinear equation}) and can be numerically quite expensive in case $N$ is huge, rendering the method inefficient in practical applications.

To overcome this drawback, an alternative technique is to use Kaczmarz-type methods which cyclically consider each equation in (\ref{nonlinear equation}) and require only calculating one equation per iteration instead of all $N$ equations in (\ref{nonlinear equation}), leading to a significant reduction of computational cost. One may refer to \cite{Kowar_2002,Haltmeier_2007,Haltmeier_20071,Leitao_2016,LiHal_2018} for some convergence results on Kaczmarz-type methods in Hilbert spaces. By combining the Landweber iteration in \cite{Rommel_2020,Jin_2022} with Kaczmarz strategy, a general version of Landweber-Kaczmarz method was developed in \cite{JinQ_2016,JinWang2013} to solve (\ref{nonlinear equation}) in Banach spaces, given by
\begin{equation}\label{lk}
\begin{array}{*{20}{l}}
{\xi _{n + 1}^\delta  = \xi _n^\delta  - t_n^\delta {{F'_{\left[ n \right]}}}{{\left( {x_n^\delta } \right)}^*}{J_r^{\mathcal{Y}_{\left[ n \right]}}}\left( {{F_{\left[ n \right]}}\left( {x_n^\delta } \right) - y_{\left[ n \right]}^\delta } \right),}\\[1mm]
{x_{n + 1}^\delta  = \arg \mathop {\min }\limits_{x \in {\cal X}} \left\{ {\theta \left( x \right) - \left\langle {\xi _{n + 1}^\delta ,x} \right\rangle } \right\},}
\end{array}
\end{equation}
where $t_n^\delta$ denotes the step size and $\left[ n \right]:= (n~\text{mod}~N)+1$. For each $i=1, \cdots ,N$, $J_r^{\mathcal{Y}_i}:{\mathcal{Y}_i} \to \mathcal{Y}_i^*$ denotes the duality mapping of $\mathcal{Y}_i$ with $1<r<\infty$. Under certain conditions, the convergence results and regularization property of the method have been derived in \cite{JinWang2013} in which the numerical simulations illustrate its nice feature. The efficiency of this method, however, depends essentially on the order and the number of the equations in (\ref{nonlinear equation}). Therefore, it is desirable to develop more efficient algorithms to solve the system (\ref{nonlinear equation}).

Stochastic gradient descent (SGD) method was originally introduced by Robbins and Monro in \cite{Robbins_1951} to solve large-scale least square type problems. Instead of utilizing the equation in a cyclic manner like Kaczmarz-type methods, SGD method uses only a randomly selected equation in each iteration and hence scales very well with the problem size. This highly desirable property has received tremendous attention in machine learning \cite{Bottou_2010,Bottou_2018}, and has stimulated considerable subsequent work, see \cite{Johnson-Zhang-2013,Reddi-Hefny-2016,Hertrich-Steidl-2022,Gitman-Lang-2019} and references therein. And extensive numerical examples in \cite{Chen_2018,Jia_2010} indicate that SGD method holds significant potential for solving large-scale systems of ill-posed equations. Despite its computational appeal, the study on the theoretical property of SGD method for inverse problems is still scarce and has just started. When $\mathcal{X}$ and $\mathcal{Y}_i$ are both Hilbert spaces, some theoretical results of SGD method have been developed in \cite{Jiao_2017,JinB_2019,Lu_2022} for linear inverse problems and in \cite{JinB_2020} for nonlinear inverse problems. Recently, SGD method was extended in \cite{JinQ_2023,JinBker2023} to solve linear inverse problems in Banach spaces. In case ${F_i}: \mathcal{X} \to \mathcal{Y}_i$ are linear operators from Banach space $\mathcal{X}$ to Hilbert space $\mathcal{Y}_i$, the SGD method with convex penalty terms in \cite{JinQ_2023}, namely SGD-$\theta$ method, includes the following form
\begin{equation}\label{sgdlinear}
\begin{array}{l}
\xi _{n + 1}^\delta  = \xi _n^\delta  - t_n^\delta {F_{{i_n}}^*}\left( {{F_{{i_n}}}x_n^\delta  - y_{{i_n}}^\delta } \right),\\[1mm]
x_{n + 1}^\delta  = \arg \mathop {\min }\limits_{x \in \mathcal{X}} \left\{ {\theta \left( x \right) - \left\langle {\xi _{n + 1}^\delta ,x} \right\rangle } \right\},
\end{array}
\end{equation}
where $i_n \in \left\{ {1,2, \cdots ,N} \right\}$ is sampled uniformly at random. The convergence analysis and convergence rate of SGD-$\theta$ method have been studied in \cite{JinQ_2023} under {\it a priori} choice of the stopping index $n$ and several choices of the step size $t_n^\delta$;  the numerical simulations  presented there clearly show its superior performance. We note that the work in \cite{JinQ_2023} requires $\mathcal{Y}_i$ to be a Hilbert space and thus the method performs well when the data is corrupted by  Gaussian noise or the  uniformly distributed noise; in many applications, however, the noisy data may be contaminated by the impulsive noise, in which case one may need to choose $\mathcal{Y}_i$ to be a Banach space to remove the effect of the noise \cite{AlanBovik2006Handbook,2012-Clason-p505-536}. Moreover, the work in \cite{JinQ_2023} only concerns linear ill-posed systems. Considering the appearance of nonlinear systems of the form (\ref{nonlinear equation}) in a wide range of applications, it is of much importance to extend the work in \cite{JinQ_2023} to cover systems of nonlinear ill-posed equations. To tackle these issues, in this work we formulate an extension of SGD-$\theta$ method for solving linear as well as nonlinear systems of the form  (\ref{nonlinear equation}) with both $\mathcal{X}$ and $\mathcal{Y}_i$ being Banach spaces. This motivates our algorithm
\begin{equation}\label{sgdintro}
\begin{array}{l}
\xi _{n + 1}^\delta  = \xi _n^\delta  - t_n^\delta {F_{{i_n}}'}\left( {x_n^\delta } \right)^*J_r^{\mathcal{Y}_{i_n}}\left( {{F_{{i_n}}}\left( {x_n^\delta } \right) - y_{{i_n}}^\delta } \right),\\[1mm]
x_{n + 1}^\delta  = \arg \mathop {\min }\limits_{x \in \mathcal{X}} \left\{ {\theta \left( x \right) - \left\langle {\xi _{n + 1}^\delta ,x} \right\rangle } \right\},
\end{array}
\end{equation}
where the random index $i_n$ is selected uniformly from the set $\left\{ {1,2, \cdots ,N} \right\}$ and $t_n^\delta$ is the suitably chosen step size.

The main contributions of this work are as follows.
First, in contrast to polynomially decreasing or constant step size customarily employed for SGD \cite{JinB_2019,Lu_2022,JinB_2020}, we incorporate the spirit of discrepancy principle into  the step size $t_n^\delta$ with the hope of suppressing the oscillation in iterates and reducing the semi-convergence of such methods. Under certain conditions, we prove the convergence results of SGD-$\theta$ method (\ref{sgdintro}) in Banach spaces under an {\it a priori} stopping rule.
Due to the possible nonlinearity of  $F_i$ and non-Hilbertian structure of $\mathcal{X}$ and $\mathcal{Y}_i$, compared with \cite{JinQ_2023}, our convergence analysis is much more challenging. The analysis is derived by exploiting nonlinear regularization theory in Banach spaces and stochastic analysis; see \cref{3themethod}. Second, we provide some numerical simulations on computed tomography and schlieren imaging in \cref{3numbericalex}. The numerical results validate the effectiveness of the step size, and illustrate that our method (\ref{sgdintro}) could effectively cope with the data containing various types of noise and works for linear as well as nonlinear systems of ill-posed equations. Third, we consider an {\it a posteriori} stopping rule for SGD-$\theta$ method and establish its finite iterations termination property.

This paper is organized as follows. In \cref{preliminaries}, we review some preliminary results concerning convex analysis and Banach spaces that are crucial in the convergence analysis of the method. In \cref{3themethod}, under certain conditions, we first establish the convergence in expectation of the method with exact data. Then, for the case of noisy data, we show the stability property and further prove the regularization results of the method under {\it a priori} stopping rule. In \cref{3numbericalex}, we provide some numerical simulations on computed tomography and schlieren imaging to illustrate the effectiveness of the method. In \cref{conclusion}, we present a summary and possible directions for future research. Finally, in \ref{appendix_sgd}, we design an {\it a posteriori} stopping rule for SGD-$\theta$ method and show its finite iterations termination property.

\section{Preliminaries}\label{preliminaries}
In this section, we introduce some notions and basic facts related to convex analysis and Banach spaces; for more information we refer the reader to \cite{Cioranescu1990Geometry,Schuster2012Regularization,Zalinescu2002Convex}.

 Let $\mathcal{X}$ and  $\mathcal{Y}$ be two Banach spaces with dual spaces ${\mathcal{X}^*}$ and ${\mathcal{Y}^*}$ respectively. The norms of $\mathcal{X}$ and  $\mathcal{Y}$ are denoted by $\left\|  \cdot  \right\|$ that should be clear from the context which one is referred to.  For $x \in \mathcal{X}\;\text{and}\;\ {x^*} \in {\mathcal{X}^*}$, we denote by $\left\langle {{x^*},x} \right\rangle  = {x^*}\left( x \right)$ the duality pairing. We use ``$\rightarrow$'' and ``$\rightharpoonup$'' to denote the strong convergence and weak convergence respectively. For a bounded linear operator $A:\mathcal{X} \to \mathcal{Y}$, let $A^*:\mathcal{Y}^* \to \mathcal{X}^*$ and $\mathcal{N}\left( A \right)$ denote its adjoint and its null space respectively.

For a convex function $f :\mathcal{X} \to \left( { - \infty ,\infty } \right]$, the effective domain of $f$ is defined by $\mathscr{D}\left( f \right): = \left\{ {x \in \mathcal{X}:f \left( x \right) < \infty } \right\}$. $f$ is called proper when $\mathscr{D}\left( f  \right) \ne \emptyset $.
  For ${x \in \mathcal{X}}$, we define
  \begin{equation}\label{2.1}
  \partial f \left( x \right): = \left\{ {{\xi} \in {\mathcal{X}^*}:f \left( \bar x \right) - f \left( x \right) - \left\langle {{\xi},\bar x - x} \right\rangle  \ge 0,~\forall~\bar x \in \mathcal{X}} \right\},
  \end{equation}
which is called the subdifferential $\partial f \left( x \right)$ of $f $. Each ${\xi} \in \partial f \left( x \right)$ is called a subgradient of $f $ at ${x}$.
Let $\mathscr{D}\left( {\partial f } \right): = \left\{ {x \in \mathscr{D}\left( f  \right):~\partial f \left( x \right) \ne \emptyset } \right\}$.
Given $x \in \mathscr{D}\left( \partial f  \right)~\text{and}~\xi  \in \partial f \left( x \right)$, the quantity
 \begin{equation}\label{bregman distance}
 {D_{\xi} }f \left( {\bar x,x} \right): = f \left( \bar x \right) - f \left( x \right) - \left\langle {\xi,\bar x - x} \right\rangle,~~~\forall \bar x \in \mathcal{X}
 \end{equation}
 is called the Bregman distance induced by $f$ at $x$ in the direction $\xi$. Clearly, ${D_{\xi} }f \left( {\bar x,x} \right)\ge 0$ for $x,\bar x \in \mathcal{X}$.

 A proper function $f :\mathcal{X} \to \left( { - \infty ,\infty } \right]$ is called $p$-convex with $p\ge 2$
  if there exists a constant ${\sigma} > 0$ such that
\begin{equation}\label{pconvex}
f \left( {tx + \left( {1 - t} \right)\bar{x}} \right) + {\sigma}t\left( {1 - t} \right){\left\| {x - \bar{x}} \right\|^p} \le tf \left( x \right) + \left( {1 - t} \right)f \left( \bar{x} \right)
\end{equation}
for all $x,\bar{x} \in \mathcal{X}$ and $0 \le t \le 1$.

Given a proper, lower semi-continuous and convex function ${f }$,  let  ${f ^*}:{\mathcal{X}^*} \to \left( { - \infty ,\infty } \right]$ denote the Legendre-Fenchel conjugate of ${f }$, i.e.,
  \begin{equation}
  \;{f ^*}\left( {\xi}  \right): = \mathop {\sup }\limits_{x \in \mathcal{X}} \left\{ {\left\langle {{\xi} ,x} \right\rangle  - f \left( x \right)} \right\}.
  \end{equation}
It can be shown that ${f ^*}$ is also proper, lower semi-continuous and convex. According to Proposition 4.4.1 in \cite{Schirotzek2007Nonsmooth}, we have,
for any $x \in \mathcal{X},{\xi} \in {\mathcal{X}^*}$,
\[\left\langle {{\xi} ,x} \right\rangle  \le f \left( x \right) + {f ^*}\left( {\xi}  \right)\]
and
\begin{equation}\label{2.2}
{\xi}  \in \partial f \left( x \right) \Leftrightarrow x \in \partial {f ^*}\left( {\xi}  \right) \Leftrightarrow \left\langle {{\xi} ,x} \right\rangle = f \left( x \right) + {f ^*}\left( {\xi}  \right).
\end{equation}
By the subdifferential calculus, there also holds
\begin{equation}\label{subcaclulus}
x \in \partial {f^*}\left( \xi  \right) \Leftrightarrow x = \arg \mathop {\min }\limits_{z \in \mathcal{X}} \left\{ {f\left( z \right) - \left\langle {\xi ,z} \right\rangle } \right\}.
\end{equation}
From (\ref{2.2}), it follows that
\begin{equation}\label{bregmanconju}
{D_\xi }f\left( {\bar x,x} \right) = f\left( {\bar x} \right)+{f ^*}\left( \xi \right) - \left\langle {\xi ,\bar x} \right\rangle
\end{equation}
for all $\bar x \in \mathcal{X}$, $x \in \mathscr{D}\left( \partial f  \right)~\text{and}~\xi  \in \partial f \left( x \right)$.

The further results of $p$-convex function $f$ and its Legendre-Fenchel conjugate $f^*$ are stated below; one refer to \cite[Corollary 3.5.11]{Zalinescu2002Convex} for more details.
\begin{lemma}\label{lemma2.2}
Let $f :\mathcal{X} \to \left( { - \infty ,\infty } \right]$ be a proper, lower semi-continuous and  $p$-convex function with $p\ge 2$ in the sense that (\ref{pconvex}) is satisfied for some constant $\sigma>0$. Then,
\begin{itemize}
\item [(i)]  for $x \in \mathscr{D}\left( {\partial f } \right),{\xi} \in \partial f \left( x \right)$, there holds
\begin{equation}\label{theta p-convex}
{D_{{\xi}}}f \left( {\bar x,x} \right) \ge {\sigma}{\left\| {\bar{x} - x} \right\|^p}, ~\forall \bar x \in \mathcal{X}.
\end{equation}
 \item [(ii)] for any $x \in \mathscr{D}\left( {\partial f } \right),{\xi} \in \partial f \left( x \right)$ and ${\eta}\in {\mathcal{X}^*}$, we have
 \begin{equation}\label{theta * p*-smooth}
{f ^*}\left( {{\eta}} \right) - {f ^*}\left( {{\xi}} \right) - \left\langle {x,{\eta} - {\xi}} \right\rangle \le \frac{1}{{{p^*{\left( 2 \sigma \right)^{ {p^*}-1}}}}}{\left\| {{\eta} - {\xi}} \right\|^{{p^*}}},
 \end{equation}
where $p^*$ is the number conjugate to $p$, i.e., ${p^{ - 1}} + {\left( {{p^*}} \right)^{ - 1}} = 1$.
\item [(iii)] $\mathscr{D}\left( {{f ^*}} \right) = {\mathcal{X}^*}$. ${{f ^*}}$ is Fr\'{e}chet {\color{black}{differentiable}} and its gradient $\nabla {f ^*}:{\mathcal{X}^*} \to \mathcal{X}$ satisfies
 \begin{equation}\label{Dtheta *}
\left\| {\nabla {f^*}\left( \eta  \right) - \nabla {f^*}\left( \xi  \right)} \right\| \le {\left( {\frac{{\left\| {\eta  - \xi } \right\|}}{{2\sigma }}} \right)^{\frac{1}{{p - 1}}}}
 \end{equation}
for all ${\eta},~ {\xi} \in {\mathcal{X}^*}$.
\end{itemize}
\end{lemma}

On a Banach space $\mathcal{Y}$, we consider for $1 < r < \infty$ the convex function $y \to {{{{\left\| y \right\|}^r}} \mathord{\left/
 {\vphantom {{{{\left\| y \right\|}^r}} r}} \right.
 \kern-\nulldelimiterspace} r}$ with its subdifferential at $y$ given by
\[J_r^\mathcal{Y}\left( y \right): = \left\{ {\xi  \in {\mathcal{Y}^*}:\left\| \xi  \right\| = {{\left\| y \right\|}^{r - 1}}~\text{and}~\left\langle {\xi ,y} \right\rangle  = {{\left\| y \right\|}^r}} \right\},\]
which gives the duality mapping $J_r^\mathcal{Y}:\mathcal{Y} \to {2^{{\mathcal{Y}^*}}}$  with gauge function $t \mapsto {t^{r - 1}}$.
When $\mathcal{Y}$ is uniformly smooth in the sense that the modulus of smoothness as
\[{\rho _\mathcal{Y}}\left( \varepsilon  \right): = \frac{1}{2}\sup \left\{ {\left\| {\bar y + y} \right\| + \left\| {\bar y - y} \right\| - 2:\left\| \bar y \right\| = 1,\left\| y \right\| \le \varepsilon } \right\}\]
satisfies {\color{black}{$\mathop {\lim }\nolimits_{\varepsilon  \to 0} {{{\rho _\mathcal{Y}}\left( \varepsilon  \right)} \mathord{\left/
 {\vphantom {{{\rho _\mathcal{Y}}\left( \varepsilon  \right)} \varepsilon }} \right.
 \kern-\nulldelimiterspace} \varepsilon } = 0$}}, the duality mapping $J_r^\mathcal{Y}$ with $1<r<\infty$ is single valued and uniformly continuous on bounded sets.

\section{Convergence}\label{3themethod}
In this section, we will establish the convergence analysis of SGD-$\theta$ method in Banach spaces under an {\it a priori} stopping rule. Before proceeding further, we impose some assumptions on $\theta$ and the operators $F_i$, where ${B_\rho }\left( {{x_0}} \right): = \left\{ {x \in \mathcal{X}:\left\| {x - {x_0}} \right\| \le \rho } \right\}$ and $\mathcal{D}: = \bigcap\nolimits_{i = 1}^N {\mathscr{D}\left( {{F_i}} \right)}  \ne \emptyset $.
\begin{assumption}\label{5asumption theta}
$\theta :\mathcal{X} \to \left( { - \infty ,\infty } \right]$ is a  proper, lower semi-continuous, $p$-convex function with $p\ge 2$ satisfying (\ref{theta p-convex}) for some constant $\sigma>0$.
\end{assumption}

\begin{assumption}
\begin{itemize}
\item [(a)] $\mathcal{X}$ is a  reflexive Banach space and each $\mathcal{Y}_i$, $i=1,2,\ldots,N$, is a uniformly smooth Banach space.
\item [(b)] There exists $\rho>0$ such that ${B_{2\rho }}\left( {{x_0}} \right) \subset \mathcal{D}$ and (\ref{nonlinear equation}) has a solution $x_*$ in $ \mathscr{D}\left( \theta  \right)$ satisfying
   ${D_{{\xi _0}}}\theta \left( { x_*,{x_0}} \right) \le {\sigma}{\rho ^p}$.
   \item [(c)]  Each ${{F_i}}$ is weakly closed on $\mathcal{D}$, i.e., for any sequence $\left\{ {{x_n}} \right\} \subset \mathcal{D}$ satisfying $ {x_n}\rightharpoonup x\in \mathcal{X}$ and {\color{black}{${F_i}\left( x_n \right) \rightharpoonup v_i\in \mathcal{Y}_i$}}, there hold $x\in \mathcal{D}$ and ${F_i}\left( x \right) =v_i$.
\item [(d)]  $F_i$ is Fr\'{e}chet differentiable on $\mathcal{D}$, and $x \to F_i'\left( x \right)$ is continuous on ${B_{2\rho }}\left( {{x_0}} \right)$; moreover,
     \[\left\| {F_i'\left( x \right)} \right\| \le {B_0},~~~~\forall x\in{B_{2\rho }}\left( {{x_0}} \right)\]
       for some constant $B_0>0$ and
    there is $0 \le \eta  < 1$ such that
\begin{equation}\label{TCC}
\left\| {{F_i}\left( x \right) - {F_i}\left( {\bar x} \right) - F_i'\left( {\bar x} \right)\left( {x - \bar x } \right)} \right\| \le \eta \left\| {{F_i}\left( x \right) - {F_i}\left( {\bar x} \right)} \right\|
\end{equation}
 for all $x,\bar x  \in {B_{2\rho }}\left( {{x_0}} \right)$ and $i=1,2,\ldots,N$.
\end{itemize}
\label{5asumption operator}
\end{assumption}

The conditions in Assumption \ref{5asumption operator} are commonly used in the convergence analysis of iterative regularization methods for ill-posed inverse problems \cite{Hanke1995A, KNS2008}. The uniform smoothness of $\mathcal{Y}_i$ in condition (a) ensures that the duality mapping $J_r^{\mathcal{Y}_{i}}$ is single-valued and continuous for each $i=1,2,\ldots,N$ and $1<r<\infty$. The inequality (\ref{TCC}) is referred to as tangential cone condition and has been verified for a class of nonlinear inverse problems \cite{Hanke1995A,MR3001307}. Note that when the operators $F_i$ are linear, (\ref{TCC}) holds with $\eta =0$.
According to Assumption \ref{5asumption operator} (b) and (\ref{minimum solution}), there holds
\begin{equation}\label{6xdag}
{D_{{\xi _0}}}\theta \left( {{x^\dag },{x_0}} \right)\le {D_{{\xi _0}}}\theta \left( {{x_*},{x_0}} \right) \le \sigma {\rho ^p},
\end{equation}
which together with Assumption \ref{5asumption theta} yields $\left\| {{x^\dag } - {x_0}} \right\| \le \rho $, i.e., $x^\dag \in {B_\rho }\left( {{x_0}} \right)\cap \mathscr{D}\left( \theta  \right)$.
Under Assumptions \ref{5asumption theta} and \ref{5asumption operator}, one can use  \cite[Lemma 3.2]{JinWang2013} to show that $x^\dag$ is uniquely defined.

 We now formulate the stochastic gradient descent (SGD-$\theta$) method with convex penalty terms for solving  inverse problems (\ref{nonlinear equation}) in Banach spaces. Let $x_0^\delta:=x_0 \in \mathscr{D}\left( {\partial \theta } \right)$ and $\xi _0^\delta:= \xi _0  \in \partial \theta\left( {x_0} \right)$ be the initial guesses, then the SGD-$\theta$ method  updates the iterates $\xi_{n}^\delta$ and $x_{n}^\delta$ by
\begin{equation}\label{sgd}
\begin{array}{l}
\xi _{n + 1}^\delta  = \xi _n^\delta  - t_n^\delta {F_{{i_n}}'}\left( {x_n^\delta } \right)^*J_r^{\mathcal{Y}_{i_n}}\left( {{F_{{i_n}}}\left( {x_n^\delta } \right) - y_{{i_n}}^\delta } \right),\\[1mm]
x_{n + 1}^\delta  = \arg \mathop {\min }\limits_{x \in \mathcal{X}} \left\{ {\theta \left( x \right) - \left\langle {\xi _{n + 1}^\delta ,x} \right\rangle } \right\},
\end{array}
\end{equation}
where ${i_n} \in \left\{ {1,2, \cdots ,N} \right\}$ is selected uniformly at random and $J_r^{\mathcal{Y}_{i_n}}$ denotes the duality mapping of $\mathcal{Y}_{i_n}$ with $1<r<\infty$. The step size $t_{{n}}^\delta$ is determined by
\begin{equation}\label{3choicestep}
t_n^\delta  = \left\{ {\begin{array}{*{20}{l}}
{\tilde t_n^\delta {{\left\| {{F_{{i_n}}}\left( {x_n^\delta } \right) - y_{{i_n}}^\delta } \right\|}^{p - r}},~~~~\left\| {{F_{{i_n}}}\left( {x_n^\delta } \right) - y_{{i_n}}^\delta } \right\| > \tau {\delta _{{i_n}}},}\\[1mm]
{0,~~~~~~~~~~~~~~~~~~~~~~~~~~~~~~~~otherwise}
\end{array}} \right.
\end{equation}
with preassigned $\tau>1$ and
\[\tilde t_n^\delta  = \min \left\{ {\frac{{{\mu _0}{{\left\| {{F_{{i_n}}}\left( {x_n^\delta } \right) - y_{{i_n}}^\delta } \right\|}^{p(r - 1)}}}}{{{{\left\| {{{F}_{{i_n}}'}{{\left( {x_n^\delta } \right)}^*}J_r^{{{\cal Y}_{{i_n}}}}\left( {{F_{{i_n}}}\left( {x_n^\delta } \right) - y_{{i_n}}^\delta } \right)} \right\|}^p}}},{\mu _1}} \right\}\]
for some positive constants $\mu_0$ and $\mu_1$.
\begin{remark}
\begin{itemize}
  \item[(i)] According to Assumptions \ref{5asumption theta} and \ref{5asumption operator}, the minimization problem for defining $x_{n+1}^\delta$ in (\ref{sgd}) has a unique minimizer, refer to \cite[Proposition 3.5.8]{Zalinescu2002Convex}. Indeed, by virtue of (\ref{subcaclulus}) and the differentiability of $\theta^*$ guaranteed by Lemma \ref{lemma2.2} (iii), we have
\begin{equation}\label{xix}
{x_{ n+ 1}^\delta} = \nabla {\theta ^*}\left( {{\xi _{ n + 1}^\delta}} \right).
\end{equation}
  \item[(ii)] In each iteration, the method (\ref{sgd}) chooses an index $i_n$ randomly from $\left\{ {1,2, \cdots ,N} \right\}$ to form the partial term
\[{\Phi _{{i_n}}}\left( x \right) = \frac{1}{r}{\left\| {{F_{{i_n}}}\left( x \right) - y_{{i_n}}^\delta } \right\|^r}\]
   of the data fidelity objective $\Phi \left( x \right) = \frac{1}{Nr}\sum\nolimits_{i = 1}^N {{{\left\| {{F_i}\left( x \right) - y_i^\delta } \right\|}^r}}$ and then uses its gradient $\partial {\Phi _{{i_n}}}\left( x _n^{\delta}\right) = F_{{i_n}}'\left( {x_n^\delta } \right)^*J_r^{{{\cal{Y}}_{{i_n}}}}\left( {{F_{{i_n}}}\left( {x_n^\delta } \right) - y_{{i_n}}^\delta } \right)$ at $x_n^\delta$ as an estimator of the full gradient $\partial \Phi \left( {x_n^\delta } \right)$  for updating $\xi_{n+1}^\delta$ from $\xi_{n}^\delta$.
  Due to the existence of the estimated gradient noise, the iterates exhibit pronounced oscillations; and due to the ill-posedness of problem (\ref{nonlinear equation}), the method using noisy data demonstrates semi-convergence, i.e., the iterates tends to the solution of (\ref{nonlinear equation}) at the beginning, and, after a certain number of iterations, the iterates diverges. Therefore, it is challenging to produce a good approximate solution. The choice of the step size is crucial for the convergence of the method and should be selected carefully. In contrast to polynomially decreasing or constant step size customarily employed for SGD \cite{JinB_2019,Lu_2022,JinB_2020}, inspired by the recent study \cite{JinQ_2023}, our step size (\ref{3choicestep}) uses an adaptive rule and  incorporates the spirit of the discrepancy principle. This step size schedule could efficiently suppress the oscillations and reduce the effect of semi-convergence, as indicated by the numerical experiments in  \cref{3numbericalex}.

\end{itemize}
\end{remark}

Due to the stochastic choice of $i_n$, the iterates $\xi_n^\delta$ and $x_n^\delta$ are random. The step size $t_n^\delta$ given by (\ref{3choicestep}) depends on  $i_n$ and the realization of $\left( {{i_0}, \cdots ,{i_{n - 1}}} \right)$ through the random variable $x_n^\delta$. To establish the convergence results of the method (\ref{sgd}), we need to use tools from stochastic analysis.
Let ${\Lambda _N} := \left\{ {1,2, \cdots ,N} \right\}$, and let $\mathcal{S}_{N}$ be the $\sigma$-algebra generated by all subsets of ${\Lambda _N}$. At each iteration, the iterates $\xi_n^\delta$ and $x_n^\delta$ are updated by choosing an index $i_n$ from ${\Lambda _N} $ uniformly at random. Therefore, for each $n\ge 1$, one could think of $\xi_n^\delta$ and $x_n^\delta$ on the sample space
$\Lambda _N^n = {\Lambda _N} \times  \cdots  \times {\Lambda _N}$
with the $\sigma$-algebra $\mathcal{S}_N^{ \otimes n}$ and the uniformly distributed probability ${\mathbb{P}_n}$. According to the Kolmogorov extension theorem \cite{Bhat_2007}, there exists a unique probability $\mathbb{P}$ on the measurable space $\left( {\Omega ,\mathcal{F}} \right): = \left( {\Lambda _N^\infty ,\mathcal{S}_N^{ \otimes \infty }} \right)$ such that each ${\mathbb{P}_n}$ is consistent with $\mathbb{P}$.
Below $ \mathbb{E}\left[ \cdot\right]$ denotes the expectation on $\left( {\Omega ,\mathcal{F},\mathbb{P}} \right)$ and $\mathcal{F}_n$ denotes the natural filtration with  $\mathcal{F}_n:=\sigma\left( {{i_0}, \ldots ,{i_{n - 1}}} \right)$. By the law of total expectation, we have
\[\mathbb{E}\left[ \phi  \right] = \mathbb{E}\left[ {\mathbb{E}\left[ {\phi |{\mathcal{F}_n}} \right]} \right]\]
for any random variable $\phi$, where ${\mathbb{E}\left[ {\phi |{\mathcal{F}_n}} \right]}$ is the conditional expectation of $\phi$ with respect to $\mathcal{F}_n$; this identity will be used in the forthcoming analysis.
To measure the convergence of SGD-$\theta$ method, we employ the expectation of the Bregman distance $ \mathbb{E}\left[ {{D_{\xi _n^\delta }}\theta \left( {\hat x,x_n^\delta } \right)} \right]$, where $\hat x \in {B_{2\rho} }\left( {{x_0}} \right)\cap \mathscr{D}\left( \theta  \right)$ denotes a solution of (\ref{nonlinear equation}). The next lemma states the monotonicity of the mean iteration error $ \mathbb{E}\left[ {{D_{\xi _n^\delta }}\theta \left( {\hat x,x_n^\delta } \right)} \right]$.
\begin{lemma}\label{smonotoniciy}
Let Assumptions \ref{5asumption theta} and \ref{5asumption operator} hold.  Let $\mu_0>0$ and $\tau>1$ be chosen such that
\begin{equation}\label{c1}
{c_0} := 1 - \eta  - \frac{{1 + \eta }}{\tau } - \frac{{p - 1}}{p}{\left( {\frac{{{\mu _0}}}{{2\sigma }}} \right)^{\frac{1}{{p - 1}}}} > 0.
\end{equation}
Then, $x_n^\delta \in {B_{2\rho} }\left( {{x_0}} \right)$ for all $n\ge 0$. Moreover, for any solution $\hat x$ of (\ref{nonlinear equation}) in ${B_{2\rho} }\left( {{x_0}} \right)\cap \mathscr{D}\left( \theta  \right)$, there hold
\begin{equation}\label{monobreg}
{D_{\xi _{n + 1}^\delta }}\theta \left( {\hat x,x_{n + 1}^\delta } \right) \le {D_{\xi _n^\delta }}\theta \left( {\hat x,x_n^\delta } \right)
\end{equation}
and
\begin{equation}\label{exmonobreg}
\mathbb{E}\left[ {{D_{\xi _{n + 1}^\delta }}\theta \left( {\hat x,x_{n + 1}^\delta } \right)} \right] -\mathbb{E}\left[ {{D_{\xi _n^\delta }}\theta \left( {\hat x,x_n^\delta } \right)} \right]\le - \frac{{{c_0}}}{N}\sum\limits_{i = 1}^N { \mathbb{E}\left[t_n^\delta {{\left\| {{F_i}\left( {x_n^\delta } \right) - y_i^\delta } \right\|}^r}\right] }
\end{equation}
for $n\ge 0$.
\end{lemma}

\begin{proof}
We first show that if $x_n^\delta \in {B_{2\rho} }\left( {{x_0}} \right)$ for some $n\ge 0$, then
\begin{equation}\label{bregdes1}
{D_{\xi _{n + 1}^\delta }}\theta \left( {\hat x,x_{n + 1}^\delta } \right) - {D_{\xi _n^\delta }}\theta \left( {\hat x,x_n^\delta } \right) \le  - {c_0}t_n^\delta {\left\| {{F_{{i_n}}}\left( {x_n^\delta } \right) - y_{{i_n}}^\delta } \right\|^r},
\end{equation}
where $\hat x$ denotes any solution of (\ref{nonlinear equation}) in ${B_{2\rho} }\left( {{x_0}} \right)\cap \mathscr{D}\left( \theta  \right)$.
If $\left\| {{F_{{i_n}}}\left( {x_n^\delta } \right) - y_{{i_n}}^\delta } \right\| \le \tau {\delta _{{i_n}}}$, we can derive from (\ref{3choicestep}) that $t_n^\delta=0$ and thus, by the definition of $\xi_{n+1}^\delta$, we have $\xi_{n+1}^\delta=\xi_{n}^\delta$; by using the fact that $x_n^\delta=\nabla {\theta ^*}\left( {\xi _n^\delta } \right)$, we further obtain $x_{n+1}^\delta=x_{n}^\delta$ and subsequently (\ref{bregdes1}) follows. It remains to consider the case $\left\| {{F_{{i_n}}}\left( {x_n^\delta } \right) - y_{{i_n}}^\delta } \right\|> \tau {\delta _{{i_n}}}$. In this case, by using (\ref{bregmanconju}) and $x_n^\delta=\nabla {\theta ^*}\left( {\xi _n^\delta } \right)$, there holds
\[\begin{array}{l}
\quad {D_{\xi _{n + 1}^\delta }}\theta \left( {\hat x,x_{n + 1}^\delta } \right) - {D_{\xi _n^\delta }}\theta \left( {\hat x,x_n^\delta } \right)\\[1.5mm]
 = {\theta ^*}\left( {\xi _{n + 1}^\delta } \right) - {\theta ^*}\left( {\xi _n^\delta } \right) - \left\langle {\xi _{n + 1}^\delta  - \xi _n^\delta ,\hat x} \right\rangle \\[1mm]
 = {\theta ^*}\left( {\xi _{n + 1}^\delta } \right) - {\theta ^*}\left( {\xi _n^\delta } \right) - \left\langle {\xi _{n + 1}^\delta  - \xi _n^\delta ,\nabla {\theta ^*}\left( {\xi _n^\delta } \right)} \right\rangle  + \left\langle {\xi _{n + 1}^\delta  - \xi _n^\delta ,x_n^\delta  - \hat x} \right\rangle.
\end{array}\]
With the help of (\ref{theta * p*-smooth}) and the definition of $\xi_{n+1}^\delta$, we further have
\begin{equation}\label{bregmantui}
\hspace{-3mm}
\begin{array}{*{20}{l}}
\quad {{D_{\xi _{n + 1}^\delta }}\theta \left( {\hat x,x_{n + 1}^\delta } \right) - {D_{\xi _n^\delta }}\theta \left( {\hat x,x_n^\delta } \right)}\\[1mm]
{ \le \frac{1}{{{p^*}{{\left( {2{\sigma}} \right)}^{{p^*} - 1}}}}{{\left\| {\xi _{n + 1}^\delta  - \xi _n^\delta } \right\|}^{{p^*}}} + \left\langle {\xi _{n + 1}^\delta  - \xi _n^\delta ,x_n^\delta  - \hat x} \right\rangle }\\[1mm]
 \le \frac{1}{{{p^*}{{\left( {2{\sigma}} \right)}^{{p^*} - 1}}}}{{{\left\| {\xi _{n + 1}^\delta  - \xi _n^\delta } \right\|}^{{p^*}}}}+ t_n^\delta \left\langle {J_r^{{\mathcal{Y}_{{i_n}}}}\left( {{F_{{i_n}}}\left( {x_n^\delta } \right) - y_{{i_n}}^\delta } \right),{F_{{i_n}}'}\left( {x_n^\delta } \right)\left( {\hat x -x_n^\delta } \right)} \right\rangle .
\end{array}
\end{equation}
Note that $F_{{i_n}}'\left( {x_n^\delta } \right)\left( { \hat x -x_n^\delta} \right) = y_{{i_n}}^\delta -{F_{{i_n}}}\left( {x_n^\delta } \right) - y_{{i_n}}^\delta +y_{i_n}-y_{i_n} +{F_{{i_n}}}\left( {x_n^\delta } \right) + F_{{i_n}}'\left( {x_n^\delta } \right)\left( { \hat x -x_n^\delta} \right)$, utilizing $\left\| {y_{i_n}^\delta  - {y_{i_n}}} \right\| \le {\delta _{i_n}}$ and the property of the duality mapping $J_r^{{\mathcal{Y}_{{i_n}}}}$, we can deduce that
\begin{equation*}
\begin{array}{l}
\quad t_n^\delta \left\langle {J_r^{{\mathcal{Y}_{{i_n}}}}\left( {{F_{{i_n}}}\left( {x_n^\delta } \right) - y_{{i_n}}^\delta } \right),{F_{{i_n}}'}\left( {x_n^\delta } \right)\left( {\hat x - x_n^\delta } \right)} \right\rangle \\[1mm]
 \le t_n^\delta {\left\| {{F_{{i_n}}}\left( {x_n^\delta } \right) - y_{{i_n}}^\delta } \right\|^{r - 1}}\left( \delta_{i_n}+ \left\| {y_{{i_n}}- {F_{{i_n}}}\left( {x_n^\delta } \right) - {F_{{i_n}}'}\left( {x_n^\delta } \right)\left( {\hat x - x_n^\delta } \right)} \right\|\right)\\[1mm]
 ~~- t_n^\delta {\left\| {{F_{{i_n}}}\left( {x_n^\delta } \right) - y_{{i_n}}^\delta } \right\|^r}.
\end{array}
\end{equation*}
Recall that $x_n^\delta \in {B_{2\rho} }\left( {{x_0}} \right)$ and $\hat x$ is a solution of (\ref{nonlinear equation}) in ${B_{2\rho} }\left( {{x_0}} \right)\cap \mathscr{D}\left( \theta  \right)$, then $F_{i_n}\left( {\hat x } \right)=y_{i_n}$ and we may use (\ref{TCC}) in Assumption \ref{5asumption operator} (d) to derive that
\begin{equation*}
\begin{array}{l}
\quad t_n^\delta \left\langle {J_r^{{\mathcal{Y}_{{i_n}}}}\left( {{F_{{i_n}}}\left( {x_n^\delta } \right) - y_{{i_n}}^\delta } \right),{F_{{i_n}}'}\left( {x_n^\delta } \right)\left( {\hat x - x_n^\delta } \right)} \right\rangle \\[1mm]
 \le t_n^\delta {\left\| {{F_{{i_n}}}\left( {x_n^\delta } \right) - y_{{i_n}}^\delta } \right\|^{r - 1}}\left( \delta_{i_n}+ \eta \left\| { {F_{{i_n}}}\left( {x_n^\delta } \right)-y_{{i_n}}} \right\|\right)
 - t_n^\delta {\left\| {{F_{{i_n}}}\left( {x_n^\delta } \right) - y_{{i_n}}^\delta } \right\|^r}\\[1mm]
 \le   t_n^\delta {\left\| {{F_{{i_n}}}\left( {x_n^\delta } \right) - y_{{i_n}}^\delta } \right\|^{r - 1}}\left( {\left( {1 + \eta } \right){\delta _{{i_n}}} + \eta \left\| {{F_{{i_n}}}\left( {x_n^\delta } \right) - y_{{i_n}}^\delta } \right\|} \right)- t_n^\delta {\left\| {{F_{{i_n}}}\left( {x_n^\delta } \right) - y_{{i_n}}^\delta } \right\|^r}.
\end{array}
\end{equation*}
By virtue of the definition of $t_n^\delta$ in (\ref{3choicestep}), it follows that
\[
t_n^\delta \left\| {{F_{{i_n}}}\left( {x_n^\delta } \right) - y_{{i_n}}^\delta } \right\|^{r-1}{\delta _{{i_n}}} \le \frac{1}{\tau }{t_n^\delta}{\left\| {{F_{{i_n}}}\left( {x_n^\delta } \right) - y_{{i_n}}^\delta } \right\|^r}.\]
Inserting this into the previous inequality and combining with (\ref{bregmantui}), we can obtain
\begin{equation}\label{bregdescent}
\begin{array}{l}
\quad {D_{\xi _{n + 1}^\delta }}\theta \left( {\hat x,x_{n + 1}^\delta } \right) - {D_{\xi _n^\delta }}\theta \left( {\hat x,x_n^\delta } \right)\\[1mm]
 \le \frac{1}{{{p^*}{{\left( {2{\sigma}} \right)}^{{p^*} - 1}}}}{{\left\| {\xi _{n + 1}^\delta  - \xi _n^\delta } \right\|}^{{p^*}}} - \left( {1 - \eta  - \frac{{1 + \eta }}{\tau }} \right)t_n^\delta {\left\| {{F_{{i_n}}}\left( {x_n^\delta } \right) - y_{{i_n}}^\delta } \right\|^r}.
\end{array}
\end{equation}
According to the definition of $\xi_{n+1}^\delta$, we have
\[{{\left\| {\xi _{n + 1}^\delta  - \xi _n^\delta } \right\|}^{{p^*}}} ={\left( {t_n^\delta } \right)^{{p^*}}}{\left\| {{F_{{i_n}}'}{{\left( {x_n^\delta } \right)}^*}J_r^{{{\cal Y}_{{i_n}}}}\left( {{F_{{i_n}}}\left( {x_n^\delta } \right) - y_{{i_n}}^\delta } \right)} \right\|^{{p^*}}}.\]
In view of the definition of $t_n^\delta$ in (\ref{3choicestep}), one can see that
\[t_n^\delta  \le \frac{{{\mu _0}{{\left\| {{F_{{i_n}}}\left( {x_n^\delta } \right) - y_{{i_n}}^\delta } \right\|}^{(p - 1)r}}}}{{{{\left\| {{F_{{i_n}}'}{{\left( {x_n^\delta } \right)}^*}J_r^{{\mathcal{Y}_{{i_n}}}}\left( {{F_{{i_n}}}\left( {x_n^\delta } \right) - y_{{i_n}}^\delta } \right)} \right\|}^p}}},\]
from which there follows
\[{\left( {t_n^\delta } \right)^{{p^*} - 1}}{\left\| {{F_{{i_n}}'}{{\left( {x_n^\delta } \right)}^*}J_r^{{\mathcal{Y}_{{i_n}}}}\left( {{F_{{i_n}}}\left( {x_n^\delta } \right) - y_{{i_n}}^\delta } \right)} \right\|^{{p^*}}} \le \mu _0^{{p^*} - 1}{\left\| {{F_{{i_n}}}\left( {x_n^\delta } \right) - y_{{i_n}}^\delta } \right\|^r}.\]
Therefore, we derive the estimate
\[{\left\| {\xi _{n + 1}^\delta  - \xi _n^\delta } \right\|^{{p^*}}} \le t_n^\delta \mu _0^{{p^*} - 1}{\left\| {{F_{{i_n}}}\left( {x_n^\delta } \right) - y_{{i_n}}^\delta } \right\|^r},\]
which together with (\ref{bregdescent}) implies (\ref{bregdes1}).

Next we show that  $x_n^\delta \in {B_{2\rho} }\left( {{x_0}} \right)$ for all $n\ge 0$. From (\ref{6xdag}) and Assumption \ref{5asumption theta}, it follows that $\left\| {{x^\dag } - {x_0}} \right\| \le \rho $, i.e., $x^\dag$ is a solution of (\ref{nonlinear equation}) in ${B_\rho }\left( {{x_0}} \right)\cap \mathscr{D}\left( \theta  \right)$. Then, using (\ref{bregdes1}) with $\hat x=x^\dag$, we can inductively prove that
\[{D_{\xi _n^\delta }}\theta \left( {{x^\dag },x_n^\delta } \right) \le {D_{{\xi _0}}}\theta \left( {{x^\dag },{x_0}} \right) \le \sigma {\rho ^p},\]
which, together with the $p$-convexity of $\theta$, leads to $\left\| {{x^\dag } - x_n^\delta } \right\| \le \rho $ and $\left\| {{x^\dag } - {x_0}} \right\| \le \rho $ and consequently $\left\| {x_n^\delta  - {x_0}} \right\| \le 2\rho $ for all $n\ge 0$. Therefore, $x_n^\delta \in {B_{2\rho} }\left( {{x_0}} \right)$ for $n\ge 0$. Consequently, (\ref{bregdes1}) holds for all $n\ge 0$, which yields the first assertion  (\ref{monobreg}) immediately.

Denoting by $\mathcal{F}_n$ the $\sigma$-algebra generated by $\left( {{i_0}, \ldots ,{i_{n - 1}}} \right)$, by taking expectation of (\ref{bregdes1}) conditioned on $\mathcal{F}_n$ and utilizing the measurability of $x_{n}^\delta$ with respect to $\mathcal{F}_n$, we have
\[\begin{array}{l}
\mathbb{E}\left[ {{D_{\xi _{n + 1}^\delta }}\theta \left( {\hat x,x_{n + 1}^\delta } \right) - {D_{\xi _n^\delta }}\theta \left( {\hat x,x_n^\delta } \right)|{\mathcal{F}_n}} \right] \le  - {c_0}\mathbb{E}\left[ {t_n^\delta {{\left\| {{F_{{i_n}}}\left( {x_n^\delta } \right) - y_{{i_n}}^\delta } \right\|}^r}|{\mathcal{F}_n}} \right]\\[1mm]
~~~~~~~~~~~~~~~~~~~~~~~~~~~~~~~~~~~~~~~~~~~~~~~~~~~~ =  - \frac{{{c_0}}}{N}\sum\limits_{i = 1}^N {t_n^\delta {{\left\| {{F_i}\left( {x_n^\delta } \right) - y_i^\delta } \right\|}^r}}.
\end{array}\]
Finally, taking the full expectation and making use of the linearity of the expectation yields the second assertion (\ref{exmonobreg}).
\end{proof}
\begin{remark}\label{stepsize}
In the proof of Lemma \ref{smonotoniciy}, the parameters $\mu_0>0$ and $\tau>1$ of the step size (\ref{3choicestep}) need to satisfy (\ref{c1}). The parameter $\mu_1>0$ can be selected to be any positive number; in fact, $\mu_1$ represents an upper bound of $\tilde t_n^\delta$, which is useful for the following theoretical analysis. To allow large step size, one can take $\mu_1$ to be a large number in the numerical simulations.

\end{remark}
\subsection{Convergence analysis for exact data}
To establish the regularization property of the SGD-$\theta$ method, in this subsection we investigate its counterpart in the noise-free case and give the corresponding convergence result. For given initial guesses $x_0 \in \mathscr{D}\left( {\partial \theta } \right)$ and $\xi _0  \in \partial \theta\left( {x_0} \right)$, by removing the subscript $\delta$ in all the quantities involved in (\ref{sgd}), the SGD-$\theta$ method with exact data takes the following form
\begin{equation}\label{sgdex}
\begin{array}{l}
{\xi _{n + 1}} = {\xi _n} - {t_n}{F_{{i_n}}'}{\left( {{x_n}} \right)^*}J_r^{\mathcal{Y}_{i_n}}\left( {{F_{{i_n}}}\left( {{x_n}} \right) - {y_{{i_n}}}} \right),\\[1mm]
{x_{n + 1}} = \arg \mathop {\min }\limits_{x \in \mathcal{X}} \left\{ {\theta \left( x \right) - \left\langle {{\xi _{n + 1}},x} \right\rangle } \right\},
\end{array}
\end{equation}
where the random index ${i_n}$ is chosen uniformly from the set $ \left\{ {1,2, \cdots ,N} \right\}$ and
\begin{equation}\label{eqs35}
{t_n} = \left\{ {\begin{array}{*{20}{l}}
{{{\tilde t}_n}{{\left\| {{F_{{i_n}}}\left( {{x_n}} \right) - {y_{{i_n}}}} \right\|}^{p - r}},~~~~{F_{{i_n}}}\left( {{x_n}} \right) \ne {y_{{i_n}}},}\\[1mm]
{0,~~~~~~~~~~~~~~~~~~~~~~~~~~~~~~~otherwise,}
\end{array}} \right.
\end{equation}
with
\[{{\tilde t}_n} = \min \left\{ {\frac{{{\mu _0}{{\left\| {{F_{{i_n}}}\left( {{x_n}} \right) - {y_{{i_n}}}} \right\|}^{p(r - 1)}}}}{{{{\left\| {{F_{{i_n}}'}{{\left( {{x_n}} \right)}^*}J_r^{{{\cal Y}_{{i_n}}}}\left( {{F_{{i_n}}}\left( {{x_n}} \right) - {y_{{i_n}}}} \right)} \right\|}^p}}},{\mu _1}} \right\}\]
for some positive constants $\mu_0$ and $\mu_1$.
\begin{lemma}\label{descentexact}
Let Assumptions \ref{5asumption theta} and \ref{5asumption operator} hold. Assume that $\mu_0>0$ is chosen such that
\[{c_1}: = 1 - \eta  - \frac{{p - 1}}{p}{\left( {\frac{{{\mu _0}}}{{2\sigma }}} \right)^{\frac{1}{{p - 1}}}} > 0.\]
Then, $x_n \in {B_{2\rho} }\left( {{x_0}} \right)$ for all $n$ and for any solution $\hat x$ of (\ref{nonlinear equation}) in $ {B_{2\rho} }\left( {{x_0}} \right)\cap \mathscr{D}\left( \theta  \right)$, there hold
\[\begin{array}{l}
{D_{{\xi _{n + 1}}}}\theta \left( {\hat x,{x_{n + 1}}} \right) \le {D_{{\xi _n}}}\theta \left( {\hat x,{x_n}} \right)
\end{array}
\]
and
\begin{equation}\label{monoex}
\hspace{-2mm}
\begin{array}{l}
\quad \mathbb{E}\left[ {{D_{{\xi _{n + 1}}}}\theta \left( {\hat x,{x_{n + 1}}} \right)} \right] - \mathbb{E}\left[ {{D_{{\xi _n}}}\theta \left( {\hat x,{x_n}} \right)} \right]
\le - {{c_1}{\mu _2}}\mathbb{E}\left[\frac{1}{N} {\sum\limits_{i = 1}^N {{{\left\| {{F_i}\left( {{x_n}} \right) - {y_i}} \right\|}^p}} } \right]
\end{array}
\end{equation}
for all $n\ge 0$, where ${\mu _2}:=  \min \left\{ {{\mu _0}B_0^{ - p},{\mu _1}} \right\}$.
\end{lemma}
\begin{proof}
By the argument analogous to the proof of Lemma \ref{smonotoniciy}, we can prove that  $x_n \in {B_{2\rho} }\left( {{x_0}} \right)$ for all $n$ and for any solution $\hat x$ of (\ref{nonlinear equation}) in $ {B_{2\rho} }\left( {{x_0}} \right)\cap \mathscr{D}\left( \theta  \right)$,
\begin{equation}\label{desexdirect}
\begin{array}{l}
{D_{{\xi _{n + 1}}}}\theta \left( {\hat x,{x_{n + 1}}} \right) - {D_{{\xi _n}}}\theta \left( {\hat x,{x_n}} \right) \le  - {c_1}{t_n}{\left\| {{F_{{i_n}}}\left( {{x_n}} \right) - {y_{{i_n}}}} \right\|^r}.
\end{array}
\end{equation}
Next we show the assertion (\ref{monoex}). If  ${F_{{i_n}}}\left( {{x_n}} \right) \ne {y_{{i_n}}}$, then by using  (\ref{eqs35}),  $\left\| {F_{i}'\left( x \right)} \right\| \le {B_0}$ for all $x\in{B_{2\rho }}\left( {{x_0}} \right)$ and $i=1,2,\ldots,N$ from Assumption \ref{5asumption operator} (d), and the property of $J_r^{\mathcal{Y}_{i_n}}$,  we have
\[\begin{array}{l}
{t_n} = \min \left\{ {\frac{{{\mu _0}{{\left\| {{F_{{i_n}}}\left( {{x_n}} \right) - {y_{{i_n}}}} \right\|}^{p(r - 1)}}}}{{{{\left\| {{F_{{i_n} }^\prime}{{\left( {{x_n}} \right)}^*}J_r^{{{\cal Y}_{{i_n}}}}\left( {{F_{{i_n}}}\left( {{x_n}} \right) - {y_{{i_n}}}} \right)} \right\|}^p}}},{\mu _1}} \right\}{\left\| {{F_{{i_n}}}\left( {{x_n}} \right) - {y_{{i_n}}}} \right\|^{p - r}}\\[1mm]
 ~~~\ge \min \left\{ {{\mu _0}B_0^{ - p},{\mu _1}} \right\}{\left\| {{F_{{i_n}}}\left( {{x_n}} \right) - {y_{{i_n}}}} \right\|^{p - r}}= {\mu _2}{\left\| {{F_{{i_n}}}\left( {{x_n}} \right) - {y_{{i_n}}}} \right\|^{p-r}}.
\end{array}\]
This together with  (\ref{desexdirect}) leads  to
\begin{equation}\label{desexdirect1}
\begin{array}{l}
{D_{{\xi _{n + 1}}}}\theta \left( {\hat x,{x_{n + 1}}} \right) - {D_{{\xi _n}}}\theta \left( {\hat x,{x_n}} \right)\le  - {c_1}{\mu _2}{\left\| {{F_{{i_n}}}\left( {{x_n}} \right) - {y_{{i_n}}}} \right\|^p},
\end{array}
\end{equation}
which holds automatically when ${F_{{i_n}}}\left( {{x_n}} \right) = {y_{{i_n}}}$, since (\ref{eqs35}) forces $t_n=0$ and subsequently $\xi_{n+1}=\xi_n$ and $x_{n+1}=x_n$.

By taking full expectation of (\ref{desexdirect1}), it follows that
\begin{equation*}
\begin{array}{l}
\quad \mathbb{E}\left[ {{D_{{\xi _{n + 1}}}}\theta \left( {\hat x,{x_{n + 1}}} \right)} \right] - \mathbb{E}\left[ {{D_{{\xi _n}}}\theta \left( {\hat x,{x_n}} \right)} \right]\\[1mm]
\le  - {c_1}{\mu _2}\mathbb{E}\left[ {{{\left\| {{F_{{i_n}}}\left( {{x_n}} \right) - {y_{{i_n}}}} \right\|}^p}} \right]=  - {c_1}{\mu _2}\mathbb{E}\left[ {\mathbb{E}\left[ {{{\left\| {{F_{{i_n}}}\left( {{x_n}} \right) - {y_{{i_n}}}} \right\|}^p}|{{\cal F}_n}} \right]} \right]\\[1mm]
=  - {{c_1}{\mu _2}}\mathbb{E}\left[\frac{1}{N} {\sum\limits_{i = 1}^N {{{\left\| {{F_i}\left( {{x_n}} \right) - {y_i}} \right\|}^p}} } \right],
\end{array}
\end{equation*}
which gives the assertion (\ref{monoex}). This completes the proof.
\end{proof}
Now we are ready to show the convergence in expectation of the  SGD-$\theta$ method (\ref{sgdex}) with exact data.
\begin{theorem}\label{sconvergence}
Let all conditions in Lemma \ref{descentexact}  hold.
Then, there exists a solution $x^* \in {B_{2\rho} }\left( {{x_0}} \right)\cap \mathscr{D}\left( \theta  \right)$ of (\ref{nonlinear equation}) such that
\[\mathbb{E}\left[ {{{\left\| {{x^*} - {x_n}} \right\|}^p}} \right] \to 0 ~~\text{and}~~\mathbb{E}\left[ {{D_{{\xi _n}}}\theta \left( {{x^*},{x_n}} \right)} \right] \to 0\]
as $n\to \infty$. If $\mathcal{N}\left( {F_i'\left( {{x^\dag }} \right)} \right) \subset \mathcal{N}\left( {F_i'\left( x \right)} \right)$ for all $x \in {B_{2\rho }}\left( {{x_0}} \right)$ and $i=1,2,\dots,N$, then $x^*=x^\dag$.
\end{theorem}
\begin{proof}
We first show that the sequence $\left\{ {{x_n}} \right\}$ has a convergent subsequence. To this end, we consider
\[R_n:=\mathbb{E}\left[ \frac{1}{N}{\sum\limits_{i = 1}^N {{{\left\| {{F_i}\left( {{x_n}} \right) - {y_i}} \right\|}^p}} } \right].\]
From (\ref{monoex}) in Lemma \ref{descentexact}, it follows that $\left\{ \mathbb{E}\left[ {{D_{{\xi _n}}}\theta \left( {\hat x,{x_n}} \right)} \right]\right\}$ is  monotonically decreasing and thus
\begin{equation}\label{emxiu}
\mathop {\lim }\limits_{n \to \infty } \mathbb{E}\left[ {{D_{{\xi _n}}}\theta \left( {\hat x,{x_n}} \right)} \right] = \varepsilon
\end{equation}
for some $\varepsilon\ge 0$. Rearranging the terms in (\ref{monoex}), we further have
\[\hspace{-5mm}
\begin{array}{*{20}{l}}
{{c_1}{\mu _2}{R_n} = {c_1}{\mu _2}\mathbb{E}\left[ {\frac{1}{N}\sum\limits_{i = 1}^N {{{\left\| {{F_i}\left( {{x_n}} \right) - {y_i}} \right\|}^p}} } \right] \le \mathbb{E}\left[ {{D_{{\xi _n}}}\theta \left( {\hat x,{x_n}} \right)} \right]}-\mathbb{E}\left[ {{D_{{\xi _{n + 1}}}}\theta \left( {\hat x,{x_{n + 1}}} \right)} \right],\end{array}\]
which, together with (\ref{emxiu}), implies that
 \begin{equation}\label{emxiu1}
\mathop {\lim }\limits_{n \to \infty } R_n= 0.
\end{equation}
If there is an integer $n$ such that $R_n= 0$, we then have
\[\sum\limits_{i = 1}^N {{{\left\| {{F_i}\left( {{x_n}} \right) - {y_i}} \right\|}^p}}  = 0
\]
along each sample path $\left( {{i_0}, \cdots ,{i_{n - 1}}} \right)$ since there are only a finite number of sample paths and the probability of each path is positive; consequently $\xi_m=\xi_n$ and $x_m=x_n$ and thus $R_m=0$ for $m \ge n$. Based on these facts, we can choose a strictly increasing subsequence $\left\{ {{n_l}} \right\}$ of integers by setting $n_0=0$ and, for each $l\ge 1$, letting $n_l$ be the the first integer such that
\[{n_l} \ge {n_{l - 1}} + 1~~\text{and}~~{R_{{n_l}}} \le {R_{{n_{l - 1}}}}.\]
For such sequence, it is easy to show that
\begin{equation}\label{nlxiaoyun}
{R_{{n_l}}} \le {R_n},~~\forall 0 \le n \le {n_l}.
\end{equation}

With above chosen $\left\{ {{n_l}} \right\}$, we next prove that the subsequence $\left\{ {{x_{n_l}}} \right\} $ is a Cauchy sequence in $ {L^p}\left( {\Omega ,\mathcal{X}} \right)$ (${L^p}\left( {\Omega ,\mathcal{X}} \right)$, the space containing all $\mathcal{X}$-valued random variables $x$ such that $\mathbb{E}\left[ {{{\left\| x \right\|}^p}} \right]<\infty$, is a Banach space with respect to the norm
${\left( {\mathbb{E}\left[ {{{\left\| x \right\|}^p}} \right]} \right)^{\frac{1}{p}}}$).
For this purpose, we first show that

 \begin{equation}\label{edxnlk}
 \mathop {\sup }\limits_{l \ge k} \mathbb{E}\left[ {{D_{{\xi _{{n_k}}}}}\theta \left( {{x_{{n_l}}},{x_{{n_k}}}} \right)} \right] \to 0~~\text{as}~~k \to \infty.
 \end{equation}
 In view of (\ref{bregman distance}), we have for any $l>k$ that
 \[{D_{{\xi _{{n_k}}}}}\theta \left( {{x_{{n_l}}},{x_{{n_k}}}} \right) = {D_{{\xi _{{n_k}}}}}\theta \left( {\hat x,{x_{{n_k}}}} \right) - {D_{{\xi _{{n_l}}}}}\theta \left( {\hat x,{x_{{n_l}}}} \right) + \left\langle {{\xi _{{n_l}}} - {\xi _{{n_k}}},{x_{{n_l}}} - \hat x} \right\rangle. \]
 By taking the expectation, we further arrive at
\begin{equation}\label{xnlk}
\begin{array}{l}
\mathbb{E}\left[ {{D_{{\xi _{{n_k}}}}}\theta \left( {{x_{{n_l}}},{x_{{n_k}}}} \right)} \right] = \mathbb{E}\left[ {{D_{{\xi _{{n_k}}}}}\theta \left( {\hat x,{x_{{n_k}}}} \right)} \right] - \mathbb{E}\left[ {{D_{{\xi _{{n_l}}}}}\theta \left( {\hat x,{x_{{n_l}}}} \right)} \right]\\[1mm]
~~~~~~~~~~~~~~~~~~~~~~~~~~~~~ + \mathbb{E}\left[ {\left\langle {{\xi _{{n_l}}} - {\xi _{{n_k}}},{x_{{n_l}}} - \hat x} \right\rangle } \right].
 \end{array}
\end{equation}
From the definition of $\xi_{n+1}$, it follows that
\[\begin{array}{l}
\left\langle {{\xi _{{n_l}}} - {\xi _{{n_k}}},{x_{{n_l}}} - \hat x} \right\rangle  = \sum\limits_{n = {n_k}}^{{n_l} - 1} {\left\langle {{\xi _{n + 1}} - {\xi _n},{x_{{n_l}}} - \hat x} \right\rangle } \\[1mm]
~~~~~~~~~~~~~~~~~~~~~~~~~=  - \sum\limits_{n = {n_k}}^{{n_l} - 1} {{t_n}\left\langle {J_r^{{{\cal Y}_{{i_n}}}}\left( {{F_{{i_n}}}\left( {{x_n}} \right) - {y_{{i_n}}}} \right),{F_{{i_n}}'}\left( {{x_n}} \right)\left( {{x_{{n_l}}} - \hat x} \right)} \right\rangle }.
\end{array}\]
By taking the expectation and utilizing the triangle inequality, we get that
\[\begin{array}{*{20}{l}}
{\quad \left| \mathbb{E}{\left[ {\left\langle {{\xi _{{n_l}}} - {\xi _{{n_k}}},{x_{{n_l}}} - \hat x} \right\rangle } \right]} \right|}\\[1mm]
{ \le \sum\limits_{n = {n_k}}^{{n_l} - 1} {\left| \mathbb{E}{\left[ {\left\langle {{t_n}J_r^{{{\cal Y}_{{i_n}}}}\left( {{F_{{i_n}}}\left( {{x_n}} \right) - {y_{{i_n}}}} \right),{F_{{i_n}}'}\left( {{x_n}} \right)\left( {{x_{{n_l}}} -{x_n}+{x_n} -\hat x} \right)} \right\rangle } \right]} \right|} }\\[1mm]
{ \le \sum\limits_{n = {n_k}}^{{n_l} - 1} {\left| \mathbb{E}{\left[ {\left\langle {{t_n}J_r^{{{\cal Y}_{{i_n}}}}\left( {{F_{{i_n}}}\left( {{x_n}} \right) - {y_{{i_n}}}} \right),{F_{{i_n}}'}\left( {{x_n}} \right)\left( {{x_n} -\hat x} \right)} \right\rangle } \right]} \right|}} \\[1.5mm]
{~~+\sum\limits_{n = {n_k}}^{{n_l} - 1} {\left| \mathbb{E}{\left[ {\left\langle {{t_n}J_r^{{{\cal Y}_{{i_n}}}}\left( {{F_{{i_n}}}\left( {{x_n}} \right) - {y_{{i_n}}}} \right),{F_{{i_n}}'}\left( {{x_n}} \right)\left( {{x_{{n_l}}} -{x_n}} \right)} \right\rangle } \right]} \right|}: = I + II.}
\end{array}\]
 Next we estimate the two terms $I$ and $II$. By using the H$\ddot{o}$lder inequality
 \begin{equation}\label{holder}
 \left| {\mathbb{E}\left[ {\left\langle {\xi ,\eta } \right\rangle } \right]} \right| \le {\left( {\mathbb{E}\left[ {{{\left\| \xi  \right\|}^p}} \right]} \right)^{{1 \mathord{\left/
 {\vphantom {1 p}} \right.
 \kern-\nulldelimiterspace} p}}}{\left( {\mathbb{E}\left[ {{{\left\| \eta  \right\|}^{{p^*}}}} \right]} \right)^{{1 \mathord{\left/
 {\vphantom {1 {{p^*}}}} \right.
 \kern-\nulldelimiterspace} {{p^*}}}}}, \frac{1}{p} + \frac{1}{{{p^*}}} = 1
 \end{equation}
with $\xi={F_{{i_n}}'}\left( {{x_n}} \right)\left( {{x_n} -\hat x} \right)$ and $\eta={t_n}J_r^{{{\cal Y}_{{i_n}}}}\left( {{F_{{i_n}}}\left( {{x_n}} \right) - {y_{{i_n}}}} \right)$, and  together with the definition of $t_n$ ($t_n\le \mu_1{{\left\| {{F_{{i_n}}}\left( {{x_n}} \right) - {y_{{i_n}}}} \right\|}^{p - r}}$) and the property of $J_r^{{\mathcal{Y}_{{i_n}}}}$, the term $I$ can be estimated as
\begin{equation}\label{esimateI}
\hspace{-10mm}
\begin{array}{*{20}{l}}
I  {\le \sum\limits_{n = {n_k}}^{{n_l} - 1} {{{\left( {\mathbb{E}\left[ {{\left( {t_n}\right)^{p^*}{\left\| {J_r^{{\mathcal{Y}_{{i_n}}}}\left( {{F_{{i_n}}}\left( {{x_n}} \right) - {y_{{i_n}}}} \right)} \right\|}^{{p^*}}}} \right]} \right)}^{{{{1 \mathord{\left/
 {\vphantom {1 {{p^*}}}} \right.
 \kern-\nulldelimiterspace} {{p^*}}}}}}}} {\left( {\mathbb{E}\left[ {{{\left\| {{F_{{i_n}}'}\left( {{x_n}} \right)\left( {{x_n} - \hat x} \right)} \right\|}^p}} \right]} \right)^{{{1 \mathord{\left/
 {\vphantom {1 p}} \right.
 \kern-\nulldelimiterspace} p}}}}}\\[1mm]
{~ \le {\mu _1}\sum\limits_{n = {n_k}}^{{n_l} - 1} {{{\left( \mathbb{E}{\left[ {{{\left\| {{F_{{i_n}}}\left( {{x_n}} \right) - {y_{{i_n}}}} \right\|}^p}} \right]} \right)}^{{{(p - 1)} \mathord{\left/
 {\vphantom {{(p - 1)} p}} \right.
 \kern-\nulldelimiterspace} p}{\rm{ }}}}{{\left( \mathbb{E}{\left[ {{{\left\| {{F_{{i_n}}'}\left( {{x_n}} \right)\left( {{x_n} - \hat x} \right)} \right\|}^p}} \right]} \right)}^{{1 \mathord{\left/
 {\vphantom {1 p}} \right.
 \kern-\nulldelimiterspace} p}}}} }\\[1mm]
~\le {\mu _1}\left( {1 + \eta } \right)\sum\limits_{n = {n_k}}^{{n_l} - 1} {\mathbb{E}\left[ {{{\left\| {{F_{{i_n}}}\left( {{x_n}} \right) - {y_{{i_n}}}} \right\|}^p}} \right]},
\end{array}
\end{equation}
where we used Assumption \ref{5asumption operator}(d) in the last inequality. Similarly, by using the H$\ddot{o}$lder inequality, Assumption \ref{5asumption operator}(d), (\ref{eqs35}) and the property of $J_r^{{\mathcal{Y}_{{i_n}}}}$, we have
\[\hspace{-7mm}
\begin{array}{l}
II
{\le \sum\limits_{n = {n_k}}^{{n_l} - 1} {{{\left( {\mathbb{E}\left[ {{\left( {t_n}\right)^{p^*}{\left\| {J_r^{{\mathcal{Y}_{{i_n}}}}\left( {{F_{{i_n}}}\left( {{x_n}} \right) - {y_{{i_n}}}} \right)} \right\|}^{{p^*}}}} \right]} \right)}^{{{{1 \mathord{\left/
 {\vphantom {1 {{p^*}}}} \right.
 \kern-\nulldelimiterspace} {{p^*}}}}}}}} {\left( {\mathbb{E}\left[ {{{\left\| {{F_{{i_n}}'}\left( {{x_n}} \right)\left( {{x_{n_l}} -  {x_n}} \right)} \right\|}^p}} \right]} \right)^{{{1 \mathord{\left/
 {\vphantom {1 p}} \right.
 \kern-\nulldelimiterspace} p}}}}}\\[1mm]
~~\le {\mu _1}\sum\limits_{n = {n_k}}^{{n_l} - 1} {{{\left( \mathbb{E}{\left[ {{{\left\| {{F_{{i_n}}}\left( {{x_n}} \right) - {y_{{i_n}}}} \right\|}^p}} \right]} \right)}^{{{(p - 1)} \mathord{\left/
 {\vphantom {{(p - 1)} p}} \right.
 \kern-\nulldelimiterspace} p}}}{{\left( \mathbb{E}{\left[ {{{\left\| {{F_{{i_n}}'}\left( {{x_n}} \right)\left( {{x_{{n_l}}} - {x_n}} \right)} \right\|}^p}} \right]} \right)}^{{1 \mathord{\left/
 {\vphantom {1 p}} \right.
 \kern-\nulldelimiterspace} p}}}}\\[1mm]
~~\le {\mu _1}\left( {1 + \eta } \right)\sum\limits_{n = {n_k}}^{{n_l} - 1} {{{\left( {\mathbb{E}\left[ {{{\left\| {{F_{{i_n}}}\left( {{x_n}} \right) - {y_{{i_n}}}} \right\|}^p}} \right]} \right)}^{{(p-1) \mathord{\left/
 {\vphantom {1 p}} \right.
 \kern-\nulldelimiterspace} p}}}{{\left( {\mathbb{E}\left[ {{{\left\| {{F_{{i_n}}}\left( {{x_{{n_l}}}} \right) - {F_{{i_n}}}\left( {{x_n}} \right)} \right\|}^p}} \right]} \right)}^{{1 \mathord{\left/
 {\vphantom {1 2}} \right.
 \kern-\nulldelimiterspace} p}}}}.
\end{array}\]
By using the triangle inequality of the norm  ${\left( {\mathbb{E}\left[ {{{\left\| \cdot\right\|}^p}} \right]} \right)^{\frac{1}{p}}}$, we further derive that
\[\hspace{-17mm}
\begin{array}{l}
 {{{\left( {\mathbb{E}\left[ {{{\left\| {{F_{{i_n}}}\left( {{x_n}} \right) - {y_{{i_n}}}} \right\|}^p}} \right]} \right)}^{{(p-1) \mathord{\left/
 {\vphantom {1 p}} \right.
 \kern-\nulldelimiterspace} p}}}{{\left( {\mathbb{E}\left[ {{{\left\| {{F_{{i_n}}}\left( {{x_{{n_l}}}} \right) - {F_{{i_n}}}\left( {{x_n}} \right)} \right\|}^p}} \right]} \right)}^{{1 \mathord{\left/
 {\vphantom {1 2}} \right.
 \kern-\nulldelimiterspace} p}}}}\\[1mm]
~~~~~~~~~~~={{{\left( {\mathbb{E}\left[ {{{\left\| {{F_{{i_n}}}\left( {{x_n}} \right) - {y_{{i_n}}}} \right\|}^p}} \right]} \right)}^{{(p-1) \mathord{\left/
 {\vphantom {1 p}} \right.
 \kern-\nulldelimiterspace} p}}}{{\left( {\mathbb{E}\left[ {{{\left\| {{F_{{i_n}}}\left( {{x_{{n_l}}}} \right) -y_{i_n}+y_{i_n}- {F_{{i_n}}}\left( {{x_n}} \right)} \right\|}^p}} \right]} \right)}^{{1 \mathord{\left/
 {\vphantom {1 2}} \right.
 \kern-\nulldelimiterspace} p}}}}\\[1mm]
~~~~~~~~~~~\le {{{\left( {\mathbb{E}\left[ {{{\left\| {{F_{{i_n}}}\left( {{x_n}} \right) - {y_{{i_n}}}} \right\|}^p}} \right]} \right)}^{{(p-1) \mathord{\left/
 {\vphantom {1 2}} \right.
 \kern-\nulldelimiterspace} p}}}{{\left( {\mathbb{E}\left[ {{{\left\| {{F_{{i_n}}}\left( {{x_{{n_l}}}} \right) - {y_{{i_n}}}} \right\|}^p}} \right]} \right)}^{{1 \mathord{\left/
 {\vphantom {1 2}} \right.
 \kern-\nulldelimiterspace} p}}}}+{\mathbb{E}\left[ {{{\left\| {{F_{{i_n}}}\left( {{x_n}} \right) - {y_{{i_n}}}} \right\|}^p}} \right]}.
 \end{array}\]
We therefore obtain the following estimate
\begin{equation}\label{estimateII}
\hspace{-5mm}
\begin{array}{l}
II \le {\mu _1}\left( {1 + \eta } \right)\sum\limits_{n = {n_k}}^{{n_l} - 1} {{{\left( {\mathbb{E}\left[ {{{\left\| {{F_{{i_n}}}\left( {{x_n}} \right) - {y_{{i_n}}}} \right\|}^p}} \right]} \right)}^{{(p-1) \mathord{\left/
 {\vphantom {1 2}} \right.
 \kern-\nulldelimiterspace} p}}}{{\left( {\mathbb{E}\left[ {{{\left\| {{F_{{i_n}}}\left( {{x_{{n_l}}}} \right) - {y_{{i_n}}}} \right\|}^p}} \right]} \right)}^{{1 \mathord{\left/
 {\vphantom {1 2}} \right.
 \kern-\nulldelimiterspace} p}}}}\\[1mm]
 ~~~~~~+{\mu _1}\left( {1 + \eta } \right)\sum\limits_{n = {n_k}}^{{n_l} - 1} {\mathbb{E}\left[ {{{\left\| {{F_{{i_n}}}\left( {{x_n}} \right) - {y_{{i_n}}}} \right\|}^p}} \right]}.
 \end{array}
 \end{equation}
Further, we have from the measurability of $x_n^\delta$ with respect to $\mathcal{F}_n$ and the definition of $R_n$ that
\[\mathbb{E}\left[ {{{\left\| {{F_{{i_n}}}\left( {{x_n}} \right) - {y_{{i_n}}}} \right\|}^p}} \right] = \mathbb{E}\left[ {\mathbb{E}\left[ {{{\left\| {{F_{{i_n}}}\left( {{x_n}} \right) - {y_{{i_n}}}} \right\|}^p}|{\mathcal{F}_n}} \right]} \right] = {R_n},\]
which, together with (\ref{esimateI}), implies
\[I \le {\mu _1}\left( {1 + \eta } \right)\sum\limits_{n = {n_k}}^{{n_l} - 1} {{R_n}}.\]
However, we can not treat $\mathbb{E}\left[ {{{\left\| {{F_{{i_n}}}\left( {{x_{{n_l}}}} \right) - {y_{{i_n}}}} \right\|}^p}} \right]$ in the first term on the right hand side of (\ref{estimateII}) similarly as $x_{n_l}$ is not necessarily $\mathcal{F}_n$-measurable. Observing that
\[{\left\| {{F_{{i_n}}}\left( {{x_{{n_l}}}} \right) - {y_{{i_n}}}} \right\|^p} \le \sum\limits_{i = 1}^N {{{\left\| {{F_i}\left( {{x_{{n_l}}}} \right) - {y_i}} \right\|}^p}}, \]
from which we can see
\[\mathbb{E}\left[ {{{\left\| {{F_{{i_n}}}\left( {{x_{{n_l}}}} \right) - {y_{{i_n}}}} \right\|}^p}} \right] \le \mathbb{E}\left[ {\sum\limits_{i = 1}^N {{{\left\| {{F_i}\left( {{x_{{n_l}}}} \right) - {y_i}} \right\|}^p}} } \right] = N{R_{{n_l}}}.\]
Inserting the above inequality into (\ref{estimateII}), and together with (\ref{nlxiaoyun}), we can derive that
 \[\begin{array}{*{20}{l}}
{II \le {\mu _1}\left( {1 + \eta } \right)\left( { N^{\frac{1}{p}} \sum\limits_{n = {n_k}}^{{n_l} - 1} {R_n^{{\left( p-1\right) \mathord{\left/
 {\vphantom {1 2}} \right.
 \kern-\nulldelimiterspace} p}}} R_{{n_l}}^{{1 \mathord{\left/
 {\vphantom {1 2}} \right.
 \kern-\nulldelimiterspace} p}}+\sum\limits_{n = {n_k}}^{{n_l} - 1} {{R_n}} } \right)}\\[1mm]
{ ~~~ \le {\mu _1}\left( {1 + \eta } \right)\left( {1 + N^{\frac{1}{p}}  } \right)\sum\limits_{n = {n_k}}^{{n_l} - 1} {{R_n}} .}
\end{array}\]
Combining the above estimation on $I$ and $II$, and by making use of (\ref{monoex}), we have  with $c_2:={ \frac{{{\mu _1}\left( {1 + \eta } \right)\left( {2 + N^{\frac{1}{p}} } \right)}}{{{c_1}{\mu _2}}}}$ that
\begin{equation}\label{exnlxnk}
\begin{array}{l}
\left| \mathbb{E}\left[ {\left\langle {{\xi _{{n_l}}} - {\xi _{{n_k}}},{x_{{n_l}}} - \hat x} \right\rangle } \right] \right| \le {\mu _1}\left( {1 + \eta } \right)\left( {2 + N^{\frac{1}{p}} } \right)\sum\limits_{n = {n_k}}^{{n_l} - 1} {{R_n}} \\[1mm]
~~~~~~~~~~~~~~~~~~~~~~~~~~~~~~\le c_2\left( {\mathbb{E}\left[ {{D_{{\xi _{{n_k}}}}}\theta \left( {\hat x,{x_{{n_k}}}} \right)} \right] - \mathbb{E}\left[ {{D_{{\xi _{{n_l}}}}}\theta \left( {\hat x,{x_{{n_l}}}} \right)} \right]} \right).
\end{array}
\end{equation}
By inserting (\ref{exnlxnk}) into (\ref{xnlk}), we arrive at
\[\mathbb{E}\left[ {{D_{{\xi _{{n_k}}}}}\theta \left( {{x_{{n_l}}},{x_{{n_k}}}} \right)} \right] \le (1+c_2)\left( {\mathbb{E}\left[ {{D_{{\xi _{{n_k}}}}}\theta \left( {\hat x,{x_{{n_k}}}} \right)} \right] - \mathbb{E}\left[ {{D_{{\xi _{{n_l}}}}}\theta \left( {\hat x,{x_{{n_l}}}} \right)} \right]} \right),\]
which, together with (\ref{emxiu}), leads to (\ref{edxnlk}).
By the  $p$-convexity of $\theta$, we further have
\[\mathop {\sup }\limits_{l \ge k} \mathbb{E}\left[ {{{\left\| {{x_{{n_l}}} - {x_{{n_k}}}} \right\|}^p}} \right] \to 0~~\text{as}~~k \to \infty, \]
which implies that $\left\{ {{x_{{n_l}}}} \right\}$ is a Cauchy sequence in ${L^p}\left( {\Omega,\mathcal{X}} \right)$.

Thus, there exists a random vector $x^* \in {B_{2\rho} }\left( {{x_0}} \right)$ in $ {L^p}\left( {\Omega ,\mathcal{X}} \right)$ such that
\begin{equation}\label{xnlcon}
\mathbb{E}\left[ {{{\left\| {{x_{{n_l}}} - {x^*}} \right\|}^p}} \right] \to 0~~\text{as}~~l \to \infty.
\end{equation}
By taking a subsequence of $\left\{ {{{{n_l}}}} \right\}$ if necessary, from (\ref{xnlcon}) and (\ref{emxiu1}), we derive that
\begin{equation}\label{xnlf}
\mathop {\lim }\limits_{l \to \infty } \left\| {{x_{{n_l}}} - {x^*}} \right\| = 0~~\text{and}~~\mathop {\lim }\limits_{l \to \infty } \sum\limits_{i = 1}^N {{{\left\| {{F_i}\left( {{x_{{n_l}}}} \right) - {y_i}} \right\|}^p}}  = 0
\end{equation}
almost surely. Consequently,
\[\sum\limits_{i = 1}^N {{{\left\| {{F_i}\left( {{x^*}} \right) - {y_i}} \right\|}^p}}  = 0\]
almost surely, which means that $x^*$ is a solution of (\ref{nonlinear equation}) almost surely.

Next we show ${x^*} \in \mathscr{D}\left( \theta  \right)$ almost surely. It is enough to prove $\mathbb{E}\left[ {\theta \left( {{x^*}} \right)} \right] < \infty $. Recall that ${\xi _{{n_l}}} \in \partial \theta \left( {{x_{{n_l}}}} \right)$, we have
\begin{equation}\label{thetaxnl}
\theta \left( {{x_{{n_l}}}} \right) \le \theta \left( x \right) + \left\langle {{\xi _{{n_l}}},{x_{{n_l}}} - x} \right\rangle ,\forall x \in \mathcal{X}.
\end{equation}
Taking the expectation in (\ref{thetaxnl}) with $x$ replaced by $x^\dag$, it follows that
\[\begin{array}{l}
\mathbb{E}\left[ {\theta \left( {{x_{{n_l}}}} \right)} \right] \le \theta \left( {{x^\dag }} \right) + \mathbb{E}\left[ {\left\langle {{\xi _{{n_l}}},{x_{{n_l}}} - {x^\dag }} \right\rangle } \right]\\[1mm]
~~~~~~~~~~~~~ \le \theta \left( {{x^\dag }} \right) + \mathbb{E}\left[ {\left\langle {{\xi _{{n_l}}} - {\xi _{{0}}},{x_{{n_l}}} - {x^\dag }} \right\rangle } \right] + \mathbb{E}\left[ {\left\langle {{\xi _{{0}}},{x_{{n_l}}} - {x^\dag }} \right\rangle } \right].
\end{array}\]
By the lower semi-continuity of $\theta$ and Fatou's lemma, and utilizing (\ref{holder}), (\ref{exnlxnk}) and (\ref{xnlf}), we have with ${p^{ - 1}} + {\left( {{p^*}} \right)^{ - 1}} = 1$ that
\begin{equation}\label{xxingfi}
\begin{array}{l}
\mathbb{E}\left[ {\theta \left( {{x^*}} \right)} \right] \le \mathbb{E}\left[ {\mathop {\lim \inf }\limits_{l \to \infty } \theta \left( {{x_{{n_l}}}} \right)} \right] \le \mathop {\lim \inf }\limits_{l \to \infty } \mathbb{E}\left[ {\theta \left( {{x_{{n_l}}}} \right)} \right]\\[3mm]
~~~~~~~~~~~~\le \theta \left( {{x^\dag }} \right) + c_2\mathbb{E}\left[ {{D_{{\xi _{{0}}}}}\theta \left( { x^\dag,{x_{{0}}}} \right)} \right]  \\[1mm]
~~~~~~~~~~~~~~+ {\left( {\mathbb{E}\left[ {{{\left\| {{\xi _{{0}}}} \right\|}^{p^*}}} \right]} \right)^{{1 \mathord{\left/
 {\vphantom {1 2}} \right.
 \kern-\nulldelimiterspace} p^*}}}{\left( {\mathbb{E}\left[ {{{\left\| {{x^*} - {x^\dag }} \right\|}^p}} \right]} \right)^{{1 \mathord{\left/
 {\vphantom {1 2}} \right.
 \kern-\nulldelimiterspace} p}}} < \infty.
\end{array}
\end{equation}

Finally we turn to show $ \mathbb{E}\left[ {{D_{{\xi _n}}}\theta \left( {{x^*},{x_n}} \right)} \right]\to 0$ as $n \to \infty$.
To this end, we first prove that for $ x^* \in {L^p}\left( {\Omega,\mathcal{X}} \right)$ that is a solution of (\ref{nonlinear equation}) in $\mathscr{D}\left( \theta  \right)$ almost surely,
\begin{equation}\label{exnlxnl}
\mathop {\lim }\limits_{l \to \infty } \mathbb{E}\left[ {\left\langle {{\xi _{{n_l}}},{x_{{n_l}}} - {x^* }} \right\rangle } \right] = 0.
\end{equation}
For any $k<l$, we can write
 \[\mathbb{E}\left[ {\left\langle {{\xi _{{n_l}}},{x_{{n_l}}} - {x^* }} \right\rangle } \right] = \mathbb{E}\left[ {\left\langle {{\xi _{{n_l}}} - {\xi _{{n_k}}},{x_{{n_l}}} - {x^* }} \right\rangle } \right] + \mathbb{E}\left[ {\left\langle {{\xi _{{n_k}}},{x_{{n_l}}} - {x^* }} \right\rangle } \right].\]
We claim that for any fixed $k$, $\mathbb{E}{\left( {\left[ {{{\left\| {{\xi _{{n_k}}}} \right\|}^{{p^*}}}} \right]} \right)^{{1 \mathord{\left/
 {\vphantom {1 {{p^*}}}} \right.
 \kern-\nulldelimiterspace} {{p^*}}}}}$ is bounded. By using the definition of $\xi_{n_k}$ and $t_{n_k}$ in (\ref{sgdex}), the triangle inequality of the norm  ${\left( {\mathbb{E}\left[ {{{\left\| \cdot\right\|}^{p^*}}} \right]} \right)^{{1 \mathord{\left/
 {\vphantom {1 {{p^*}}}} \right.
 \kern-\nulldelimiterspace} {{p^*}}}}}$, and Assumption \ref{5asumption operator} (d), we have
 \[ \hspace{-10mm}
 \begin{array}{l}
{\left( {\mathbb{E}\left[ {{{\left\| {{\xi _{{n_k}}}} \right\|}^{{p^*}}}} \right]} \right)^{{1 \mathord{\left/
 {\vphantom {1 {{p^*}}}} \right.
 \kern-\nulldelimiterspace} {{p^*}}}}}\\[1mm]
 \quad = {\left( {{\mathbb{E}}\left[ {{{\left\| {{\xi _{{n_k} - 1}} - {t_{{n_k} - 1}}{F_{{i_{{n_k} - 1}}}'}{{\left( {{x_{{n_k} - 1}}} \right)}^*}J_r^{{{\cal Y}_{{i_{{n_k} - 1}}}}}\left( {{F_{{i_{{n_k} - 1}}}}\left( {{x_{{n_k} - 1}}} \right) - {y_{{i_{{n_k} - 1}}}}} \right)} \right\|}^{{p^*}}}} \right]} \right)^{{1 \mathord{\left/
 {\vphantom {1 {{p^*}}}} \right.
 \kern-\nulldelimiterspace} {{p^*}}}}}\\[1mm]
 \quad\le {\left( {\mathbb{E}\left[ {{{\left\| {{t_{{n_k} - 1}}{F_{{i_{{n_k} - 1}}}'}{{\left( {{x_{{n_k} - 1}}} \right)}^*}J_r^{{{\cal Y}_{{i_{{n_k} - 1}}}}}\left( {{F_{{i_{{n_k} - 1}}}}\left( {{x_{{n_k} - 1}}} \right) - {y_{{i_{{n_k} - 1}}}}} \right)} \right\|}^{{p^*}}}} \right]} \right)^{{1 \mathord{\left/
 {\vphantom {1 {{p^*}}}} \right.
 \kern-\nulldelimiterspace} {{p^*}}}}}\\
 \quad \quad +{\left( {\mathbb{E}\left[ {{{\left\| {{\xi _{{n_k} - 1}}} \right\|}^{{p^*}}}} \right]} \right)^{{1 \mathord{\left/
 {\vphantom {1 {{p^*}}}} \right.
 \kern-\nulldelimiterspace} {{p^*}}}}}\\[1mm]
 \quad\le {\mu _1}{B_0}{\left( {\mathbb{E}\left[ {{{\left\| {{F_{{i_{{n_k} - 1}}}}\left( {{x_{{n_k} - 1}}} \right) - {y_{{i_{{n_k} - 1}}}}} \right\|}^p}} \right]} \right)^{{1 \mathord{\left/
 {\vphantom {1 {{p^*}}}} \right.
 \kern-\nulldelimiterspace} {{p^*}}}}}+{\left( {\mathbb{E}\left[ {{{\left\| {{\xi _{{n_k} - 1}}} \right\|}^{{p^*}}}} \right]} \right)^{{1 \mathord{\left/
 {\vphantom {1 {{p^*}}}} \right.
 \kern-\nulldelimiterspace} {{p^*}}}}}\\[1.5mm]
 \quad=  {\mu _1}{B_0}{\left( {{R_{{{{n_k} - 1}}}}} \right)^{{1 \mathord{\left/
 {\vphantom {1 {{p^*}}}} \right.
 \kern-\nulldelimiterspace} {{p^*}}}}}+{\left( {\mathbb{E}\left[ {{{\left\| {{\xi _{{n_k} - 1}}} \right\|}^{{p^*}}}} \right]} \right)^{{1 \mathord{\left/
 {\vphantom {1 {{p^*}}}} \right.
 \kern-\nulldelimiterspace} {{p^*}}}}},
\end{array}\]
 where we used the definition of $R_{n_k-1}$ in the last step. Therefore, we can recursively deduce that
 \[{\left( {\mathbb{E}\left[ {{{\left\| {{\xi _{{n_k}}}} \right\|}^{{p^*}}}} \right]} \right)^{{1 \mathord{\left/
 {\vphantom {1 {{p^*}}}} \right.
 \kern-\nulldelimiterspace} {{p^*}}}}} \le {\left( {\mathbb{E}\left[ {{{\left\| {{\xi _0}} \right\|}^{{p^*}}}} \right]} \right)^{{1 \mathord{\left/
 {\vphantom {1 {{p^*}}}} \right.
 \kern-\nulldelimiterspace} {{p^*}}}}} + {\mu _1}{B_0}\sum\nolimits_{n = 0}^{{n_k} - 1} {{{\left( {{R_n}} \right)}^{{1 \mathord{\left/
 {\vphantom {1 {{p^*}}}} \right.
 \kern-\nulldelimiterspace} {{p^*}}}}}}, \]
 which together with (\ref{emxiu1}) implies that for any fixed $k$, the term $\mathbb{E}{\left( {\left[ {{{\left\| {{\xi _{{n_k}}}} \right\|}^{{p^*}}}} \right]} \right)^{{1 \mathord{\left/
 {\vphantom {1 {{p^*}}}} \right.
 \kern-\nulldelimiterspace} {{p^*}}}}}$ is bounded. Consequently, using the H$\ddot{o}$lder inequality (\ref{holder}) and (\ref{xnlcon}), there holds, for any fixed $k$,
\[\mathbb{E}\left[ {\left\langle {{\xi _{{n_k}}},{x_{{n_l}}} - {x^*}} \right\rangle } \right] \le {\left( {\mathbb{E}\left[ {{{\left\| {{\xi _{{n_k}}}} \right\|}^{p^*}}} \right]} \right)^{{1 \mathord{\left/
 {\vphantom {1 2}} \right.
 \kern-\nulldelimiterspace} p^*}}}{\left( {\mathbb{E}\left[ {{{\left\| {{x_{{n_l}}} - {x^*}} \right\|}^p}} \right]} \right)^{{1 \mathord{\left/
 {\vphantom {1 2}} \right.
 \kern-\nulldelimiterspace} p}}} \to 0\]
 as $l \to \infty$. Therefore,
 \[\mathop {\limsup }\limits_{l \to \infty } \left| \mathbb{E}\left[ {\left\langle {{\xi _{{n_l}}},{x_{{n_l}}} - {x^*}} \right\rangle } \right] \right| \le \mathop {\sup }\limits_{l \ge k} \left| \mathbb{E}\left[ {\left\langle {{\xi _{{n_l}}} - {\xi _{{n_k}}},{x_{{n_l}}} - {x^*}} \right\rangle } \right]\right|.\]
 By (\ref{exnlxnk}) and letting $k\to \infty$, we thus obtain (\ref{exnlxnl}).
 Taking (\ref{thetaxnl}) with $x=x^*$ and using (\ref{exnlxnl}), we have
 \[\mathop {\limsup }\limits_{l \to \infty } \mathbb{E}\left[ {\theta \left( {{x_{{n_l}}}} \right)} \right] \le \mathbb{E}\left[ {\theta \left( {{x^*}} \right)} \right] + \mathop {\lim }\limits_{l \to \infty } \mathbb{E}\left[ {\left\langle {{\xi _{{n_l}}},{x_{{n_l}}} - {x^*}} \right\rangle } \right] = \mathbb{E}\left[ {\theta \left( {{x^*}} \right)} \right].\]
 This together with (\ref{xxingfi}) gives
 \begin{equation}\label{ethetacon}
 \mathop {\lim }\limits_{l \to \infty } \mathbb{E}\left[ {\theta \left( {{x_{{n_l}}}} \right)} \right] = \mathbb{E}\left[ {\theta \left( {{x^*}} \right)} \right].
 \end{equation}
 Thus, by using (\ref{exnlxnl}) and (\ref{ethetacon}), we can deduce that
 \[\mathop {\lim }\limits_{l \to \infty } \mathbb{E}\left[ {{D_{{\xi _{{n_l}}}}}\theta \left( {{x^*},{x_{{n_l}}}} \right)} \right] = \mathop {\lim }\limits_{l \to \infty } \left( {\mathbb{E}\left[\theta \left( {{x^*}} \right) - \theta \left( {{x_{{n_l}}}} \right) + \left\langle {{\xi _{{n_l}}},{x_{{n_l}}} - {x^*}} \right\rangle  \right]} \right) = 0,\]
from which, by the monotonicity of $\left\{ \mathbb{E}\left[ {{D_{{\xi _n}}}\theta \left( {x^*,{x_n}} \right)} \right]\right\}$, there follows
 \[\mathop {\lim }\limits_{n \to \infty } \mathbb{E}\left[ {{D_{{\xi _n}}}\theta \left( {{x^*},{x_n}} \right)} \right] = 0.\]
 Since $\theta$ is $p$-convex, we further conclude that $\mathop {\lim }\limits_{n \to \infty } \mathbb{E}\left[ {{{\left\| {{x^*} - {x_n}} \right\|}^p}} \right] = 0$.

Finally, under the condition $\mathcal{N}\left( {F_i'\left( {{x^\dag }} \right)} \right) \subset \mathcal{N}\left( {F_i'\left( x \right)} \right)$ for all $x \in {B_{2\rho }}\left( {{x_0}} \right)$ and $i=1,2,\dots,N$, we can use the similar argument to  \cite[Proposition 3.6]{JinWang2013}  to obtain $x^*=x^\dag$.
\end{proof}
\subsection{Regularization property}
In what follows, we establish the regularization property of SGD-$\theta$ method (\ref{sgd}) under an a {\it priori} stopping rule. Before that, we show the stability property of the method in the following lemma.
\begin{lemma}\label{sstability}
Let Assumptions \ref{5asumption theta} and \ref{5asumption operator} hold. Assume that $\mu_0>0$ and $ \tau>1$ are chosen such that (\ref{c1}) holds.
For any fixed integer $n\ge0$ and any path $\left( {{i_0},
 \cdots ,{i_{n - 1}}} \right) \in \mathcal{F}_n$, let $\left\{ {{\xi _n},{x_n}} \right\}$ and  $\left\{ {\xi _n^\delta ,x_n^\delta } \right\}$ be the sequences generated by SGD-$\theta$ method along the same path in the case of exact and noisy data, respectively. Then
\[\mathbb{E}\left[ {{{\left\| {x_n^\delta  - {x_n}} \right\|}^p}} \right] \to 0~~\text{and}~~\mathbb{E}\left[ {{{\left\| {\xi _n^\delta  - {\xi _n}} \right\|}^{p^*}}} \right] \to 0\]
 as $\delta \to 0$.
\end{lemma}

\begin{proof}
We use an inductive proof on $n$. Since $\xi_0^\delta=\xi_0$ and $x_0^\delta=x_0$, the result is trivial for $n=0$. Assume that the result is true for some $n$, we next prove that it also holds for $n+1$. Since there are only a finite number of sample paths of the form $\left( {{i_0},
 \cdots ,{i_{n - 1}}} \right)$ and the probability of each path is positive, we can derive from the induction hypothesis that
\begin{equation}\label{hyn}
{\left\| {x_n^\delta  - {x_n}} \right\|} \to 0~~\text{and}~~{\left\| {\xi _n^\delta  - {\xi _n}} \right\|} \to 0~~\text{as}~~\delta\to0
\end{equation}
along each sample path $\left( {{i_0}, \cdots ,{i_{n - 1}}} \right)$. We now show that along each path $\left( {{i_0}, \cdots ,{i_{n - 1}},{i_n}} \right)$, there hold
\[{\left\| {x_{n + 1}^\delta  - {x_{n + 1}}} \right\|} \to 0~~\text{and}~~ {\left\| {\xi_{n + 1}^\delta  - {\xi_{n + 1}}} \right\|} \to 0~~\text{as}~~\delta\to0.\]

Next we consider two cases. We first consider the case that ${F_{{i_n}}}\left( {{x_n}} \right) = {y_{{i_n}}}$. For this situation, ${\xi_{n + 1}} = {\xi_n}$ and
\[\xi_{n + 1}^\delta  - {\xi_{n + 1}} = \xi_n^\delta  - {\xi_n} - t_{{n}}^\delta {F_{{i_n}}'}{\left( {x_n^\delta } \right)^*}J_r^{\mathcal{Y}_{i_n}}\left( {{F_{{i_n}}}\left( {x_n^\delta } \right) - y_{{i_n}}^\delta } \right).\]
Since ${t_n^\delta  \le {\mu _1}{{\left\| {{F_{{i_n}}}\left( {x_n^\delta } \right) - y_{{i_n}}^\delta } \right\|}^{p - r}}}$ and through Assumption \ref{5asumption operator} (d), we have
\[{\left\| {\xi _{n + 1}^\delta  - {\xi _{n + 1}}} \right\| \le \left\| {\xi _n^\delta  - {\xi _n}} \right\| + {\mu _1}{B_0}{{\left\| {{F_{{i_n}}}\left( {x_n^\delta } \right) - y_{{i_n}}^\delta } \right\|}^{p - 1}}}.\]
Consequently, by using (\ref{hyn}) and the continuity of $F_{i_n}$, we obtain $\left\| {\xi_{n + 1}^\delta  - {\xi_{n + 1}}} \right\| \to 0$ as $\delta \to 0$. By the continuity of $\nabla {\theta ^*}$ guaranteed by Lemma \ref{lemma2.2} (iii), we further have $x_{n+ 1}^{{\delta }} = \nabla {\theta ^*}\left( {\xi _{n+ 1}^{{\delta }}} \right) \to \nabla {\theta ^*}\left( {{\xi _{n+1}}} \right) = {x_{n + 1}}$ as $\delta \to 0$.

Next we consider the situation that ${F_{{i_n}}}\left( {{x_n}} \right)\ne {y_{{i_n}}}$. By the definition of $\xi_{n+1}^\delta$  and $\xi_{n+1}$, we have
\[\hspace{-18mm}
\begin{array}{l}
\xi _{n + 1}^\delta  - {\xi _{n + 1}} = \xi _n^\delta  - {\xi _n} - t_n^\delta {F_{{i_n}}'}{\left( {x_n^\delta } \right)^*}J_r^{{{\cal Y}_{{i_n}}}}\left( {{F_{{i_n}}}\left( {x_n^\delta } \right) - y_{{i_n}}^\delta } \right)+ {t_n}{F_{{i_n}}'}{\left( {{x_n}} \right)^*}J_r^{{{\cal Y}_{{i_n}}}}\left( {{F_{{i_n}}}\left( {{x_n}} \right) - {y_{{i_n}}}} \right)\\[1mm]
~~~~~~~~~~~~~~~~  = \left( {\xi _n^\delta  - {\xi _n}} \right) + \left( {{t_n} - t_n^\delta } \right){F_{{i_n}}'}{\left( {{x_n}} \right)^*}J_r^{{{\cal Y}_{{i_n}}}}\left( {{F_{{i_n}}}\left( {{x_n}} \right) - {y_{{i_n}}}} \right)\\[1mm]
~~~~~~~~~~~~~~~~~~  + t_n^\delta {F_{{i_n}}'}{\left( {{x_n}} \right)^*}\left( {J_r^{{{\cal Y}_{{i_n}}}}\left( {{F_{{i_n}}}\left( {{x_n}} \right) - {y_{{i_n}}}} \right) - J_r^{{{\cal Y}_{{i_n}}}}\left( {{F_{{i_n}}}\left( {x_n^\delta } \right) - y_{{i_n}}^\delta } \right)} \right)\\[1mm]
~~~~~~~~~~~~~~~~~~  + t_n^\delta \left( {{F_{{i_n}}'}{{\left( {{x_n}} \right)}^*} - {F_{{i_n}}'}{{\left( {x_n^\delta } \right)}^*}} \right)J_r^{{{\cal Y}_{{i_n}}}}\left( {{F_{{i_n}}}\left( {x_n^\delta } \right) - y_{{i_n}}^\delta } \right).
\end{array}\]
By using Assumption \ref{5asumption operator} (d) and the property of $J_r^{{{\cal Y}_{{i_n}}}}$, we further obtain
\begin{equation}\label{stabilityxi}
\hspace{-10mm}
\begin{array}{l}
\left\| {\xi _{n + 1}^\delta  - {\xi _{n + 1}}} \right\| \le \left\| {\xi _n^\delta  - {\xi _n}} \right\| + \left| {{t_n} - t_n^\delta } \right|\left\| {{F_{{i_n}}'}{{\left( {{x_n}} \right)}^*}J_r^{{{\cal Y}_{{i_n}}}}\left( {{F_{{i_n}}}\left( {{x_n}} \right) - {y_{{i_n}}}} \right)} \right\|\\[1mm]
~~~~~~~~~~~~~~~~~~~~~+ t_n^\delta \left\| {{F_{{i_n}}'}{{\left( {{x_n}} \right)}^*}} \right\|\left\| {J_r^{{{\cal Y}_{{i_n}}}}\left( {{F_{{i_n}}}\left( {{x_n}} \right) - {y_{{i_n}}}} \right) - J_r^{{{\cal Y}_{{i_n}}}}\left( {{F_{{i_n}}}\left( {x_n^\delta } \right) - y_{{i_n}}^\delta } \right)} \right\|\\[1.5mm]
~~~~~~~~~~~~~~~~~~~~~ + t_n^\delta \left\| {{F_{{i_n}}'}{{\left( {{x_n}} \right)}^*} - {F_{{i_n}}'}{{\left( {x_n^\delta } \right)}^*}} \right\|\left\| {J_r^{{{\cal Y}_{{i_n}}}}\left( {{F_{{i_n}}}\left( {x_n^\delta } \right) - y_{{i_n}}^\delta } \right)} \right\|\\[1mm]
~~~~~~~~~~~~~~~~~~~ \le \left\| {\xi _n^\delta  - {\xi _n}} \right\| + {B_0}\left| {{t_n} - t_n^\delta } \right|{\left\| {{F_{{i_n}}}\left( {{x_n}} \right) - {y_{{i_n}}}} \right\|^{r - 1}}\\[1mm]
~~~~~~~~~~~~~~~~~~~~~~  + {B_0}t_n^\delta\left\| {J_r^{{{\cal Y}_{{i_n}}}}\left( {{F_{{i_n}}}\left( {{x_n}} \right) - {y_{{i_n}}}} \right) - J_r^{{{\cal Y}_{{i_n}}}}\left( {{F_{{i_n}}}\left( {x_n^\delta } \right) - y_{{i_n}}^\delta } \right)} \right\|\\[1mm]
~~~~~~~~~~~~~ ~~~~~~~~~ + t_n^\delta \left\| {{F_{{i_n}}'}{{\left( {{x_n}} \right)}^*} - {F_{{i_n}}'}{{\left( {x_n^\delta } \right)}^*}} \right\|{\left\| {{F_{{i_n}}}\left( {x_n^\delta } \right) - y_{{i_n}}^\delta } \right\|^{r - 1}}.
\end{array}
\end{equation}
We next prove that
\begin{equation}\label{tsta}
t_{{n}}^\delta  \to {t_{{n}}}~\text{as}~\delta \to 0.
\end{equation}
Recall  that
$${t_{{n}}} = \min \left\{ {\frac{{{\mu _0}{{\left\| {{F_{{i_n}}}\left( {{x_n}} \right) - {y_{{i_n}}}} \right\|}^{p(r-1)}}}}{{{{\left\| {{F_{{i_n}}'}{{\left( {{x_n}} \right)}^*}J_r^{\mathcal{Y}_{i_n}}\left( {{F_{{i_n}}}\left( {{x_n}} \right) - {y_{{i_n}}}} \right)} \right\|}^p}}},{\mu _1}} \right\}{\left\| {{{F_{{i_n}}}\left( {x_n } \right) - y_{{i_n}} } } \right\|^{p - r}}$$
and, by (\ref{hyn}) and the continuity of $F_{i_n}$, we have ${\left\| {{F_{{i_n}}}\left( {x_n^\delta } \right) - y_{{i_n}}^\delta } \right\| > \tau {\delta _{{i_n}}}}$ for small $\delta>0$, and thus
$$t_{{n}}^\delta  = \min \left\{ {\frac{{{\mu _0}{{\left\| {{F_{{i_n}}}\left( {x_n^\delta } \right) - y_{{i_n}}^\delta } \right\|}^{p(r-1)}}}}{{{{\left\| {{{F}_{{i_n}}'}{{\left( {x_n^\delta } \right)}^*}J_r^{\mathcal{Y}_{i_n}}\left( {{F_{{i_n}}}\left( {x_n^\delta } \right) - y_{{i_n}}^\delta } \right)} \right\|}^p}}},{\mu _1}} \right\}{\left\| {{{F_{{i_n}}}\left( {x_n^\delta } \right) - y_{{i_n}}^\delta } } \right\|^{p - r}}.$$
If ${F_{{i_n}}'}{{\left( {{x_n}} \right)}^*}J_r^{\mathcal{Y}_{i_n}}\left( {{F_{{i_n}}}\left( {{x_n}} \right) - {y_{{i_n}}}} \right)=0$, then we have ${t_{{n}}}=\mu_1{\left\| {{{F_{{i_n}}}\left( {x_n} \right) - y_{{i_n}} } } \right\|^{p - r}}$; from Assumption \ref{5asumption operator} (d) and the uniform smoothness of $\mathcal{Y}_{i_n}$, it is known that $F_{i_n}$, $F_{i_n}'$ and $J_r^{\mathcal{Y}_{i_n}}$ are continuous, and thus $t_{{n}}^\delta=\mu_1{\left\| {{{F_{{i_n}}}\left( {x_n^\delta } \right) - y_{{i_n}}^\delta } } \right\|^{p - r}}$ for sufficiently small $\delta$, which together with the continuity of $F_{i_n}$ gives $t_{{n}}^\delta  \to {t_{{n}}}$ as $\delta \to 0$. On the other hand, when ${F_{{i_n}}'}{{\left( {{x_n}} \right)}^*}J_r^{\mathcal{Y}_{i_n}}\left( {{F_{{i_n}}}\left( {{x_n}} \right) - {y_{{i_n}}}} \right)\neq0$, it follows from the induction hypothesis that $t_{{n}}^\delta  \to {t_{{n}}}$ as $\delta \to 0$.

Further, by utilizing the induction hypothesis (\ref{hyn}), (\ref{tsta}) and the continuity of $F_{i_n}$, $F_{i_n}'$, $J_r^{\mathcal{Y}_{i_n}}$, we can deduce from (\ref{stabilityxi}) that $\left\| {\xi_{n + 1}^\delta  - {\xi_{n + 1}}} \right\| \to 0$ as $\delta \to 0$. By the continuity of $\nabla {\theta ^*}$,  there follows $\left\| {x_{n + 1}^\delta  - {x_{n + 1}}} \right\| \to 0$ as $\delta \to 0$.

Finally, since there are only a finite number of sample paths of the form $\left( {{i_0}, \cdots ,{i_{n - 1}},{i_n}} \right)$, we can conclude that $\mathbb{E}\left[ {{{\left\| {x_{n + 1}^\delta  - {x_{n + 1}}} \right\|}^p}} \right] \to 0$ and $~\mathbb{E}\left[ {{{\left\| {\xi _{n+1}^\delta  - {\xi _{n+1}}} \right\|}^{p^*}}} \right] \to 0$ as $\delta  \to 0$. The inductive proof is therefore completed.
\end{proof}

Now we turn to show the regularization property of the SGD-$\theta$ method (\ref{sgd}) under an {\it a priori} stopping rule.
\begin{theorem}\label{6lemma}
Let Assumptions \ref{5asumption theta} and \ref{5asumption operator} hold. Assume that $\mu_0>0$ and $ \tau>1$ are chosen such that (\ref{c1}) holds. Assume that the stopping index $n_{\delta}$ satisfies $n_{\delta} \to \infty$ as $\delta \to 0$. Then, there exists a solution $x^* \in {B_{2\rho} }\left( {{x_0}} \right)\cap \mathscr{D}\left( \theta  \right)$ of (\ref{nonlinear equation}) such that
\[\mathop {\lim }\limits_{\delta  \to 0} \mathbb{E}\left[ {{{\left\| {x_{{n_\delta }}^\delta  - {x^*}} \right\|}^p}} \right] = 0 ~~\text{and}~~\mathop {\lim }\limits_{\delta  \to 0} \mathbb{E}\left[ {{D_{\xi _{{n_\delta }}^\delta }}\theta \left( {{x^*},x_{{n_\delta }}^\delta } \right)} \right] = 0.\]
Further, if $\mathcal{N}\left( {F_i'\left( {{x^\dag }} \right)} \right) \subset \mathcal{N}\left( {F_i'\left( x \right)} \right)$ for all $x \in {B_{2\rho }}\left( {{x_0}} \right)$ and $i=1,2,\dots,N$, then $x^*=x^\dag$.
\end{theorem}
\begin{proof}
Let $x_*$ be the solution of (\ref{nonlinear equation}) determined in Theorem \ref{sconvergence}. By the $p$-convexity of $\theta$, it suffices to show $\mathbb{E}\left[ {{D_{\xi _{{n_\delta }}^\delta }}\theta \left( {{x^*},x_{{n_\delta }}^\delta } \right)} \right] \to 0$ as $\delta \to 0$. Since $n_{\delta} \to \infty$ as  $\delta \to 0$, thus for any  fixed integer $n>0$, we have ${n_{\delta}} > n$ for sufficiently small $\delta$. We may use Lemma \ref{smonotoniciy} to obtain
\begin{equation}\label{remo}
\mathbb{E}\left[ {{D_{\xi _{{n_\delta }}^\delta }}\theta \left( {{x^*},x_{{n_\delta }}^\delta } \right)} \right] \le \mathbb{E}\left[ {{D_{\xi _n^\delta }}\theta \left( {{x^*},x_n^\delta } \right)} \right],
\end{equation}
from which there follows
\[\begin{array}{l}
\mathop {\lim \sup }\limits_{\delta  \to 0} \mathbb{E}\left[ {{D_{\xi _{{n_\delta }}^\delta }}\theta \left( {{x^*},x_{{n_\delta }}^\delta } \right)} \right] \le \mathop {\lim \sup }\limits_{\delta  \to 0} \mathbb{E}\left[ {{D_{\xi _n^\delta }}\theta \left( {{x^*},x_n^\delta } \right)} \right]\\[1mm]
~~~~~~~~~~~~~~~~~~~~~~~~~~~~~~~ \le \mathop {\lim \sup }\limits_{\delta  \to 0} \left( {\theta \left( {{x^*}} \right) - \mathbb{E}\left[ {\theta \left( {x_n^\delta } \right)} \right] - \mathbb{E}\left[ {\left\langle {\xi _n^\delta ,{x^*} - x_n^\delta } \right\rangle } \right]} \right)
\end{array}\]
for any $n\ge 0$. Using H$\ddot{o}$lder inequality (\ref{holder}), we get that
\[\begin{array}{l}
\quad \left| {\mathbb{E}\left[ {\left\langle {\xi _n^\delta ,{x^*} - x_n^\delta } \right\rangle } \right] - \mathbb{E}\left[ {\left\langle {{\xi _n},{x^*} - {x_n}} \right\rangle } \right]} \right|\\[1.5mm]
 \le \left| {\mathbb{E}\left[ {\left\langle {\xi _n^\delta  - {\xi _n},{x^*} - x_n^\delta } \right\rangle } \right]} \right| + \left| {\mathbb{E}\left[ {\left\langle {{\xi _n},{x_n} - x_n^\delta } \right\rangle } \right]} \right|\\[1.5mm]
 \le {\left( {\mathbb{E}\left[ {{{\left\| {\xi _n^\delta  - {\xi _n}} \right\|}^{p^*}}} \right]} \right)^{\frac{1}{p^*}}}{\left( {\mathbb{E}\left[ {{{\left\| {{x^*} - x_n^\delta } \right\|}^{p}}} \right]} \right)^{\frac{1}{{p}}}} + {\left( {\mathbb{E}\left[ {{{\left\| {{\xi _n}} \right\|}^{p^*}}} \right]} \right)^{\frac{1}{{p^*}}}}{\left( {\mathbb{E}\left[ {{{\left\| {{x_n} - x_n^\delta } \right\|}^p}} \right]} \right)^{\frac{1}{p}}},
\end{array}\]
which together with Lemma \ref{sstability} implies that
\[\mathop {\lim }\limits_{\delta  \to 0} \mathbb{E}\left[ {\left\langle {\xi _n^\delta ,{x^*} - x_n^\delta } \right\rangle } \right] = \mathbb{E}\left[ {\left\langle {{\xi _n},{x^*} - {x_n}} \right\rangle } \right].\]
Since Lemma \ref{sstability} shows that ${\left\| {x_n^\delta  - {x_n}} \right\|} \to 0~\text{as}~\delta\to0$  along each sample path $\left( {{i_0},\cdots ,{i_{n - 1}}} \right)$, using the lower semi-continuity of $\theta$ and Fatou's lemma, we arrive at
\[\mathbb{E}\left[ {\theta \left( {{x_n}} \right)} \right] \le \mathbb{E}\left[ {\mathop {\lim \inf }\limits_{\delta  \to 0} \theta \left( {x_n^\delta } \right)} \right] \le \mathop {\lim \inf }\limits_{\delta  \to 0} \mathbb{E}\left[ {\theta \left( {x_n^\delta } \right)} \right].\]
Consequently, we can deduce that
\[\begin{array}{l}
\mathop {\lim \sup }\limits_{\delta  \to 0} \mathbb{E}\left[ {{D_{\xi _{{n_\delta }}^\delta }}\theta \left( {{x^*},x_{{n_\delta }}^\delta } \right)} \right] \le \theta \left( {{x^*}} \right) - \mathop {\liminf}\limits_{\delta  \to 0} \mathbb{E}\left[ {\theta \left( {x_n^\delta } \right)} \right] - \mathop {\lim }\limits_{\delta  \to 0} \mathbb{E}\left[ {\left\langle {\xi _n^\delta ,{x^*} - x_n^\delta } \right\rangle } \right]\\[1mm]
 ~~~~~~~~~~~~~~~~~~~~~~~~~~~~~~~~~~~~\le \theta \left( {{x^*}} \right) - \mathbb{E}\left[ {\theta \left( {{x_n}} \right)} \right] - \mathbb{E}\left[ {\left\langle {{\xi _n},{x^*} - {x_n}} \right\rangle } \right]\\[1mm]
 ~~~~~~~~~~~~~~~~~~~~~~~~~~~~~~~~~~~~ = \mathbb{E}\left[ {{D_{{\xi _n}}}\theta \left( {{x^*},{x_n}} \right)} \right].
\end{array}\]
Letting  $n \to \infty$, we derive from Theorem \ref{sconvergence} that $\mathbb{E}\left[ {{D_{\xi _{{n_\delta }}^\delta }}\theta \left( {{x^*},x_{{n_\delta }}^\delta } \right)} \right] \to 0$ as $\delta \to 0$.   If $\mathcal{N}\left( {F_i'\left( {{x^\dag }} \right)} \right) \subset \mathcal{N}\left( {F_i'\left( x \right)} \right)$ for all $x \in {B_{2\rho }}\left( {{x_0}} \right)$ and $i=1,2,\dots,N$, one may follow Theorem \ref{sconvergence} to show that $x^*=x^\dag$. The proof is thus completed.
\end{proof}

\begin{remark}
For adaptive step sizes $\left\{ {t_n^\delta } \right\}_{n = 1}^\infty $ satisfying $\sum\nolimits_{n = 1}^\infty  {t_n^\delta }  = \infty $ and $\sum\nolimits_{n = 1}^\infty  {{{\left( {t_n^\delta } \right)}^2}}  < \infty$, the SGD method has been shown in \cite{JinB_2020} to be a regularization method in Hilbert spaces when the stopping index $n_\delta$ is chosen such that
 \[{\lim _{\delta  \to 0}}{n_\delta } = \infty~~ \text{and}~~ {\lim _{\delta  \to 0}}{\delta ^2}\sum\limits_{n = 1}^{{n_\delta }} {t_n^\delta }=0.\]
Similar results can be founded in \cite{JinBker2023} for linear inverse problems in Banach spaces. In particular, for polynomially decaying step sizes $t_n^\delta  = {t_0}{n^{ - \alpha }},$ the conditions
 $ \frac{1}{2} < \alpha  < 1$ and ${t_0}{\max _i}{\sup _{x\in B_{2\rho}(x_0)}}\left\| {F_i'(x)} \right\|^2 \le 1$ gives a valid step-size choice, and the stopping index $n_\delta$ should satisfy ${\lim _{\delta  \to 0}}n_\delta=\infty$ and ${\lim _{\delta  \to 0}} n_\delta{\delta ^{\frac{2}{{1 - \alpha }}}}=0 $. For the constant step sizes, the correspondence between the noise level and the step size reduces to  ${\lim _{\delta  \to 0}}\delta^2 n_\delta=0$. The proof of Theorem \ref{6lemma} only requires the stopping index $n_{\delta}$ satisfying  $n_{\delta} \to \infty$ as $\delta \to 0$, because the discrepancy principle is incorporated into the step size (\ref{3choicestep}). Such condition on $n_\delta$ is analogous to that for deterministic regularization methods \cite{Hanke1995A,JinWang2013}.
\end{remark}

\begin{remark}
In Theorem \ref{6lemma} we have shown the convergence result of SGD-$\theta$ method under an {\it a priori} stopping rule. The issue of the {\it a posteriori} stopping rule for stochastic iterative methods is desirable and challenging, even for Hilbert spaces \cite{Jahn-Jin-2020,Jin-Chen-2024,Jin-Liu-2024}. As a preliminary attempt, we design  an {\it a posteriori} stopping rule for SGD-$\theta$ method and give its finite-iteration termination property  in \ref{appendix_sgd}.
\end{remark}
\begin{remark}
Let $I_n= \left\{ {{i_1},\cdots ,{i_{N_b}}} \right\}$ be the subset randomly selected from $ \left\{ {1,2, \cdots ,N} \right\}$ with batch size $\left| {{I_n}} \right| = {N_b}$; let ${{\cal Y}_{{I_n}}} := {{\cal Y}_{{i_1}}} \times  \cdots  \times {{\cal Y}_{{i_{{N_b}}}}}$, $y_{{I_n}}^\delta : = \left( {y_{{i_1}}^\delta , \ldots ,y_{{i_{{N_b}}}}^\delta } \right)$ and define ${F_{{I_n}}}:\mathcal{X} \to {{\cal Y}_{{I_n}}}$ by ${F_{{I_n}}} = \left( {{F_{{i_1}}}, \ldots ,{F_{{i_{N_b}}}}} \right)$, then the mini-batch version of SGD-$\theta$ method can be formulated as
\begin{equation}\label{minibatch_SGD}
\begin{array}{l}
\xi _{n + 1}^\delta  = \xi _n^\delta  - t_n^\delta {F_{{I_n}}'}{\left( {x_n^\delta } \right)^*}J_r^{{{\cal Y}_{{I_n}}}}\left( {{F_{{I_n}}}\left( {x_n^\delta } \right) - y_{{I_n}}^\delta } \right),\\[1mm]
x_{n + 1}^\delta  = \arg \mathop {\min }\limits_{x \in {\cal X}} \left\{ {\theta \left( x \right) - \left\langle {\xi _{n + 1}^\delta ,x} \right\rangle } \right\},
\end{array}
\end{equation}
where $J_r^{\mathcal{Y}_{I_n}}$ denotes the duality mapping of $\mathcal{Y}_{I_n}$ with $1<r<\infty$ and $t_n^\delta$ is determined by (\ref{3choicestep}) with suitable modification. Following the arguments above, the convergence of the mini-batch version of SGD-$\theta$ method can be easily established. In the next section we will also use mini-batch SGD-$\theta$ method to do the numerics.
\end{remark}

\section{Numerical experiments}\label{3numbericalex}
In this section, we provide some numerical simulations on the SGD-$\theta$ method (\ref{sgd}) with the step size chosen by (\ref{3choicestep}). To show the performance of SGD-$\theta$ method, we compare the computational results with the ones obtained by \\[1mm]
\noindent (\textbf{SGD-Decaying}) SGD method with polynomially decaying step size \cite{JinB_2019,JinB_2020,JinBker2023}, i.e., the iteration (\ref{sgd}) with the step size $t_n^{\delta}$ determined by
\[t_n^\delta  = {t_0}{n^{ - \alpha }},\]
where $\alpha  \in \left( {\frac{1}{2} ,1} \right)$ and $t_0>0$. It is suggested in \cite{JinB_2019,JinB_2020} that the conditions
 $ \frac{1}{2} < \alpha  < 1$ and ${t_0}{\max _i}{\sup _{x\in B_{2\rho}(x_0)}}\left\| {F_i'(x)} \right\|^2 \le 1$ yields a valid step-size schedule and the smaller exponent $\alpha$ is desirable for convergence; in the numerical simulations, we take $\alpha=0.51$ and the parameter $t_0$ will be specified for each experiment.\\[1mm]
\noindent (\textbf{SGD-NDP}) SGD method with adaptive step size considered in \cite{JinQ_2023}, for which the discrepancy principle is not incorporated; that is, the iteration (\ref{sgd}) using the step size
(\ref{3choicestep}) with $\tau=0$, given by
\[t_n^\delta  = \min \left\{ {\frac{{{\mu _0}{{\left\| {{F_{{i_n}}}\left( {x_n^\delta } \right) - y_{{i_n}}^\delta } \right\|}^{p(r - 1)}}}}{{{{\left\| {{F_{{i_n}}'}{{\left( {x_n^\delta } \right)}^*}J_r^{{\mathcal{Y}_{{i_n}}}}\left( {{F_{{i_n}}}\left( {x_n^\delta } \right) - y_{{i_n}}^\delta } \right)} \right\|}^p}}},{\mu _1}} \right\}{\left\| {{F_{{i_n}}}\left( {x_n^\delta } \right) - y_{{i_n}}^\delta } \right\|^{p - r}}\]
for some positive constants $\mu_0$ and $\mu_1$.

For fair comparison, the constants $\mu_0$ and $\mu_1$ of SGD-NDP take the same values as in (\ref{3choicestep}). In view of Remark \ref{stepsize}, the choice of $\mu_0$ needs to satisfy (\ref{c1}) and the parameter $\mu_1>0$ can be selected to be any positive number, which will be specified for each experiment.
All the computational results shown below are derived from a single stochastic run, as is usually done in practice. During the implementation of (\ref{sgd}), a key ingredient is to solve the minimization problem
\begin{equation}\label{minimization problem}
x = \arg \mathop {\min }\limits_{z \in \mathcal{X}} \left\{ {\theta \left( z \right) - \left\langle {\xi ,z} \right\rangle } \right\}
\end{equation}
for any $\xi  \in {\mathcal{X}^*}$. For many important choices of $\theta$, (\ref{minimization problem}) can be solved easily.
In case  $\mathcal{X}={L^2}\left( \Omega  \right)$ and the sought solution satisfies the constraint $x\in \mathcal{C}$ with $\mathcal{C} \subset \mathcal{X}$ being a closed convex set, we will choose
\begin{equation}\label{thetal1}
\theta \left( x \right) = \frac{1}{{2}}\left\| x \right\|_{{L^2}\left( \Omega  \right)}^2 + {\mathbb{I}_{\mathcal{C}}}\left( x \right),
\end{equation}
where $\mathbb{I}_{\mathcal{C}}$ is the indicator function of $\mathcal{C}$. It is clear that this function $\theta$ satisfies Assumption \ref{5asumption theta} with $p=2$ and $\sigma=\frac{1}{{2 }}$. In particular, when $\mathcal{C} = \left\{ {x \in \mathcal{X}: x \ge 0~\text{a.e. on}~\Omega } \right\}$, the minimizer of  (\ref{minimization problem}) can be explicitly given by
\[\begin{array}{l}
x = \max \left\{ {\xi ,0} \right\}.
\end{array}\]
In case $\mathcal{X}={L^2}\left( \Omega  \right)$ and the sought solution is piecewise constant, we will take
\begin{equation}\label{thetatv}
\theta \left( x \right) = \frac{1}{{2\beta }}\left\| x \right\|_{{L^2}\left( \Omega  \right)}^2 + {\left| x \right|_{TV}},
\end{equation}
where ${\left| x \right|_{TV}}$ is the total variation of $x$ and the parameter $\beta>0$ reflects the role of the term ${\left| x \right|_{TV}}$.  The function $\theta$ in (\ref{thetatv})  satisfies Assumption \ref{5asumption theta} with $p=2$ and $\sigma=(2\beta)^{-1}$, refer to \cite{Bo2011Iterative}; the minimization problem (\ref{minimization problem}) becomes the total variation denoising problem \cite{Rudin1992Nonlinear}
\begin{equation}\label{TVSOLVE}
\begin{array}{l}
x = \arg \mathop {\min }\limits_{z \in {{L^2}\left( \Omega  \right)}} \left\{ {\frac{1}{{2\beta }}\left\| {z - \beta \xi } \right\|_{{L^2}\left( \Omega  \right)}^2 + {\left| z \right|_{TV}}} \right\}.
\end{array}
\end{equation}
Though there is no explicit form of the minimizer of (\ref{TVSOLVE}), it can be solved efficiently by the primal dual hybrid gradient (PDHG) method \cite{Zhu2008An}. During the computation, the minimization problem (\ref{TVSOLVE}) is solved by PDHG method, which is terminated if the relative duality gap is small than $10^{-3}$ or the number of iterations exceeds 200.
\subsection{Computed Tomography}
We first consider the computed tomography (CT) which concerns determining the density of cross sections of an object by measuring the attenuation of X-rays as they propagate through the object. Mathematically, one requires to recover a function supported on a bounded domain from its Radon transform \cite{Natterer2001}.
In our experiments, we assume that the sought image is supported on a square domain in $\mathbb{R}^2$, the size is $256\times256$ and it can be represented by a vector $x\in \mathbb{R}^Q$ with $Q=256\times256$. We consider the 2D parallel beam geometry, with 90 projection angles equally distributed between 1 and 180 degrees, 256 lines per projection. Let $y$ be the measurement data of attenuation along the rays which can be represented by a vector in $\mathbb{R}^M$ with $M=90\times256$. We use the function \texttt{paralleltomo} in the Matlab package AIR TOOLS \cite{Hansen2012} to generate the discrete problem. This leads to solve a linear ill-conditioned problem  $y=Fx$,
where $F$ is a sparse matrix with size $M\times Q$. Here, $M=23040$ and $Q=65536$. Let $F_i$ be $i$th row of $F$ and $y_i$ be the corresponding element of $y$.  This gives a linear system of the form (\ref{nonlinear equation}), that is,
\[{F_i}x = {y_i},i = 1,2, \ldots ,N\]
with problem size $N=23040$, where $\mathcal{X}= \left( {{\mathbb{R}^Q},{{\left\|  \cdot  \right\|}_2}} \right)$ and $\mathcal{Y}_i=\left( {\mathbb{R},{{\left\|  \cdot  \right\|}_r}} \right) (1<r<\infty)$ are Banach spaces. For executing SGD-$\theta$ method, we use mini batches with batch size $N_b$.
The batch size $N_b$ is taken to be $N_b=160$, unless otherwise stated. It is easy to check that the tangential cone condition (\ref{TCC}) is satisfied with $\eta =0$; thus,
Assumption \ref{5asumption operator}(d) holds.
In our simulations, the sought image is taken to be the Shepp-Logan phantom discretized on 256$\times$256 pixel grid and the pixel values vary in the interval $\left[ {0,1} \right]$; see figure \ref{ctexact}(c). Let $x^{\dag}$ be the vector generated by stacking all the columns of the sought image. Let $y_i={F_i}x^{\dag}$, $i=1,2,\ldots,N$, be the exact data; see figure \ref{ctexact}(a). In order to capture the nonnegativity of the sought solution, we take the function $\theta$ to be the form (\ref{thetal1}).
When implementing the SGD-$\theta$ method (\ref{sgd})-(\ref{3choicestep}), we take $\tau=1.1$. According to (\ref{c1}), we require $\mu_0< 2\left( {1 - \frac{1}{\tau }} \right)$. Thus, we take $\mu_0=0.18$ and $\mu_1=10000$. For performing SGD-Decaying method, we pick $t_0=0.01$. We take the initial guess $\xi_0=x_0=0$.
\begin{figure}[htbp]
\centering
\vspace{-0.35cm} 
\subfigtopskip=2pt 
\subfigbottomskip=6pt 
\subfigcapskip=-5pt 
\subfigure[$\delta_{rel}=0.5$]{
 \includegraphics[width=5cm,height=4cm]{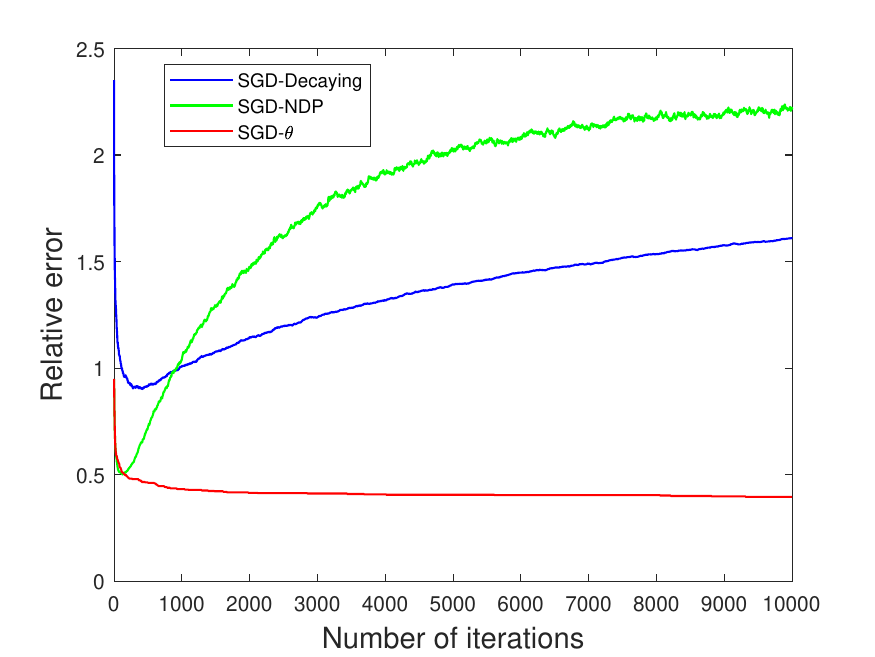}}
\hspace{-1em}
\subfigure[$\delta_{rel}=0.1$]{
\includegraphics[width=5cm,height=4cm]{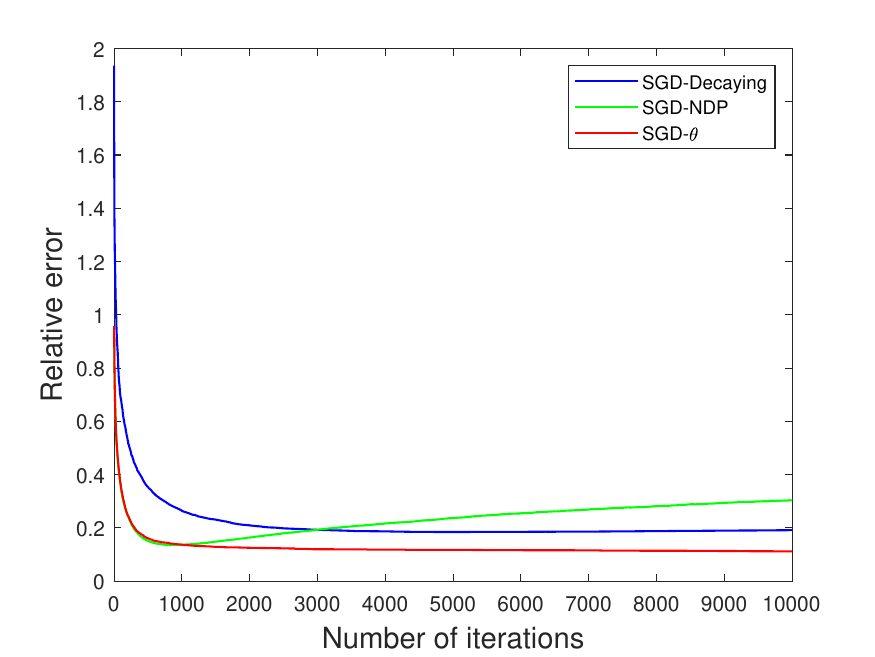}}
\hspace{-1em}
\subfigure[$\delta_{rel}=0.01$]{
\includegraphics[width=5cm,height=4cm]{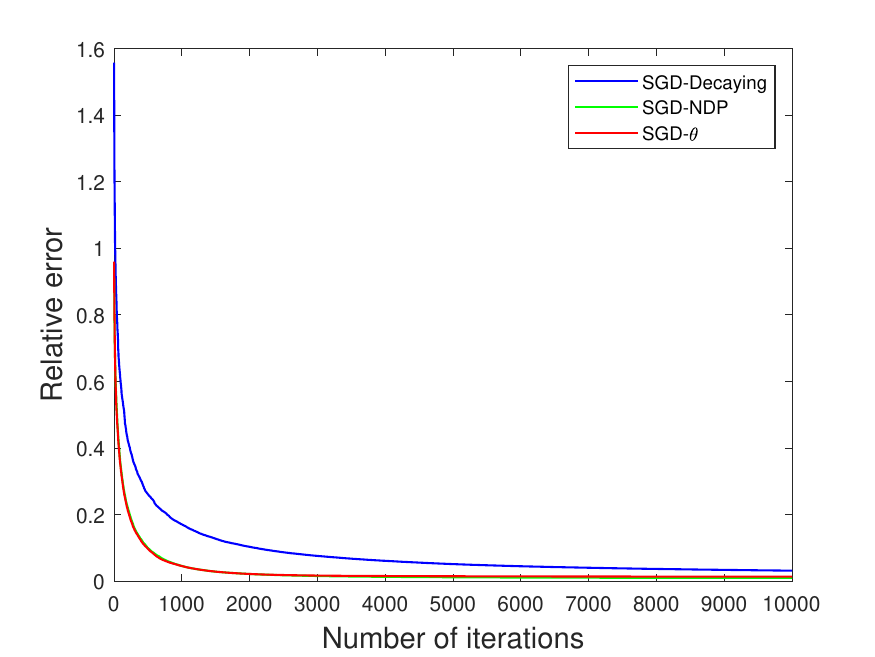}}
\captionsetup{font={small}}
\vspace*{-5pt}
\caption{Computational results for CT under Gaussian noise.}
\label{ctgaussian}
\end{figure}

We first investigate the case that the data is corrupted by Gaussian noise and validate the effectiveness of SGD-$\theta$ method. The noisy data $y_i^{\delta}$ is generated by
\[
y_i^\delta  = {y_i} + {\delta _{rel}}{\left\| {{y_i}} \right\|_{{2}}}{\varepsilon _i},~~i=1,2,\ldots,N,\]
where ${\delta _{rel}}$ is the relative noise level and ${\varepsilon _i}$s follow the standard Gaussian distribution. Then instead of $y_i$ only noisy data $y_i^{\delta}$ are available with the noise levels ${\delta _i} = {\delta _{rel}}{\left\| {{y_i}} \right\|_{{2}}},i=1,2,\ldots,N$; we will use $\left\{ {y_i^\delta } \right\}_{i = 1}^N$ to reconstruct $x^{\dag}$. To remove the Gaussian noise efficiently, we choose $r=2$ (i.e., ${\mathcal{Y}_i} = \left( {\mathbb{R},{{\left\|  \cdot  \right\|}_2}} \right)$) and execute the SGD-$\theta$ method; as a comparison, we also perform SGD-Decaying and SGD-NDP methods. In figure \ref{ctgaussian} we display the evolution of relative errors ${{{{{\left\| {x_n^\delta  - {x^\dag }} \right\|}_{{2}}}} \mathord{\left/
 {\vphantom {{{{\left\| {x_n^\delta  - {x^\dag }} \right\|}_{{2}}}} {\left\| {{x^\dag }} \right\|}}} \right.
 \kern-\nulldelimiterspace} {\left\| {{x^\dag }} \right\|}}_{{2}}}$ by SGD-$\theta$, SGD-Decaying and SGD-NDP methods, under relative noise levels ${\delta _{rel}}=0.5,0.1$ and $0.01$, for 10000 iterations. As can be seen from figure \ref{ctgaussian}, the iterative errors of SGD-Decaying and SGD-NDP methods may suffer from oscillations that are especially severe for larger noise levels; moreover, SGD-Decaying and SGD-NDP methods exhibit semi-convergence behavior, that is, the iterates first converges to the sought solution, and then begin to diverge as the iteration proceeds. We also observe that, since SGD-NDP method allows using larger step size than SGD-Decaying, it converges faster in the initial stage, and diverges more quickly from the sought solution. However, our SGD-$\theta$ method could remarkably suppress the oscillations and reduce the semi-convergence by incorporating the spirit of the discrepancy principle into the step size (\ref{3choicestep}), which clearly demonstrates the effectiveness of SGD-$\theta$ method.
  \begin{figure}[ht]
\centering
\vspace{-0.35cm} 
\subfigtopskip=2pt 
\subfigbottomskip=6pt 
\subfigcapskip=-5pt 
\subfigure[Exact data]{
\includegraphics[scale=0.29]{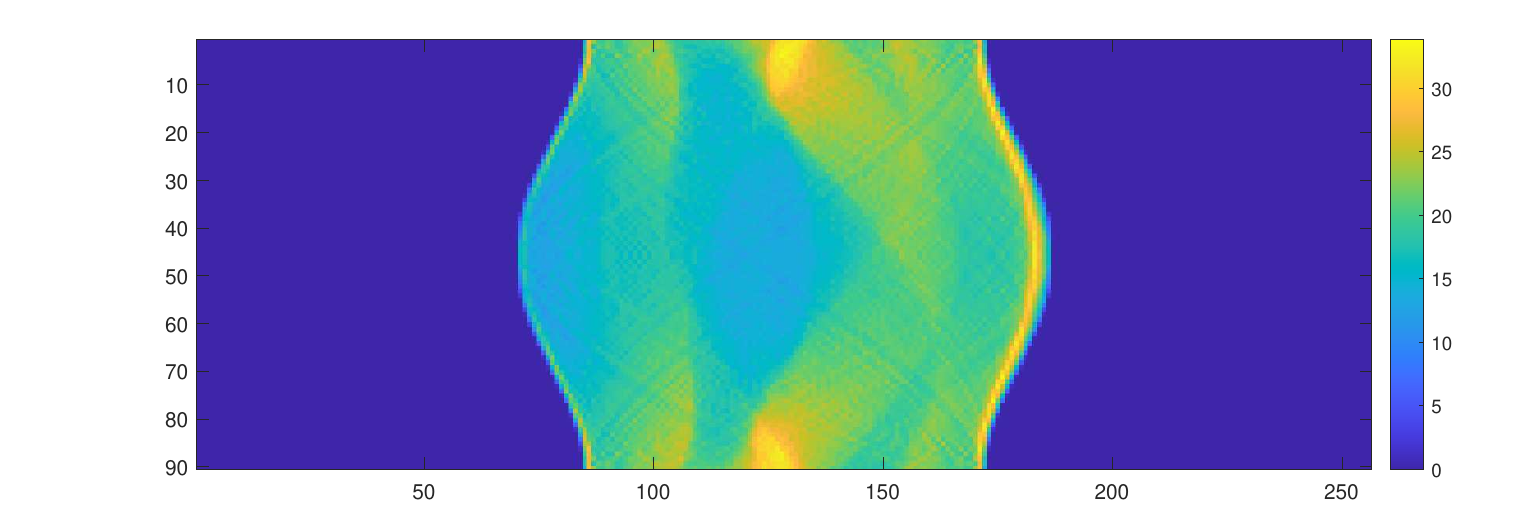}}
\hspace{-1em}
\subfigure[Noisy data ($\delta_{rel}=0.01$)]{
 \includegraphics[scale=0.29]{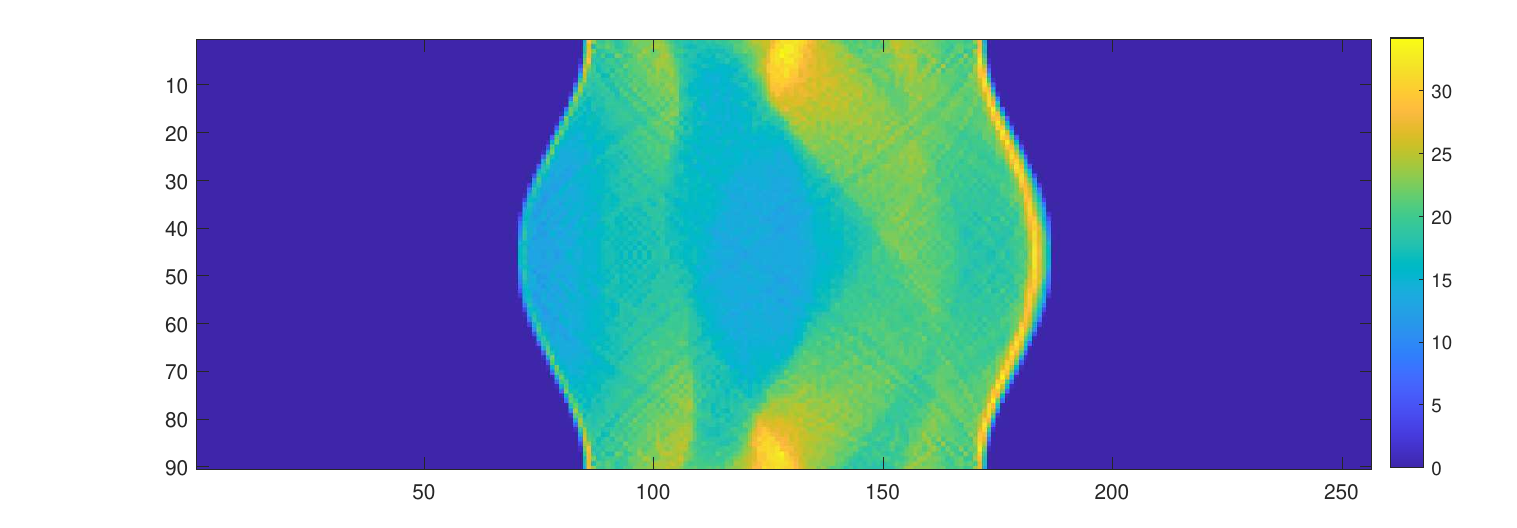}}\\[-2.5mm]
 \subfigure[True solution]{
\includegraphics[scale=0.5]{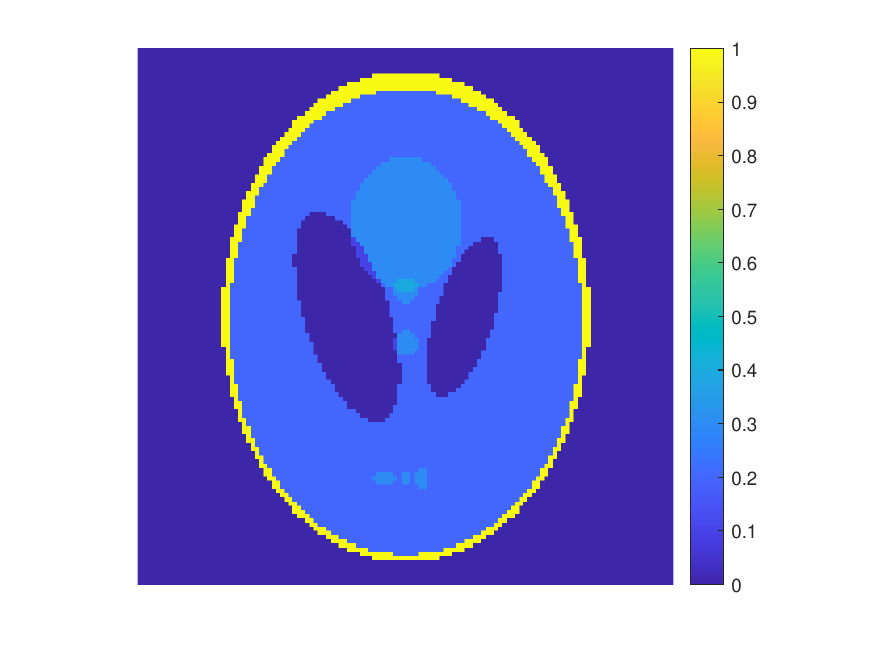}}
\hspace{-1em}
\subfigure[Reconstruction by SGD-$\theta$]{
 \includegraphics[scale=0.5]{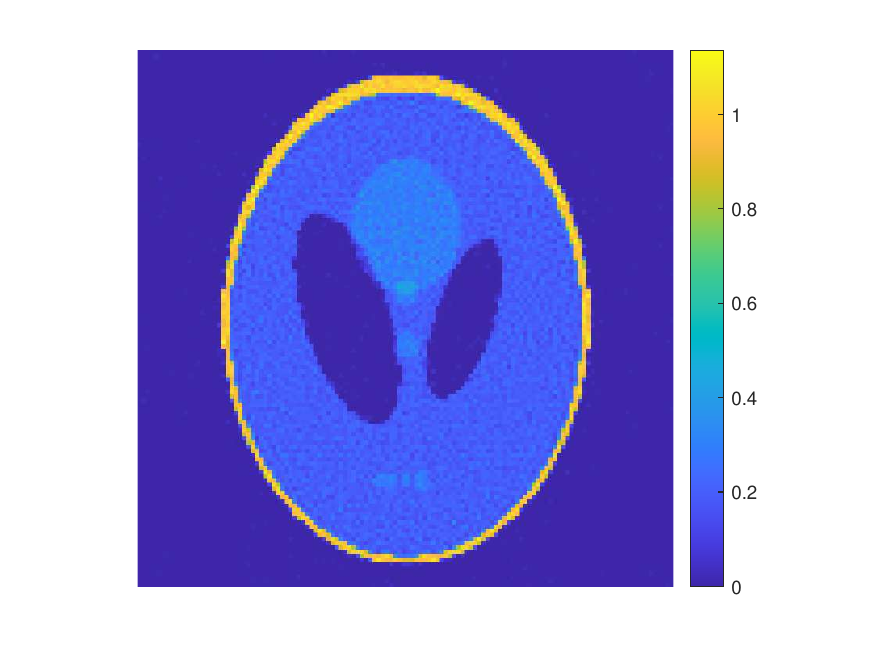}}
\captionsetup{font={small}}
\vspace*{-5pt}
\caption{Results for CT under Gaussian noise with ${\delta _{rel}}=0.01$.}
\label{ctexact}
\end{figure}

 To visualize the quality of reconstructions of SGD-$\theta$ method, in figure \ref{ctexact}(b) and (d) we show the results with relative noise level ${\delta _{rel}}=0.01$. \Cref{ctexact}(b) plots the noisy data. \Cref{ctexact}(d) presents the reconstruction of SGD-$\theta$ using ${\mathcal{Y}_i} = \left( {\mathbb{R},{{\left\|  \cdot  \right\|}_2}} \right)$ after 10000 iterations, from which one can observe that the reconstructed solution agrees well with the true solution.

 \begin{figure}[ht]
\centering
\vspace{-0.35cm} 
\subfigtopskip=2pt 
\subfigbottomskip=6pt 
\subfigcapskip=-5pt 
\subfigure[Salt-and-pepper noisy data ($\kappa=5\%$)]{
\includegraphics[scale=0.29]{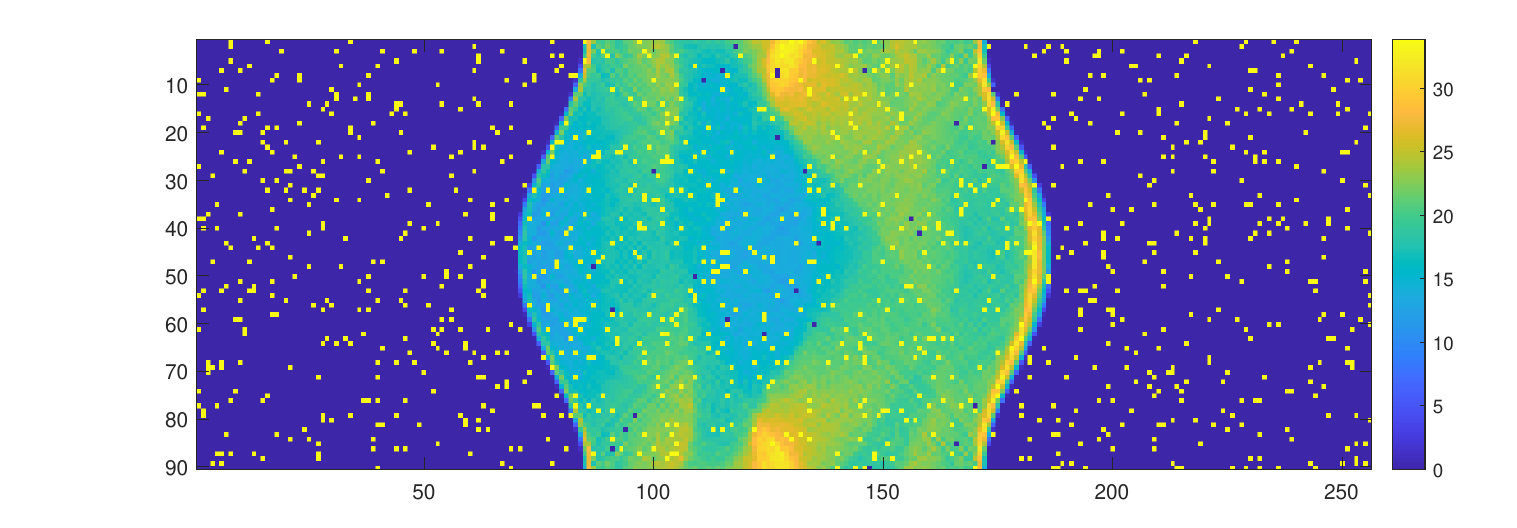}}
\hspace{-1em}
\subfigure[Salt-and-pepper noisy data ($\kappa=10\%$)]{
\includegraphics[scale=0.29]{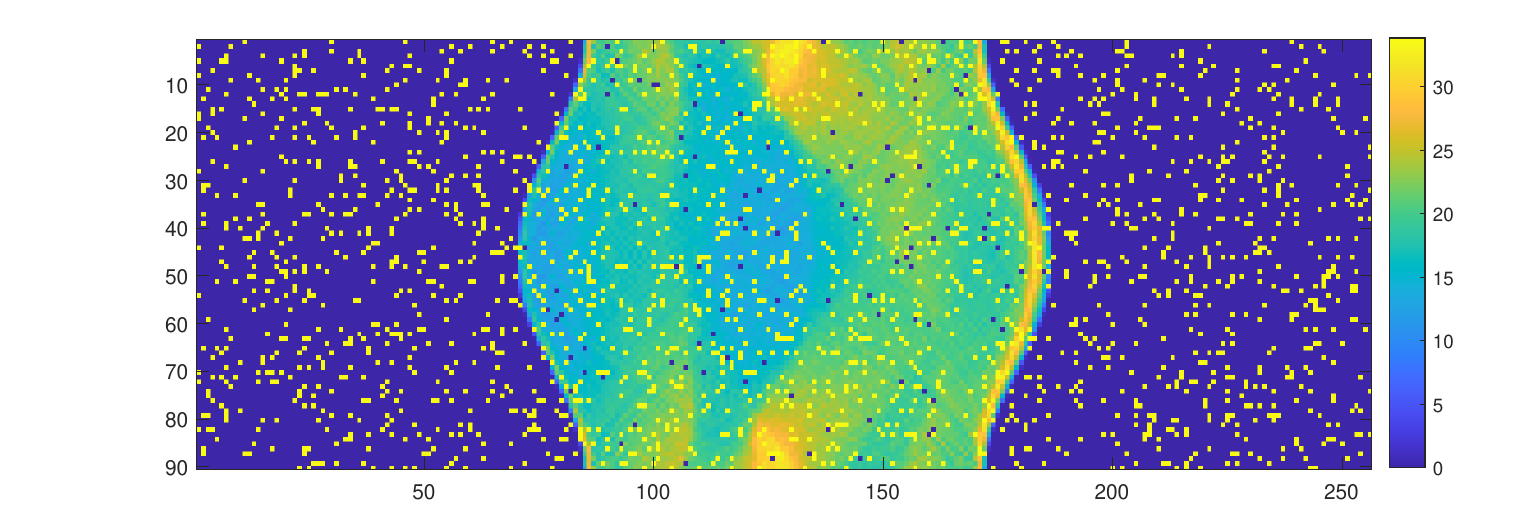}}\\[-2.5mm]
\subfigure[Evolution of relative error ($\kappa=5\%$)]{
\includegraphics[scale=0.5]{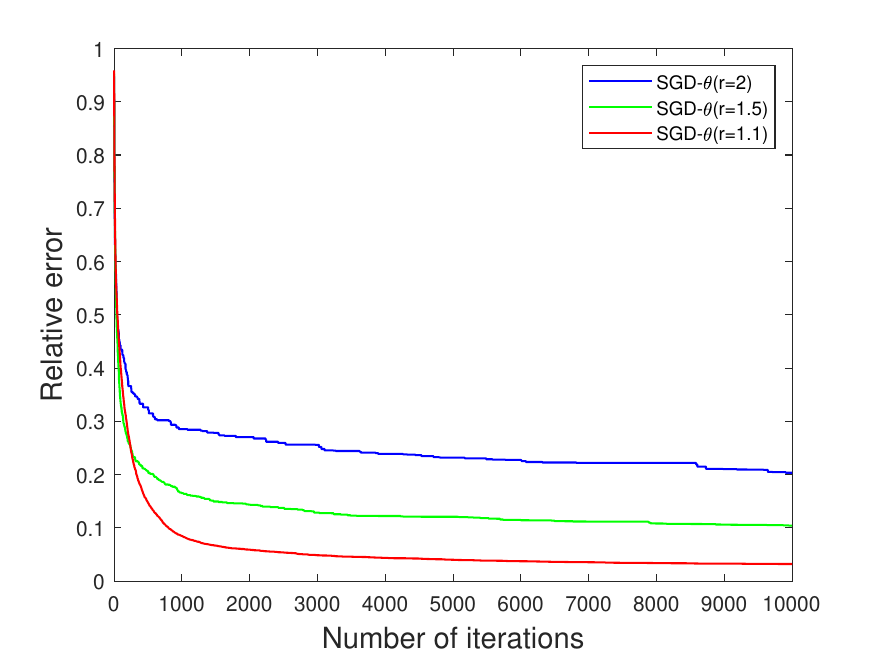}}
\hspace{-1em}
\subfigure[Evolution of relative error ($\kappa=10\%$)]{
\includegraphics[scale=0.5]{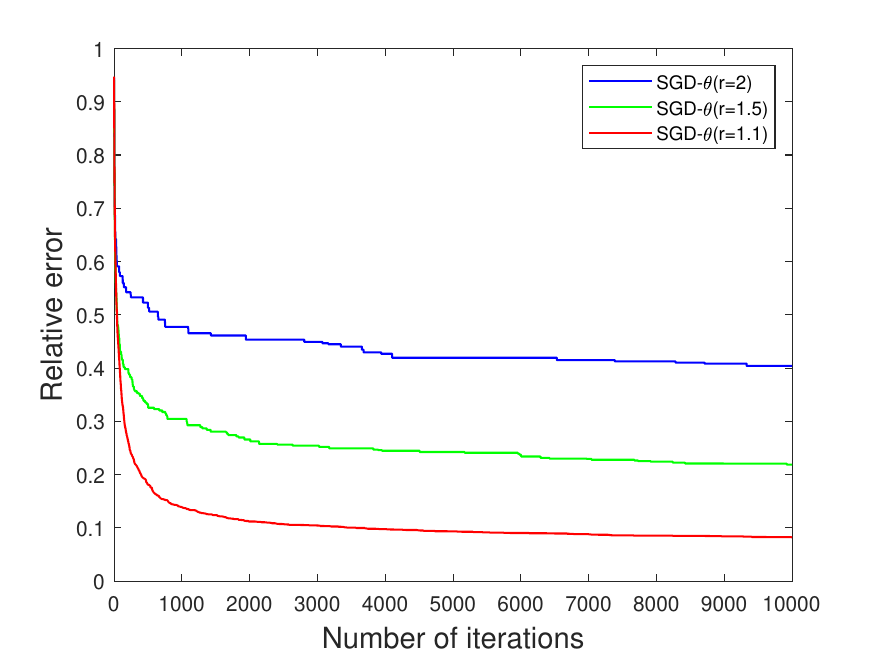}}
\captionsetup{font={small}}
\vspace*{-5pt}
\caption{Computational results for CT under salt-and-pepper noise. Left column: $\kappa=5\%$; Right column: $\kappa=10\%$.}
\label{ctsaltandpepper}
\end{figure}
\begin{figure}[htbp]
\centering
\vspace{-0.35cm} 
\subfigtopskip=2pt 
\subfigbottomskip=6pt 
\subfigcapskip=-5pt 
\subfigure[SGD-$\theta$ with r=1.1]{
\includegraphics[width=5cm,height=4cm]{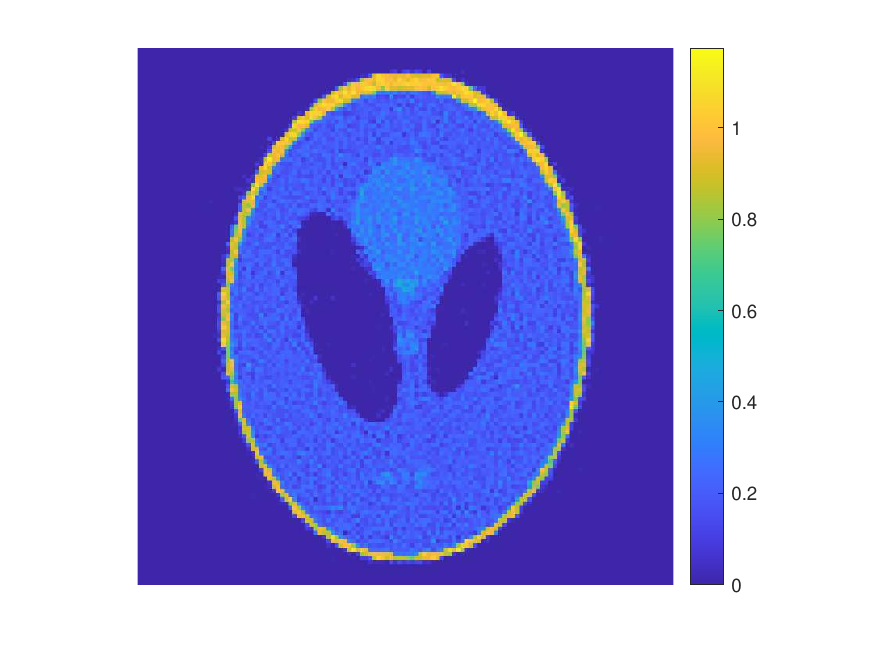}}
\hspace{-1em}
\subfigure[SGD-$\theta$ with r=1.5]{
\includegraphics[width=5cm,height=4cm]{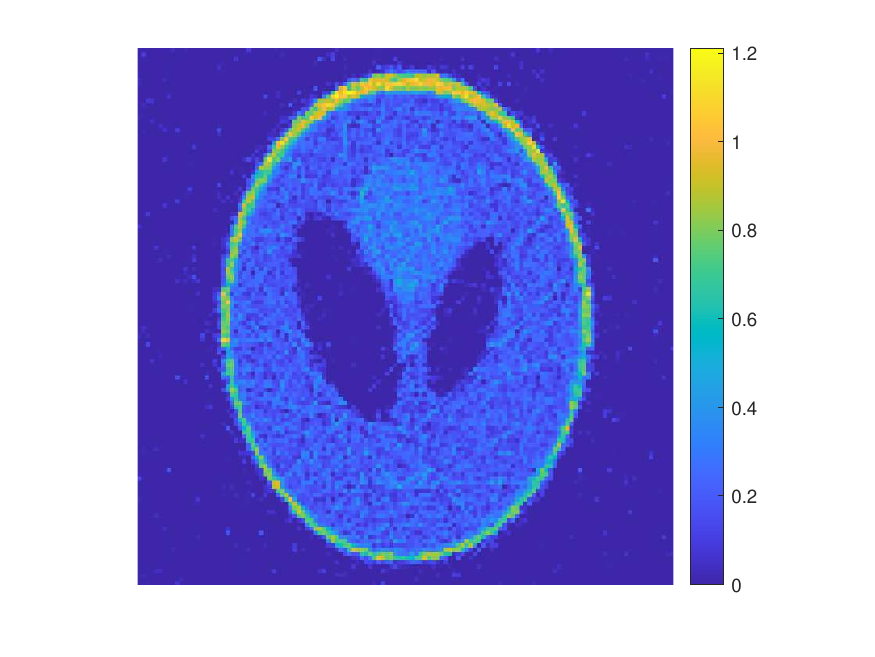}}
\hspace{-1em}
\subfigure[SGD-$\theta$ with r=2]{
\includegraphics[width=5cm,height=4cm]{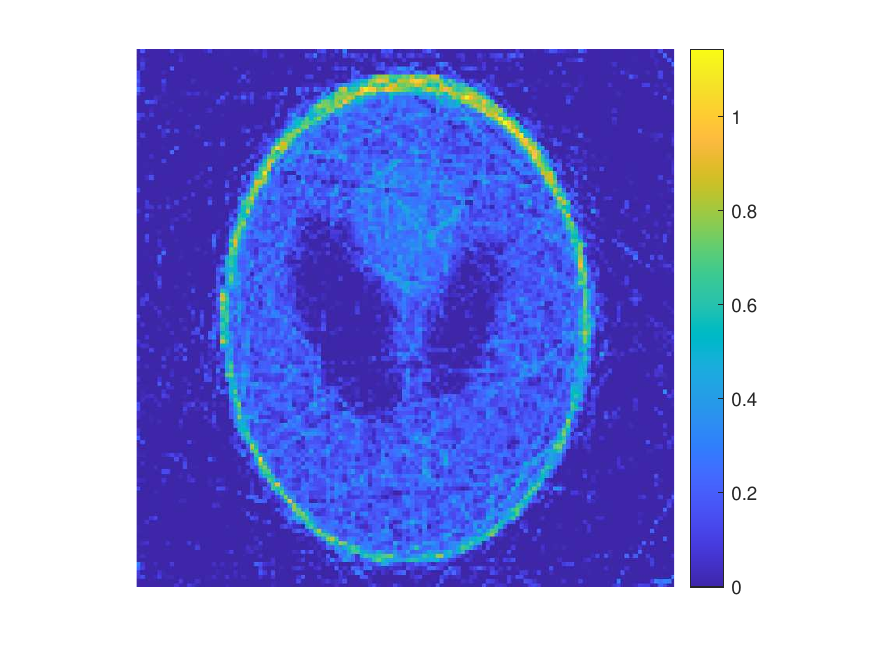}}\\[-2.5mm]
\subfigure[SGD-$\theta$ with r=1.1]{
\includegraphics[width=5cm,height=4cm]{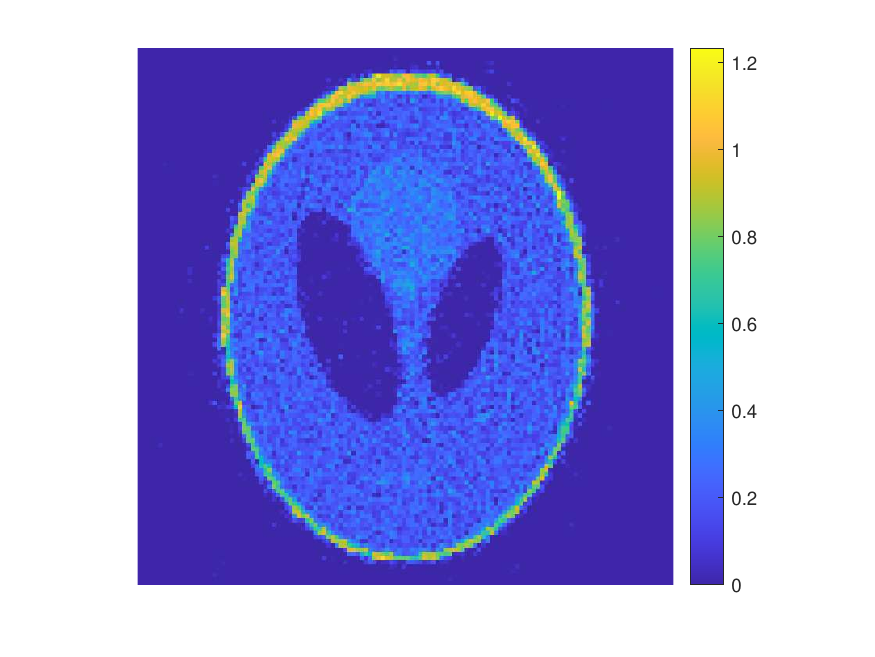}}
\hspace{-1em}
\subfigure[SGD-$\theta$ with r=1.5]{
\includegraphics[width=5cm,height=4cm]{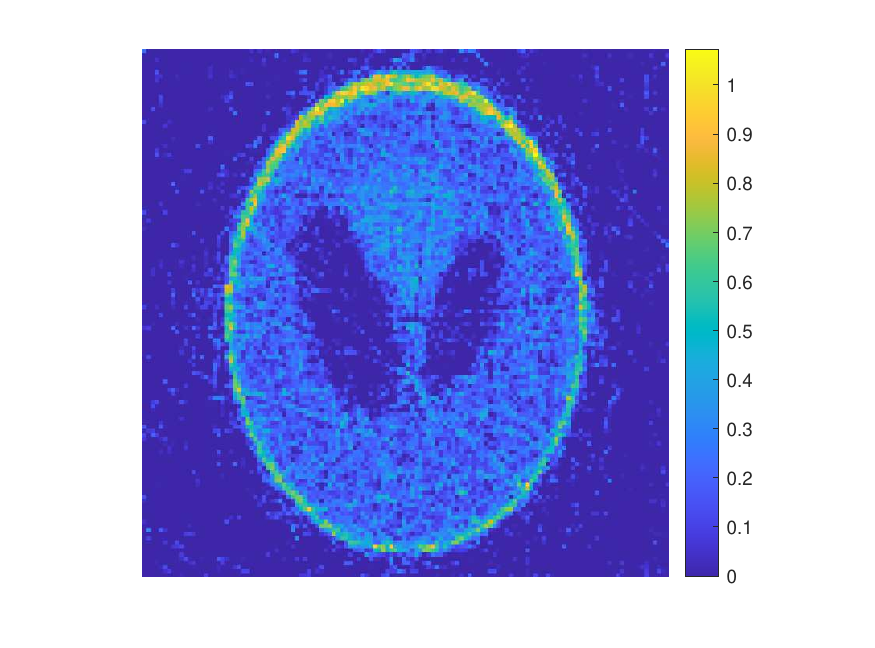}}
\hspace{-1em}
\subfigure[SGD-$\theta$ with r=2]{
\includegraphics[width=5cm,height=4cm]{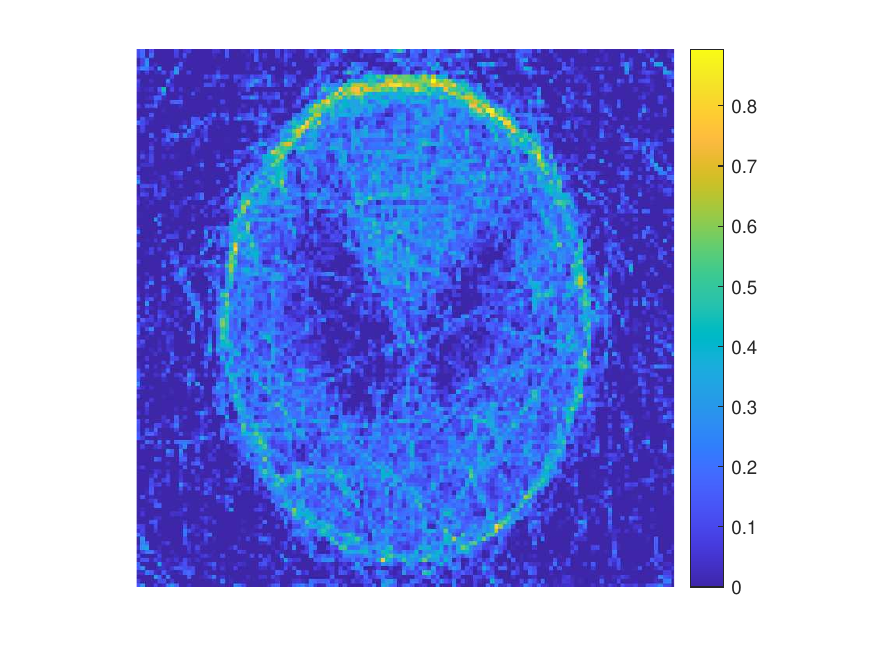}}
\captionsetup{font={small}}
\vspace*{-5pt}
\caption{Reconstruction results for CT by SGD-$\theta$ method with $r=1.1$ (left column), $r=1.5$ (middle column) and $r=2$ (right column) after 10000 iterations. Top row: salt-and-pepper noise with $\kappa=5\%$; Bottom row: salt-and-pepper noise with $\kappa=10\%$.}
\label{ctsaltandpepper_recon}
\end{figure}
Next we consider the situation that the data contains salt-and-pepper noise and investigate the effect of image space $\mathcal{Y}_i=\left( {\mathbb{R},{{\left\|  \cdot  \right\|}_r}} \right)(1<r<\infty)$ by varying the parameter $r$. The contaminated data $y_i^{\delta}$ is formed by
\begin{equation}\label{salt noise}
y_i^\delta  = \left\{ \begin{array}{l}
{y_i},~~~~~\text{with probability}~1 - \kappa, \\[1em]
{y_{\max }},~~\text{with probability} ~\frac{\kappa }{2},\\[1em]
{y_{\min }},~~\text{with probability}~\frac{\kappa }{2},
\end{array} \right.
\end{equation}
where $y_{\max}$ and $y_{\min}$ are the maximum and minimum of the data, respectively, and $\kappa  \in \left( {0,1} \right)$ presents the percentage of the corrupted data points. The noise level $\delta_i$ is defined by ${\delta _i} = {\left\| {y_i^\delta  - {y_i}} \right\|_{{r}}}$, $i=1,2,\ldots,N$. The salt-and-pepper noisy data are displayed in figure \ref{ctsaltandpepper}(a) for $\kappa=5\%$ and figure \ref{ctsaltandpepper}(b) for $\kappa=10\%$, respectively. To examine the impact of image space $\mathcal{Y}_i$, we execute SGD-$\theta$ method using Banach space  $\mathcal{Y}_i=\left( {\mathbb{R},{{\left\|  \cdot  \right\|}_r}} \right)$ with $r=1.1$, $r=1.5$ and $r=2$ for 10000 iterations, respectively. In figure \ref{ctsaltandpepper}(c)-(d) we report the evolution of the relative error ${{{{{\left\| {x_n^\delta  - {x^\dag }} \right\|}_{{2}}}} \mathord{\left/
 {\vphantom {{{{\left\| {x_n^\delta  - {x^\dag }} \right\|}_{{2}}}} {\left\| {{x^\dag }} \right\|}}} \right.
 \kern-\nulldelimiterspace} {\left\| {{x^\dag }} \right\|}}_{{2}}}$ during the first 10000 iterations. It is observed that the relative error of SGD-$\theta$ with $r=1.1$ is the smallest, followed by SGD-$\theta$ with $r=1.5$, then SGD-$\theta$ using $\ell^2$ data fitting term is the largest, after performing the same iterations. To visually compare the quality of the reconstructions, we plot in figure \ref{ctsaltandpepper_recon} the reconstruction results by SGD-$\theta$ method, with $r$ being $1.1, 1.5$ and $2$, after 10000 iterations. One can see that SGD-$\theta$ method using $\ell^{1.1}$ misfit term can captures not only the feature but also the magnitude of the sought image, leading to higher-quality reconstruction results; while the reconstructions obtained by utilizing $r=1.5$ and $r=2$ are less unsatisfactory. It shows that the suitable choice of Banach space $\mathcal{Y}_i$ enables our SGD-$\theta$ method to efficiently eliminate the effect of impulsive noise.
\begin{table}[htbp]
\renewcommand\arraystretch{1.2}
\centering
\caption{ Numerical results of SGD-$\theta$ method with different batch sizes $N_b$.}
\vspace*{-5pt}
\begin{tabular}{cccccc}
\hline
~$N_b$~~~~~~~~~~~~~~~~~~&~1~~~~~~~~~~~~&~10~~~~~~&160~~~&640~~&~2560\\
\hline
Iteration Number~&~81495~~~~~~&12860~&761~~&250~~&206~~\\
Time(s)~~~~~~~~~~~~~& ~~~22275.1121~&~~~~~1085.8117~&~~~~27.4257~&~~~~13.6905~&~~~~13.8024~\\
\hline
\end{tabular}
\label{batch size}
\end{table}

To investigate the impact of the batch size $N_b$, we use the SGD-$\theta$ method (\ref{minibatch_SGD}) with different batch sizes $N_b$ for solving the CT problem in the case that noisy data is generated by (\ref{salt noise}) with $\kappa=5\%$. For comparison, we fix $\mathcal{Y}_i=\left( {\mathbb{R},{{\left\|  \cdot  \right\|}_{1.1}}} \right)$ and terminate SGD-$\theta$ method as long as the relative error ${{{{{\left\| {x_n^\delta  - {x^\dag }} \right\|}_{{2}}}} \mathord{\left/
 {\vphantom {{{{\left\| {x_n^\delta  - {x^\dag }} \right\|}_{{2}}}} {\left\| {{x^\dag }} \right\|}}} \right.
 \kern-\nulldelimiterspace} {\left\| {{x^\dag }} \right\|}}_{{2}}}$ is less than $10^{-1}$. In \cref{batch size}, we summarize the number of iterations and computational times required for SGD-$\theta$ method, under various values of batch size, to reach the same relative error. It can be seen that the larger the batch size $N_b$ used in each iteration, the fewer iterations are required; however, as $N_b$ increases, SGD-$\theta$ method spends more time calculating $\sum\nolimits_{i \in {I_n}} {{F_{{i }}'}{{\left( {x_n^\delta } \right)}^*}J_r^{{\mathcal{Y}_i}}\left( {{F_i}\left( {x_n^\delta } \right) - y_i^\delta } \right)} $ for updating $\xi_{n+1}^\delta$ per iteration. The choice of batch size for optimal performance deserves further research.
 \subsection{Schlieren Imaging}
Next we consider the problem of determining the 3D pressure fields on cross-section of a water tank generated by an ultrasound transducer from Schlieren data \cite{Hanafy1991}. Mathematically, we need to reconstruct a function $f$ supported on a bounded domain $D\subset \mathbb{R}^2$ from
\[\mathcal{I}f\left( {s,\sigma } \right) = {\left( {\int_{\mathbb{R}} {f\left( {s\sigma  + r{\sigma ^ \bot }} \right)} dr} \right)^2},~~\left( {s,\sigma } \right) \in {\mathbb{R}} \times {{\mathbb{S}^1}}.\]
In  our numerical simulations, we pick $D = \left[ { - 1,1} \right] \times \left[ { - 1,1} \right]$ and reconstruct a function $f$ from $N=100$ recording directions ${\sigma _i} \in {\mathbb{S}^1}, i=1,2,\ldots,N$ uniformly distributed on the circle; that is,
${\sigma _i} = \left( {\cos {\varphi _i},\sin {\varphi _i}} \right),$
where ${\varphi _i} \in \left[ {0,2\pi } \right)$, for $i=1,2,\ldots,N$, is the angle.
For each direction ${\sigma _i} $, we define
\[{F_i}\left( f \right)\left( s \right): = \mathcal{I}f\left( {s,{\sigma _i}} \right): = {\left( {{R_i}f\left( s \right)} \right)^2},~~i=1,2,\ldots,N,\]
where ${R_i}f\left( s \right) = \int_{\mathbb{R}} {f\left( {s{\sigma _i} + r\sigma _i^ \bot } \right)} dr$ is the Radon transform along the direction $\sigma_i$. Let $f^\dag$ be the sought solution and let ${y_i} = {F_i}\left( f^\dag \right)$. Then the reconstruction of $f^\dag$ reduces to solve the nonlinear system of the form (\ref{nonlinear equation}), that is,
\begin{equation}\label{nonsystem}
{F_i}\left( f \right) = {y_i},~i=1,2,\ldots,N
\end{equation}
with $N=100$. From \cite{Haltmeier_20071}, it is known that each ${F_i}: \mathcal{X}=H_0^1\left( D \right) \to \mathcal{Y}_i=L^2\left( {\left[ { - \sqrt 2 ,\sqrt 2 } \right]} \right)$, for $i=1,2,\ldots,N$, is continuous and Fr\'{e}chet differentiable with
\[{F_i'}\left( f \right)h = 2{R_i}f \cdot {R_i}h,~~\forall h \in H_0^1\left( D \right).\]
The adjoint ${F_i'}\left( f \right)^*:{L^2}\left( {\left[ { - \sqrt 2 ,\sqrt 2 } \right]} \right) \to H_0^1\left( D \right)$ is given by
\[{F_i'}\left( f \right)^*g = {\left( {I - \Delta } \right)^{ - 1}}\left( {2{R_i^*}\left( {g{R_i}f} \right)} \right),~~\forall g\in {L^2}\left( {\left[ { - \sqrt 2 ,\sqrt 2 } \right]} \right),\]
where $I$ is the identity operator, $\Delta$ is the Laplace operator and $R_i^*:{L^2}\left( {\left[ { - \sqrt 2 ,\sqrt 2 } \right]} \right)\to$ $ {L^2}\left( D \right)$ is the adjoint of $R_i$ given by $\left( {{R_i^*}g} \right)\left( x \right) = g\left( {\left\langle {x,{\sigma _i}} \right\rangle } \right)$.

We next consider the situation that the data contains uniformly distributed noise and test the performance of SGD-$\theta$ method. The noisy data is generated by
\[
y_i^\delta  = {y_i} + {\delta _{rel}}{\left\| {{y_i}} \right\|_{{L^2}}}{\varepsilon _i},~~i=1,2,\ldots,N,\]
\begin{figure}[htbp]
\centering
\vspace{-0.35cm} 
\subfigtopskip=2pt 
\subfigbottomskip=6pt 
\subfigcapskip=-5pt 
\subfigure[$\delta_{rel}=0.05$]{
\includegraphics[width=5cm,height=4cm]{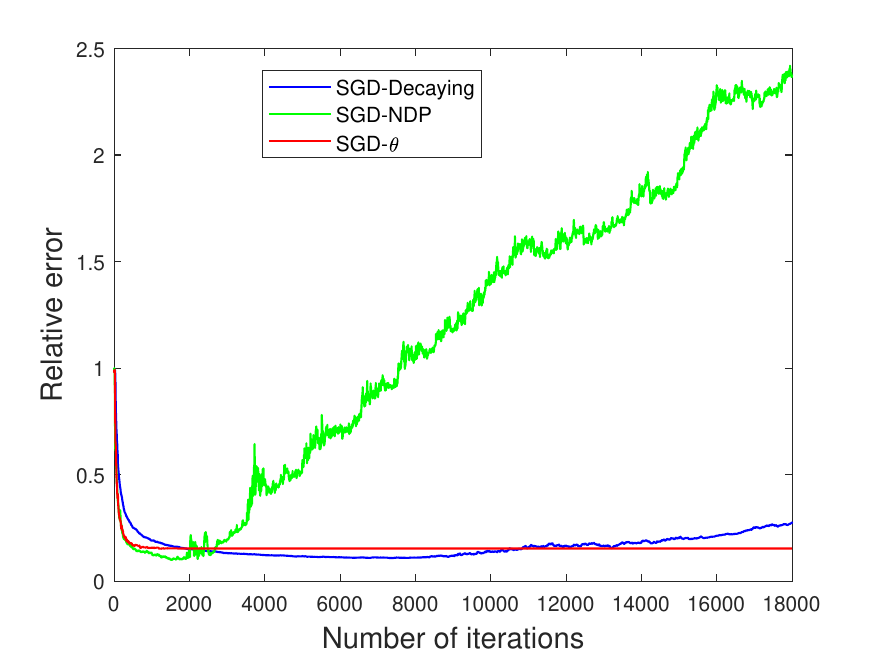}}
\hspace{-1em}
\subfigure[$\delta_{rel}=0.01$]{
\includegraphics[width=5cm,height=4cm]{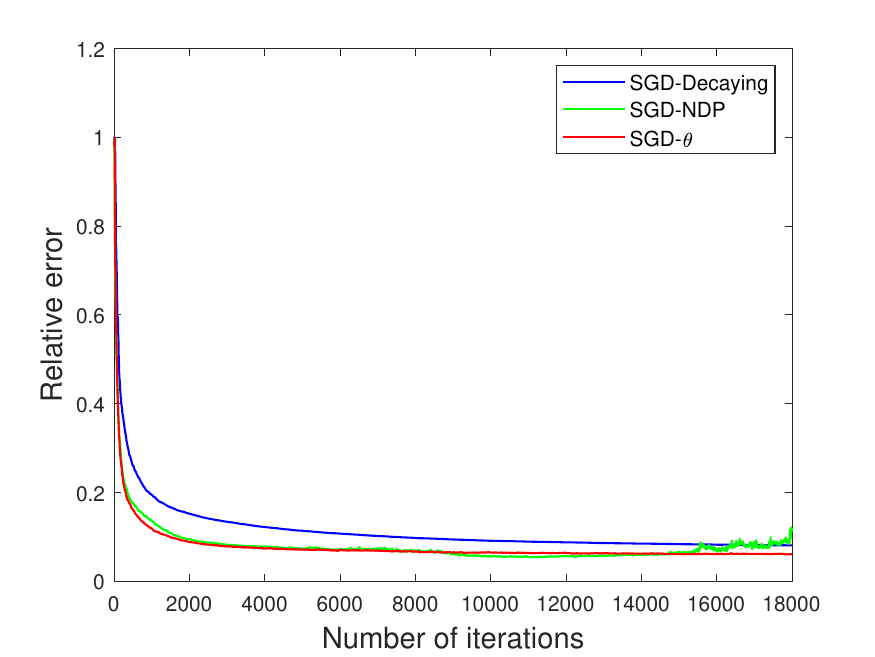}}
\hspace{-1em}
\subfigure[$\delta_{rel}=0.005$]{
 \includegraphics[width=5cm,height=4cm]{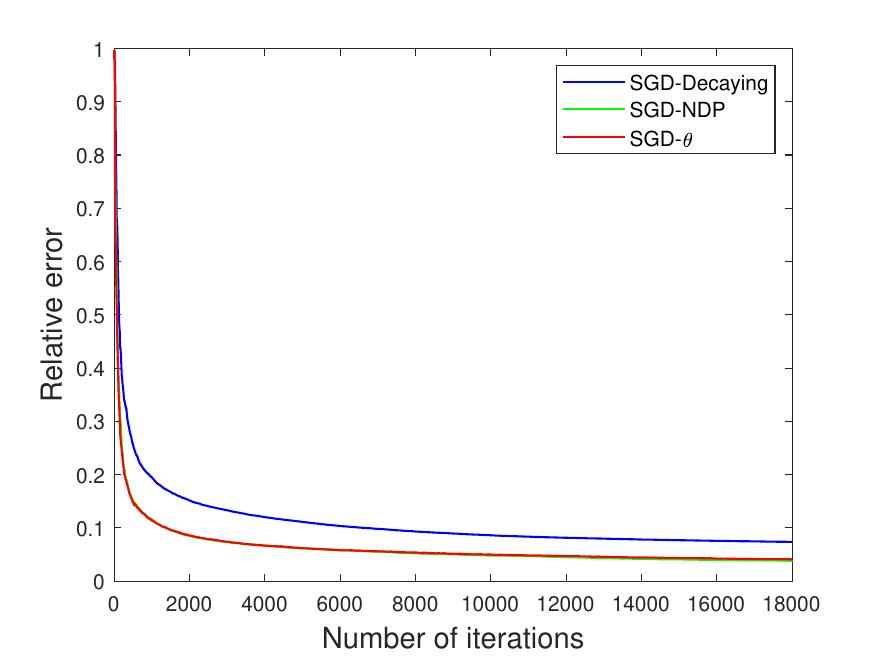}}\\[-2.5mm]
\captionsetup{font={small}}
\vspace*{-5pt}
\caption{Computational results for Schlieren imaging under uniformly distributed noise.}
\label{schgaussian}
\end{figure}
where ${\varepsilon _i}$s are uniform noise distributed on $[-1,1]$ and the noise level is ${\delta _i} = {\delta _{rel}}{\left\| {{y_i}} \right\|_{{L^2}}},i=1,2,\ldots,N$. Assume that the sought solution $f^\dag$ is piecewise constant, whose graph together with the exact data are shown in figure \ref{schre}.  In order to carry out the computation, we take $120\times 120$ grid points uniformly distributed over $D$ and pick $\xi_0=0.01$ and $f_{0}  = \arg \mathop {\min }\nolimits_{x \in \mathcal{X}} \left\{ {\theta \left( x \right) - \left\langle {\xi _0 ,x} \right\rangle } \right\}$ as the initial guess. When implementing SGD-$\theta$ method, we use (\ref{thetatv}) with $\beta=5$ to capture the features of $f^\dag$ and choose $\mu_0={{0.01} \mathord{\left/
 {\vphantom {{0.01} \beta }} \right.
 \kern-\nulldelimiterspace} \beta }$, $\mu_1=10^6$ and $\tau=1.02$; we take $t_0=10^6$ for performing SGD-Decaying method. In figure \ref{schgaussian} we report the evolution of the reconstruction error ${{{{{\left\| {f_n^\delta  - {f^\dag }} \right\|}_{{L^2}}}} \mathord{\left/
 {\vphantom {{{{\left\| {x_n^\delta  - {x^\dag }} \right\|}_{{L^2}}}} {\left\| {{x^\dag }} \right\|}}} \right.
 \kern-\nulldelimiterspace} {\left\| {{f^\dag }} \right\|}}_{{L^2}}}$ of SGD-$\theta$ method during the first 18000 iterations under three different relative noise levels $\delta_{rel}=0.05, 0.01, 0.005$; for comparison, we also present the results obtained by using SGD-Decaying and SGD-NDP methods. One can observe that SGD-Decaying and SGD-NDP methods admit semi-convergence property and the iterates exhibit dramatic oscillations, particularly in the case of relatively large noise level; while such semi-convergence property and oscillations could be dramatically reduced by utilizing our SGD-$\theta$ method, which obviously shows the superiority of the step size chosen by (\ref{3choicestep}).
 To visualize the performance of SGD-$\theta$ method, we depict in figure \ref{schre}(b) the noisy data with relative noise level $\delta_{rel}=0.005$ and (d) the reconstruction result of SGD-$\theta$ after 18000 iterations. The results in figure \ref{schre} clearly demonstrate the effectiveness of our SGD-$\theta$ method in handling the uniformly distributed noise and the capacity of SGD-$\theta$ method for capturing the piecewise constant feature of the sought solution. We note that SGD-$\theta$ method performs well for solving (\ref{nonsystem}), even if the condition (\ref{TCC}) in Assumption \ref{5asumption operator} cannot be verified.
\begin{figure}[htbp]
\centering
\vspace{-0.35cm} 
\subfigtopskip=2pt 
\subfigbottomskip=6pt 
\subfigcapskip=-5pt 
\subfigure[Exact data]{
\includegraphics[scale=0.5]{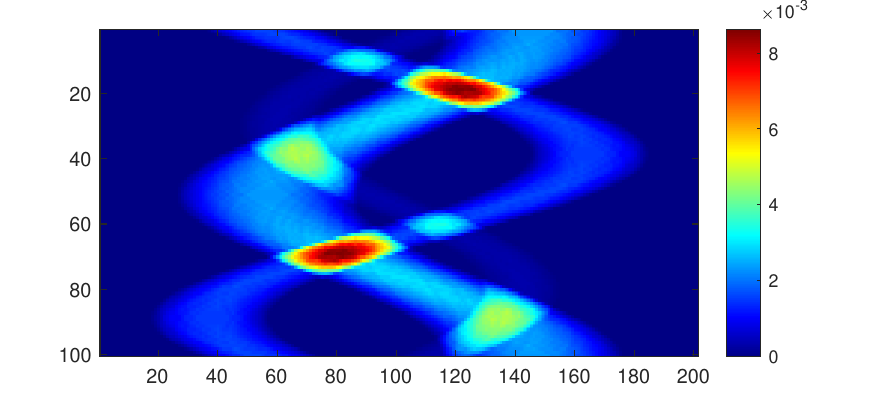}}
\hspace{-1em}
\subfigure[Noisy data ($\delta_{rel}=0.005$)]{
\includegraphics[scale=0.5]{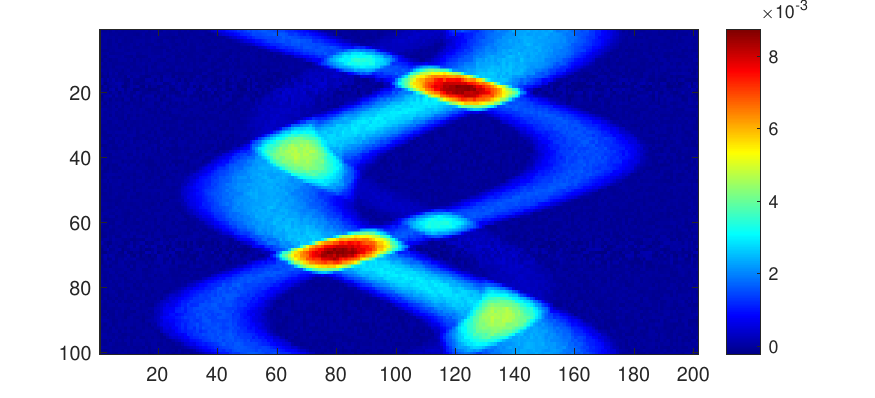}}\\[-2.5mm]
\captionsetup{font={small}}
\subfigure[Exact solution]{
 \includegraphics[scale=0.5]{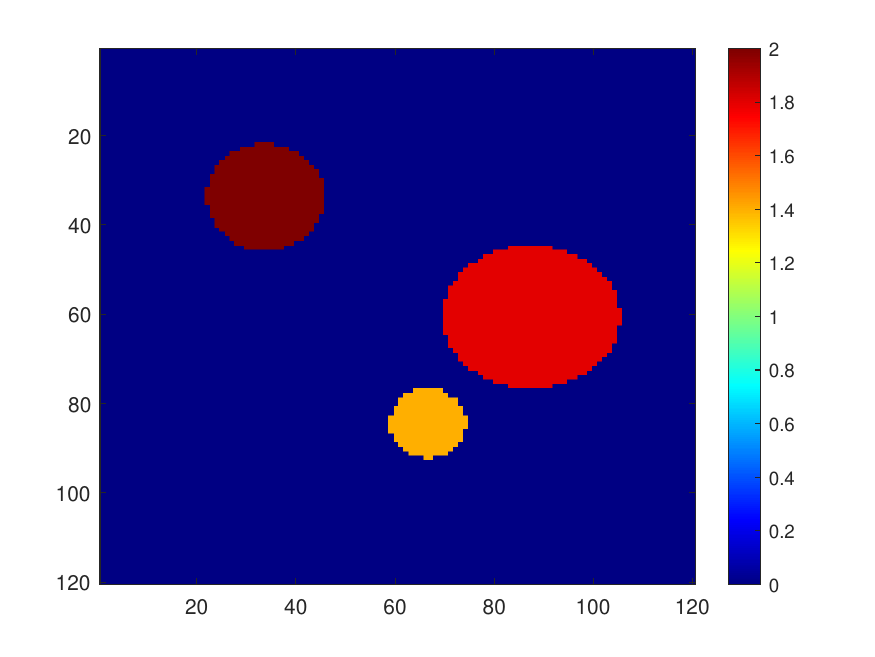}}
 \hspace{-1em}
\subfigure[Reconstruction by SGD-$\theta$]{
\includegraphics[scale=0.5]{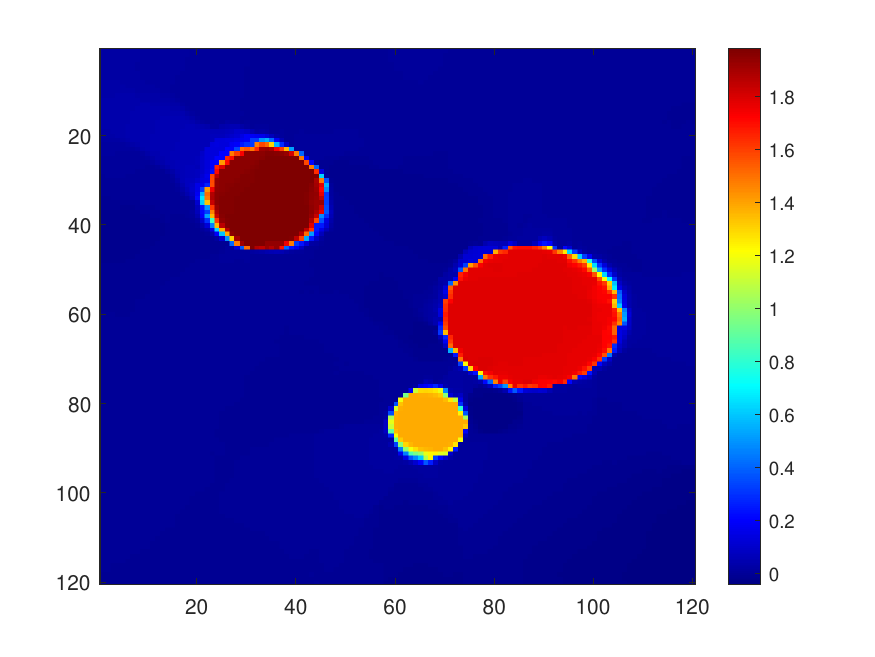}}
\vspace*{-5pt}
\caption{Results for Schlieren imaging under uniformly distributed noise with $\delta_{rel}=0.005$.}
\label{schre}
\end{figure}

In summary, from the above numerical results, it can be concluded that our SGD-$\theta$ method could efficiently eliminate different types of noise by appropriately selecting the Banach space $\mathcal{Y}_i$, and turned out to be efficient for solving linear as well as nonlinear systems, which generalizes the works \cite{JinQ_2023,JinBker2023} to cover nonlinear inverse problems and to deal with the case that the data is contaminated by impulsive noise. Furthermore, SGD-$\theta$ method outperforms SGD-Decaying and SGD-NDP methods, which validates the effectiveness of the step size (\ref{3choicestep}).

\section{Conclusion and outlook}\label{conclusion}
In this paper, we have investigated and analyzed the stochastic gradient descent method with convex penalty for solving linear and nonlinear inverse problems in Banach spaces. Under the traditional assumptions on the operators $F_i$ and suitable choice of the step size, by using techniques from regularization theory in Banach spaces and stochastic analysis, we have established the regularization property of the method equipped with an {\it a priori} stopping rule. The numerical simulations on computed tomography and schlieren imaging clearly demonstrate the superiority of step size schedule and the effectiveness of the method in dealing with the data containing various types of noise. Moreover, the method turned out to be feasible for solving large-scale systems of linear and nonlinear ill-posed equations.

There are several possible lines for future research.  First, it is of much interest to analyze the convergence rates of SGD-$\theta$ method by employing variational source conditions or conditional stability estimates \cite{Cheng-YAMAMOTO-2000,Fu-Jin-2020}. Second, the assumptions, especially the tangential cone condition (\ref{TCC}) for the forward operators $F_i$, should be verified for specific nonlinear inverse problems; and the Fr\'{e}chet differentiability of the forward operators in Assumption \ref{5asumption operator} may be relaxed \cite{2019-Clason-p789-832}.

\ack{R Gu is partially supported by the National Natural Science Foundation of China (No. 12301535) and and China Postdoctoral Science Foundation (No. 2024M750292). Z Fu is partially supported by the National Natural Science Foundation of China (No. 12101159) and China Postdoctoral Science Foundation (No. 2022M710969). B Han is partially supported by the National Natural Science Foundation of China (No. 12271129). H Fu is partially supported by the National Natural Science Foundation
of China (No. 42274166).}
\appendix

\section{SGD method under a posteriori stopping rule}\label{appendix_sgd}
\renewcommand{\thesection}{A}
In section \ref{3themethod}, we have shown the convergence results of SGD-$\theta$ method under an {\it a priori} stopping rule. The issue of the {\it a posteriori} stopping rule for stochastic iterative methods is desirable and challenging, even for Hilbert spaces \cite{Jahn-Jin-2020,Jin-Chen-2024,Jin-Liu-2024}. Recently, an {\it a posteriori} stopping rule was considered in \cite{Jin-Chen-2024} for stochastic variance reduced gradient method and in \cite{Jin-Liu-2024} for randomized block coordinate descent method for linear inverse problems in Hilbert spaces. Motivated by the work in \cite{Jin-Chen-2024,Jin-Liu-2024}, in this appendix we design an {\it a posteriori} stopping rule for SGD-$\theta$ method and show its finite-iteration termination property. We consider the problems (\ref{nonlinear equation}), where $F_i:\mathscr{D}(F_i)\subset \mathcal{X} \rightarrow {\mathcal{Y}_i}$ are linear or nonlinear operators from Banach spaces $\mathcal{X}$ to Hilbert spaces ${\mathcal{Y}_i}$ with domain $\mathscr{D}(F_i)$. Suppose that the function $\theta :\mathcal{X} \to \left( { - \infty ,\infty } \right]$ is  proper, lower semi-continuous, 2-convex satisfying (\ref{theta p-convex}) for some constant $\sigma>0$. The corresponding SGD-$\theta$ method can be formulated as
\begin{equation}\label{sgd_posteriori}
\begin{array}{l}
\xi _{n + 1}^\delta  = \xi _n^\delta  - t_n^\delta {F_{{i_n}}'}\left( {x_n^\delta } \right)^*\left( {{F_{{i_n}}}\left( {x_n^\delta } \right) - y_{{i_n}}^\delta } \right),\\[1mm]
x_{n + 1}^\delta  = \arg \mathop {\min }\limits_{x \in \mathcal{X}} \left\{ {\theta \left( x \right) - \left\langle {\xi _{n + 1}^\delta ,x} \right\rangle } \right\},
\end{array}
\end{equation}
where ${i_n} \in \left\{ {1,2, \cdots ,N} \right\}$ is selected uniformly at random. The step size $t_{{n}}^\delta$ is taken as the constant step size, given by $t_n^\delta=\bar{t} \gamma_n$, where
\begin{equation}\label{step_posteriori}
0< \bar{t} \le\min \left\{ {\frac{{{\mu _0}{{\left\| {{F_{{i_n}}}\left( {x_n^\delta } \right) - y_{{i_n}}^\delta } \right\|}^2}}}{{{{\left\| {{F_{{i_n}}'}{{\left( {x_n^\delta } \right)}^*}\left( {{F_{{i_n}}}\left( {x_n^\delta } \right) - y_{{i_n}}^\delta } \right)} \right\|}^2}}},{\mu _1}} \right\}
\end{equation}
and
\[{\gamma _n} = \left\{ \begin{array}{l}
1,~~~~\sum\limits_{i = 1}^N {{{\left\| {{F_i}\left( {x_n^\delta } \right) - y_i^\delta } \right\|}^2}}  > \sum\limits_{i = 1}^N {{{\left( {\tau {\delta _i}} \right)}^2}} \\[1mm]
0,~~~~otherwise
\end{array} \right.\]
for some constant $\tau>1$.

Note that the SGD-$\theta$ method (\ref{sgd_posteriori})-(\ref{step_posteriori}) defines an infinite sequence
$\left\{ {\left( \xi _n^\delta ,x_n^\delta \right)} \right\}$ by incorporating the parameter $\gamma_n$, which is for the convenience of analysis. In the numerical simulations, the iteration  (\ref{sgd_posteriori}) can be stopped as long as
\begin{equation}\label{stopping rule}
\sum\limits_{i = 1}^N {{{\left\| {{F_i}\left( {x_{{n_\delta }}^\delta } \right) - y_i^\delta } \right\|}^2}}  \le \sum\limits_{i = 1}^N {{{\left( {\tau {\delta _i}} \right)}^2}}  < \sum\limits_{i = 1}^N {{{\left\| {{F_i}\left( {x_n^\delta } \right) - y_i^\delta } \right\|}^2}},~~0\le n<n_\delta,
\end{equation}
where the stopping index $n_\delta$ depends crucially on the sample path and thus is a random integer.
Next we show that the stopping rule (\ref{stopping rule}) terminates the SGD-$\theta$ method (\ref{sgd_posteriori}) after finite steps almost surely.
\begin{proposition}
If $\mu_0>0$ and $\tau>1$ are chosen such that
\[c_3:={1 - \eta  - \frac{{{\mu _0}}}{{4\sigma }} - \frac{{1 + \eta }}{2} - \frac{{1 + \eta }}{{2{\tau ^2}}}}>0,\]
then for any solution $\hat x$ of (\ref{nonlinear equation}) in ${B_{2\rho} }\left( {{x_0}} \right)\cap \mathscr{D}\left( \theta  \right)$, there holds
\begin{equation}\label{exmon_posteriori}
\mathbb{E}\left[ {{D_{\xi _{n + 1}^\delta }}\theta \left( {\hat x,x_{n + 1}^\delta } \right)} \right] -\mathbb{E}\left[ {{D_{\xi _n^\delta }}\theta \left( {\hat x,x_n^\delta } \right)} \right]\le - \frac{{{c_3}}}{N}\mathbb{E}\left[ {t_n^\delta \sum\limits_{i = 1}^n{{{\left\| {{F_i}\left( {x_n^\delta } \right) - y_i^\delta } \right\|}^2}} } \right]
\end{equation}
for $n\ge 0$. Moreover, SGD-$\theta$ method stops in finite steps almost surely.
\end{proposition}
\begin{proof}
The argument below is inspired by  \cite[Proposition 5.1]{Jin-Chen-2024} and \cite[Theorem 3.5]{Jin-Liu-2024}. By utilizing the argument analogous to the proof of Lemma \ref{smonotoniciy}, we can deduce that
\[\begin{array}{*{20}{l}}
{\quad {D_{\xi _{n + 1}^\delta }}\theta \left( {\hat x,x_{n + 1}^\delta } \right) - {D_{\xi _n^\delta }}\theta \left( {\hat x,x_n^\delta } \right)}\\[1mm]
{ \le  - \left( {1 - \eta  - \frac{{{\mu _0}}}{{4\sigma }}} \right)t_n^\delta {{\left\| {{F_{{i_n}}}\left( {x_n^\delta } \right) - y_{{i_n}}^\delta } \right\|}^2} + \left( {1 + \eta } \right)t_n^\delta \left\| {{F_{{i_n}}}\left( {x_n^\delta } \right) - y_{{i_n}}^\delta } \right\|{\delta _{{i_n}}}.}
\end{array}\]
By using Young's inequality ($ab \le \frac{{{a^2}}}{2} + \frac{{{b^2}}}{2}$), we can derive that
\[\begin{array}{*{20}{l}}
{\quad {D_{\xi _{n + 1}^\delta }}\theta \left( {\hat x,x_{n + 1}^\delta } \right) - {D_{\xi _n^\delta }}\theta \left( {\hat x,x_n^\delta } \right)}\\[1mm]
{ \le  - \left( {1 - \eta  - \frac{{{\mu _0}}}{{4\sigma }}} \right)t_n^\delta {{\left\| {{F_{{i_n}}}\left( {x_n^\delta } \right) - y_{{i_n}}^\delta } \right\|}^2} + \frac{{1 + \eta }}{2}t_n^\delta {{\left\| {{F_{{i_n}}}\left( {x_n^\delta } \right) - y_{{i_n}}^\delta } \right\|}^2} + \frac{{1 + \eta }}{2}t_n^\delta {{\left( {{\delta _{{i_n}}}} \right)}^2}.}
\end{array}\]
Taking the expectation conditioned on $\mathcal{F}_n$, there  follows
\begin{equation}\label{post_555}
\begin{array}{*{20}{l}}
{\quad \mathbb{E}\left[ {{D_{\xi _{n + 1}^\delta }}\theta \left( {\hat x,x_{n + 1}^\delta } \right) - {D_{\xi _n^\delta }}\theta \left( {\hat x,x_n^\delta } \right)|\mathcal{F}_n} \right]}\\[1mm]
{ \le  - \left( {1 - \eta  - \frac{{{\mu _0}}}{{4\sigma }}-\frac{{1 + \eta }}{2}} \right)t_n^\delta \frac{1}{N}\sum\limits_{i = 1}^N {{{\left\| {{F_i}\left( {x_n^\delta } \right) - y_i^\delta } \right\|}^2}}  + \frac{{1 + \eta }}{2}t_n^\delta \frac{1}{N}\sum\limits_{i = 1}^N {{{\left( {{\delta _i}} \right)}^2}} .}
\end{array}
\end{equation}
By the definition of $t_n^\delta$, there holds
\[t_n^\delta \sum\limits_{i = 1}^N {{{\left( {{\delta _i}} \right)}^2}}  < t_n^\delta \frac{1}{{{\tau ^2}}}\sum\limits_{i = 1}^N {{{\left\| {{F_i}\left( {x_n^\delta } \right) - y_i^\delta } \right\|}^2}}. \]
Inserting the above inequality into (\ref{post_555}), and by taking the full expectation, we  derive that
\[\begin{array}{*{20}{l}}
{\quad \mathbb{E}\left[ {{D_{\xi _{n + 1}^\delta }}\theta \left( {\hat x,x_{n + 1}^\delta } \right)} \right] - \mathbb{E}\left[ {{D_{\xi _n^\delta }}\theta \left( {\hat x,x_n^\delta } \right)} \right]}\\[1mm]
{ \le  - \left( {1 - \eta  - \frac{{{\mu _0}}}{{4\sigma }} - \frac{{1 + \eta }}{2}-\frac{{1 + \eta }}{{2{\tau ^2}}}} \right)\frac{1}{N}\mathbb{E}\left[ {t_n^\delta \sum\limits_{i = 1}^n{{{\left\| {{F_i}\left( {x_n^\delta } \right) - y_i^\delta } \right\|}^2}} } \right]},
\end{array}\]
from which  (\ref{exmon_posteriori}) follows.

We next consider the event
\[\Psi : = \left\{ {\sum\limits_{i = 1}^N {{{\left\| {{F_i}\left( {x_n^\delta } \right) - y_i^\delta } \right\|}^2}}  > \sum\limits_{i = 1}^N {{{\left( {\tau {\delta _i}} \right)}^2}}~\text{for~all~integer}~ n \ge 0} \right\}.\]
It is enough to show $\mathbb{P}\left( \Psi  \right) = 0$. By virtue of (\ref{exmon_posteriori}), we have with $\hat x=x^\dag$ that
\[{\frac{{{c_3}}}{N}\mathbb{E}\left[ {t_n^\delta \sum\limits_{i = 1}^N {{{\left\| {{F_i}\left( {x_n^\delta } \right) - y_i^\delta } \right\|}^2}} } \right] \le \mathbb{E}\left[ {{D_{\xi _n^\delta }}\theta \left( {{x^\dag },x_n^\delta } \right)} \right] - \mathbb{E}\left[ {{D_{\xi _{n + 1}^\delta }}\theta \left( {{x^\dag },x_{n + 1}^\delta } \right)} \right]}\]
and thus for any integer $l\ge 0$ that
\begin{equation}\label{summary}
\frac{{{c_3}}}{N}\sum\limits_{n = 0}^l {\mathbb{E}\left[ {t_n^\delta \sum\limits_{i = 1}^N {{{\left\| {{F_i}\left( {x_n^\delta } \right) - y_i^\delta } \right\|}^2}} } \right]}  \le {D_{{\xi _0}}}\theta \left( {{x^\dag },{x_0}} \right) < \infty.
 \end{equation}
Let ${\chi _\Psi }$ denote the characteristic function of $\Psi$, that is ${\chi _\Psi }\left( w \right)=1$ if $w \in \Psi$, and 0 otherwise, by the definition of $t_n^\delta$, we then have
\[\begin{array}{l}
\mathbb{E}\left[ {t_n^\delta \sum\limits_{i = 1}^N {{{\left\| {{F_i}\left( {x_n^\delta } \right) - y_i^\delta } \right\|}^2}} } \right] \ge \mathbb{E}\left[ {t_n^\delta \sum\limits_{i = 1}^N {{{\left\| {{F_i}\left( {x_n^\delta } \right) - y_i^\delta } \right\|}^2}{\chi _\Psi }} } \right] \\[1mm]
\ge \bar{t}\sum\limits_{i = 1}^N {{{\left( {\tau {\delta _i}} \right)}^2}} \mathbb{E}\left[ {{\chi _\Psi }} \right]
 = \bar{t}\sum\limits_{i = 1}^N {{{\left( {\tau {\delta _i}} \right)}^2}\mathbb{P}} \left( \Psi  \right).
\end{array}\]
The combination of the above inequality with (\ref{summary}) yields
\[\bar{t}\sum\limits_{i = 1}^N {{{\left( {\tau {\delta _i}} \right)}^2}} \frac{{{c_3}}}{N}\left( {l+1} \right)\mathbb{P}\left( \Psi  \right) < {D_{{\xi _0}}}\theta \left( {{x^\dag },{x_0}} \right) < \infty \]
for all $l \ge 0$ and thus $\mathbb{P}\left( \Psi  \right)\to 0$ as $l \to  \infty$. Therefore, $\mathbb{P}\left( \Psi  \right) = 0$.
\end{proof}
\begin{remark}
The above proposition demonstrates that the stopping rule (\ref{stopping rule}) terminates the SGD-$\theta$ method (\ref{sgd_posteriori}) with constant step size in finite iteration steps almost surely. We note that the computation of $\sum\nolimits_{i = 1}^N {{{\left\| {{F_i}\left( {x_n^\delta } \right) - y_i^\delta } \right\|}^2}}$ at each iteration is expensive, the design of the computationally efficient stopping rule and the convergence analysis of the iterates under the posteriori stopping rule will be the future work.
\end{remark}
\section*{References}

\bibliographystyle{unsrt}

\end{document}